\documentclass[12pt,twoside]{amsart}
\usepackage{mathrsfs}
\usepackage{mathtools}
\usepackage{amssymb}
\usepackage{verbatim}
\usepackage{t1enc}
\usepackage{amsmath}
\usepackage{comment}
\usepackage{amsthm,thmtools,xcolor}
\usepackage[colorlinks,linkcolor=blue,citecolor=blue, hypertexnames=false,pdfstartview=FitH]
{hyperref}
\usepackage{backref}
\usepackage{bm}
\usepackage{a4wide}
\usepackage{babel}
\usepackage[latin1]{inputenc}
\usepackage[T1]{fontenc}
\usepackage{times}
\usepackage{hyperref}
\usepackage{amssymb,latexsym}
\usepackage{enumerate}

\newcommand{\dbar}{\ensuremath{\overline\partial}}

\usepackage{etoolbox}
\makeatletter
\patchcmd\maketitle
{\uppercasenonmath\shorttitle}
{}
{}{}
\patchcmd\maketitle
{\@nx\MakeUppercase{\the\toks@}}
{\the\toks@}
{}
{}{}
\patchcmd\@settitle
{\uppercasenonmath\@title}
{}
{}{}
\patchcmd\@setauthors
{\MakeUppercase{\authors}}
{\authors}
{}{}

\usepackage{enumerate}

\makeatletter
\newcommand{\sumprime}{\if@display\sideset{}{'}\sum%
	\else\sum'\fi}
\makeatother

\begin{document}

	\begin{sloppypar}	
		
		\numberwithin{equation}{section}
		
		\newtheorem{theorem}{Theorem}[section]
		\newtheorem{proposition}[theorem]{Proposition}
		\newtheorem{conjecture}[theorem]{Conjecture}
		\def\theconjecture{\unskip}
		\newtheorem{corollary}[theorem]{Corollary}
		\newtheorem{lemma}[theorem]{Lemma}
		\newtheorem{observation}[theorem]{Observation}
		\newtheorem{definition}[theorem]{Definition}
		
		\newtheorem{remark}{Remark}
		\def\theremark{\unskip}
		\newtheorem{kl}{Key Lemma}
		\def\thekl{\unskip}
		\newtheorem{question}{Question}
		\def\thequestion{\unskip}
		\newtheorem{example}{Example}
		\def\theexample{\unskip}
		\newtheorem{problem}{Problem}
		
		\bigskip
		
		\address{School of Mathematics and Information Science, Guangzhou University,
			No. 230 Waihuanxi Road, Guangzhou Higher Education Mega Center,
			Panyu District, Guangzhou, Guangdong Province,
			P.R. China, 510006.}
		\email{yanhe@gzhu.edu.cn}
		
		\address{DEPARTMENT OF MATHEMATICAL SCIENCES, NORWEGIAN UNIVERSITY OF SCIENCE AND TECHNOLOGY, NO-7491 TRONDHEIM, NORWAY}
		
		\email{johannes.testorf@ntnu.no}

		\address{DEPARTMENT OF MATHEMATICAL SCIENCES, NORWEGIAN UNIVERSITY OF SCIENCE AND TECHNOLOGY, NO-7491 TRONDHEIM, NORWAY} 
		
		\email{xu.wang@ntnu.no}
		
		\title{Ross-Witt Nystr\"om correspondence and Ohsawa-Takegoshi extension}
		
		\author{Yan He, Johannes Testorf and Xu Wang}
		\date{\today}


		\begin{abstract} This paper is an attempt to study a few \emph{effective} universality phenomenons in complex geometry using deformations and the H\"ormander $L^2$-method. The main idea of the proof is to deform the general case to the toric (or convex) case and to apply two new Berndtsson-Lempert type Ohsawa-Takegoshi extension theorems. The first one is based on the Ross-Witt Nystr\"om correspondence picture, Darvas-Xia-Zhang's asymptotic slope formula for Ding-type functionals and Berndtsson's monotonicity theorem. The second version is proved using an $S^1$-symmetrization method for PSH potentials based on Witt Nystr\"om's canonical K\"ahler deformations and Berndtsson-P\u aun's positivity theorem for Stein fibrations associated to deformations to normal bundles. The resulting effective universality phenomenons include a sharp lower bound of the Bergman kernel for compact Riemann surfaces, an effective Okounkov body construction and a sharp Faber-Widom type Bergman approximation of the logarithmic capacity. 
		\end{abstract}

		\maketitle

\section{Introduction}

"In statistical mechanics, universality is the observation that there are properties for a large class of systems that are independent of the dynamical details of the system" (from Wikipedia). In mathematics, it is usually used to describe the phenomenon that certain class of mathematical models share a common type of scale-invariant limit. The most basic example is perhaps \emph{Laplace's method}. A special case is the following: let $u$ be a smooth strictly  convex \emph{even} function on $\mathbb R$ with $u(0)=0$, then the leading $t\to \infty$ asymptotic term of the Laplace type integral $\int e^{-tu}$  depends only on the behavior of $u$ around its minimum point $0$, more precisely we have  
\begin{equation}\label{eq:Laplace}
\int_{\mathbb R} e^{-tu} = \sqrt{\frac{2\pi}{tu''(0)}} + O(t^{-1}) \ \ \ \text{as $t\to \infty$.}
\end{equation}
The proof of \eqref{eq:Laplace} is not hard (see \cite[Proposition 2.1, page 323]{SS12}), but \emph{effective} estimate of $\int_{\mathbb R} e^{-tu}$ for fixed $t>0$ can be very difficult. One version of such an estimate is
\begin{equation}\label{eq:Laplace1}
\frac{4}{\int_{\mathbb R} e^{-(tu)^*}}\leq 	\int_{\mathbb R} e^{-tu}  \leq \frac{2\pi}{\int_{\mathbb R} e^{-(tu)^*}},
\end{equation}
where 
$
u^*(\alpha):=\sup_{x\in \mathbb R} \{\alpha x-u(x)\}	  
$
denotes the \emph{Legendre transform} of $u$ (so that $(tu)^*(\alpha)=t u^*(\alpha/t)$). Assume further that $u'(\mathbb R)=\mathbb R$ so that $u^*$ is also smooth, strictly convex, even on $\mathbb R$ with 
$
(u^*)''(0) u''(0) =1.
$
If we apply \eqref{eq:Laplace} to both $u$ and $u^*$, then we have
$$
F(t):=	\int_{\mathbb R} e^{-tu} 	\int_{\mathbb R} e^{-(tu)^*} = t 	\int_{\mathbb R} e^{-tu} 	\int_{\mathbb R} e^{-tu^*} \to 2\pi \ \ \ \text{as $t\to \infty$.}
$$
$2\pi$ is the upper bound for $F$ since we have the functional form of the Blaschke-Santal\'o inequality (see \cite{Bal86})
$
\int_{\mathbb R^n} e^{-u} 	\int_{\mathbb R^n} e^{-u^*} \leq (2\pi)^n
$
for all smooth strictly  convex \emph{even} function $u$ on $\mathbb R^n$. In case $n=1$, Fradelizi and Meyer \cite{FM08} further prove that
$
(tF(t))' \geq 4
$
 for all $t>0$, 
which gives $tF(t) \geq 4t$ and $F(1)\geq 4$, i.e.
$$
\int_{\mathbb R} e^{-u} 	\int_{\mathbb R} e^{-u^*} \geq 4.
$$
The lower bound $4^n$ for general $n$ is the (still open) functional version of Mahler conjecture (see \cite[Chapter 6]{BFR26} and references therein for the background and recent progresses). Other famous universality results include the central limit theorem, concentration of measure phenomenon and the Gaussian approximation theory (see \cite{BCG23}). To obtain the  effective version, one approach is to use Poincar\'e type inequalities (see \cite{KL25}). Our motivation is to study a few effective universality phenomenons (Bergman kernel asymptotics, Okounkov body construction and polynomial approximation of  the logarithmic capacity) in complex geometry using complex versions of the Legendre transforms (so called partial Legendre transform \cite{Ki78}) and the \emph{H\"ormander $L^2$ theory} \cite{Hor65} (a complex version of the Poincar\'e  inequality). 

\subsection{Effective Bergman kernel asymptotics}

Let $(L, e^{-\phi})$ be a positive (means $\phi$ is smooth and $\omega:=i\partial\dbar\phi>0$)  line bundle over a compact complex manifold $X$. We look at the Bergman kernel form
$$
K_{m\phi}:=\sup_{f\in H^0(X, K_X+mL)} \frac{i^{n^2} f\wedge \bar f e^{-m\phi}}{\int_X i^{n^2} f\wedge \bar f e^{-m\phi}}.
$$
A classical result in K\"ahler geometry is the following asymptotic formula (see \cite{Bou90, Cat99, Lu00, Ruan98, Tian90, Yau87, Zel98})
$$
K_{m\phi} = \frac{c_1^n}{n!} \left(m^n - \frac S2 m^{n-1} + O(m^{n-2}) \right)
$$
for large $m$, where $c_1:=(i\partial\dbar\phi)/(2\pi)$ and  
$$
S:=\frac{n (i\dbar\partial \log c_1^n) \wedge (i\partial\dbar\phi)^{n-1}}{ (i\partial\dbar\phi)^n}
$$ 
denotes the scalar curvature of the K\"ahler metric $\omega:=i\partial\dbar\phi$. A naive effective version would be the following: assume that $|S|\leq 1$, does there exist a universal constant $C$ such that
$$
| K_{m\phi} - \frac{c_1^n}{n!} m^n| \leq C m^{n-1}\frac{c_1^n}{n!}. 
$$
However, Lu's result \cite[Theorem 1.2]{Lu01} implies that such an estimate is not true even for $n=1$ and suggests extra assumptions on the injectivity radius or the \emph{infimum $L_0$ of the length of non-constant closed geodesics} for $(X, i\partial\dbar\phi)$. With that assumption the following is proved:

\begin{theorem}[Theorem A in \cite{Wang23}]\label{th:new0} Let $(L, e^{-\phi})$ be a positive  line bundle over a compact Riemann surface $X$. Assume that $
		S\leq 1$ and   $L_0 \geq 2\pi
	$, 
	then $K_\phi\geq c_1/4=(i\partial\dbar \phi)/(8\pi)$.	
\end{theorem}

A better lower bound $
K_\phi \geq c_1/2
$  is conjectured in \cite{Wang23}. This lower bound is sharp, to see this, note that it implies
$$
\dim H^0(X, K_X+L) =\int_X K_\phi \geq \int_X  c_1/2=  \frac{\deg L}{2}.
$$
In case $X=\mathbb P^1$, $L=-K_X =\mathcal O(2)$, we have
$$
1=\dim H^0(X, K_X+L) = \frac{\deg L}{2} =1,
$$ 
and $\phi=2\log (1+|z|^2)$ satisfies $
S=1, L_0 =2\pi
$,  thus $
K_\phi = c_1/2
$ is the sharp estimate under $
S\leq 1$ and   $L_0 \geq 2\pi
$. We confirm this optimal estimate in this paper.

\medskip

\noindent
\textbf{Theorem A.} \emph{Let $(L, e^{-\phi})$ be a positive line bundle over a compact Riemann surface $X$. Assume that $
	S\leq 1$ and  $L_0 \geq 2\pi
	$, 
	then $
	K_\phi \geq c_1/2
	$.}

\medskip

\noindent
\textbf{Remark.} \emph{This result is not directly related to the Suita conjecture \cite{Sui72} proved in \cite{Blo13, GZ15a}, which is mainly for the Bergman kernel on open Riemann surfaces. As far as we know, the statement for the compact case of the Suita conjecture is unknown (see \cite{Ho19, Kik23, Ina22} for related result). The second author of this paper further  proved that \cite[Theorem A, C]{Tes26}
\begin{equation}\label{eq:Bm-n=1}
	c_1\left(m-\frac12\right) \leq   K_{m\phi} \leq c_1\left(m+ \frac32\right),
\end{equation} 
for all $m\geq 1$ if $|S|\leq 1$ and $L_0\geq 2\pi$. }

\subsection{Effective Okounkov body construction}

Let $T, L$ be holomorphic line bundles over an $n$-dimensional  projective manifold $X$. The \emph{volume} of $L$ is defined as
\begin{equation}
	{\rm vol}(L):=\lim_{m\to\infty}\frac{\dim_{\mathbb C} H^0(X, T+mL)}{m^n},
\end{equation}
(the limit exists and does not depend on $T$, see \cite[Theorem 1.1]{DX24}). We shall always assume that $L$ is \emph{big} (i.e. ${\rm vol}(L)>0$). Given $x\in X$ and a holomorphic coordinate  $z$ around $x$ with $z(x)=0$, one may write a section $F\in H^0(X, T+mL)$ locally as
\begin{equation}\label{eq:taylor}
	F(z)= \sum_{\alpha\in\mathbb N^n} F_\alpha z^\alpha \, \bm e_T\otimes \bm e_L^{\otimes m},
\end{equation}
where $\bm e_L, \bm e_T$ are holomorphic local frames of $L, T$ around $x$. We shall always assume that the coordinate $z$ is \emph{admissible} in the sense that each $z_j=0$ extends to an irreducible subvariety of $X$. Fix an additive order (a rank $n$ generalization of the vanishing order, see the definition below), one may assign a convex body to $L$ in the following way:

\begin{definition}\label{de:additive} By an additive order on $\mathbb N^n$, we mean a total order $\prec$ on $\mathbb N^n$ such that $0\prec \alpha$ for all $\alpha\in \mathbb N^n\setminus\{0\}$ and
	$$
	\alpha \prec \beta \Rightarrow \alpha+\gamma \prec \beta+\gamma, \ \ \ \forall \ \alpha, \beta, \gamma\in \mathbb N^n.
	$$
	Denote by $\min$ the minimum with respect to an additive order $\prec$, we call
	\begin{equation}\label{eq:add-val}
		\nu(F):=\min\{\alpha\in\mathbb N^n: F_\alpha\neq 0 \ \text{in \eqref{eq:taylor}}\}
	\end{equation}
	the valuation vector of $F$ (associated to $\prec$ and admissible $z$) and 
	\begin{equation}\label{eq:par-Oko}
		\Delta(L):=\text{closure of the convex hull of} \left\lbrace \frac{\nu(F)}m: F\in H^0(X, mL) \setminus\{0\} \right\rbrace,
	\end{equation}
	where $m$ runs through all positive integers, the Okounkov body of $L$. 
\end{definition}

\noindent
\textbf{Remark.} \emph{The Okounkov body $\Delta(L)$ depends on $z, \prec$ but not on $\bm e_L, \bm e_T$. It is well known that (see \cite[Theorem 4.9]{KK12} , \cite[Theorem A]{LM09} and \cite{Ok96} for the background) the Lebesgue measure $|\Delta(L)|$ of $\Delta(L)$ is equal to the volume ${\rm vol}(L)$ of $L$. Put
	\begin{equation}\label{eq:par-Oko-T}
	\Delta_{m, T}(L):=\text{convex hull of} \left\lbrace \frac{\nu(F)}m: F\in H^0(X, T+mL) \setminus\{0\} \right\rbrace,
\end{equation}
we also have that $\Delta_{m, T}(L)$ converges to $\Delta(L)$ with respect to the Hausdorff metric (see \cite[Theorem 5.17]{Xia25}) as $m\to \infty$. Our aim is to study the relation between $\Delta(L)\cap (\mathbb N/m)^n$ and $ \Delta_{m, T}(L)$. The following result is a partial answer for $T=K_X$.}

\medskip

\noindent
\textbf{Theorem B.} \emph{Let $L$ be a holomorphic line bundle over a projective manifold $X$. Then with the notations above, for every $v\in \mathbb N^n$ with $v+(1,\cdots, 1)$ lies in the interior $\Delta(mL)^{0}$ of  $\Delta(mL)$, there exists $F\in H^0(X, K_X+mL)$ with $\nu(F)=v$, i.e
\begin{equation}\label{eq:thB-new}
		\left(\Delta(mL)^{0}- ( 1, \cdots, 1)\right) \cap \mathbb N^n \subset   \left\lbrace \nu(F): F\in H^0(X, K_X+mL) \setminus\{0\} \right\rbrace,
\end{equation}
	for all positive integer $m$, in particular, 
	$
	\left(\Delta(L)^{0}- \left( \frac1m, \cdots, \frac1m\right)\right) \bigcap \left(\frac{\mathbb N}{m}\right)^n \subset 	\Delta_{m, K_X}(L).
	$}

\medskip

\noindent
\textbf{Remark.} \emph{In case $(X,L)$ is defined by a Delzant polytope , then $-K_X=\sum D_f$ is the sum of face divisors, which explains why the vector $(1,\cdots, 1)$ appears in \eqref{eq:thB-new}. Theorem B is the minimal singularity case (replace $L$ by $mL$) of the following theorem.}

\medskip

\noindent
\textbf{Theorem B'.} \emph{Let $(L, e^{-\phi})$ be a pseudoeffective (means $\phi$ is psh) line bundle over a projective manifold $X$. If for some $v\in \mathbb N^n$, the interior $\Delta_\phi^0$ of the convex set
	\begin{equation}\label{eq:partial-OK}
		\Delta_\phi:=	\text{closure of the convex hull of} \left\lbrace \frac{\nu(F)}k: F\in H^0(X, kL), \,  \sup_X |F|^2e^{-k\phi}= 1 \right\rbrace,
	\end{equation}
contains $v+(1,\cdots, 1)$, then there exists $F\in H^0(X, \mathcal O(K_X+L)\otimes \mathcal I(\phi))$ with $\nu(F)=v$.}

\medskip

\noindent
\textbf{Remark.} \emph{In case $\phi$ has only analytic singularities and $i\partial\dbar\phi>0$, $	\Delta_\phi$ defined in \eqref{eq:partial-OK} is precisely the partial Okounkov body defined in \cite[see (5-1) and (5-7)]{Xia25}. A direct application of Theorem B' is the following lower bound estimate for the dimension of partial section space $H^0(X, \mathcal O(K_X+mL)\otimes \mathcal I(m\phi))$ (replace $L$ by $mL$ and note that if $v$ is integral interior point of a set in $\mathbb R_{\geq 0}^n$, then $v-(1,\cdots, 1)\in \mathbb N^n$).}

\medskip

\noindent
\textbf{Corollary B'.} \emph{With the notations in the above theorem, we have
	\begin{equation}\label{eq:coB-new}
		\dim H^0(X, \mathcal O(K_X+mL)\otimes \mathcal I(m\phi)) \geq \#   \left( m\Delta_\phi^{0}\cap \mathbb Z^n\right).
	\end{equation}}

\subsection{Effective polynomial approximation of capacity}

Recall that the Lelong class $\mathcal L(\mathbb C)$ is the space of subharmonic functions $\varphi$ on $\mathbb C$ such that 
$$
\sup_{z\in \mathbb C} \left\lbrace \varphi (z)-\log (1+|z|^2) \right\rbrace <\infty.
$$
Let $K$ be a compact set in $\mathbb C$. Assume that the unbounded component $\Omega$ of $\mathbb C\setminus K$ is \emph{Dirichlet regular} so that the \emph{equilibrium potential}
\begin{equation}\label{eq:equi-po}
\phi_K:=\sup\left\lbrace \psi\in\mathcal L(\mathbb C): \sup_K\psi \leq 0 \right\rbrace
\end{equation}
lies in $\mathcal L(\mathbb C)$ and is continuous on $\mathbb C$. We know that $\phi_K-\log |z^2|$ is harmonic near $z=\infty$. The \emph{Robin constant} of $K$ is defined by
$
\gamma_K:=\lim_{z\to \infty} \{\phi_K(z)-\log |z^2|\},
$
and we call $C(K):=e^{-\gamma_K}$ the \emph{logarithmic capacity} of $K$. Faber-Fekete-Szeg\"o \cite{Fab19, Fek23, Sze24} proved that 
\begin{equation}\label{eq:FFS}
C(K) \leq \inf_{P\in \mathcal P_m} ||P||_K^{2/m} \ \text{for all $m\geq 1$ and} \lim_{m\to\infty} \inf_{P\in \mathcal P_m} ||P||_K^{2/m} = C(K),
\end{equation}
where $\mathcal P_m$ denotes the space of monic degree $m$ polynomials on $\mathbb C$ and $||P||_K:=\sup_K |P|$. Notice that the definition of $\phi_K$ implies
$$
||P||_K^2= \sup_{z\in\mathbb C} |P(z)|^2e^{-m\phi_K(z)}, \ \ \forall \ P\in \mathcal P_m,
$$
which gives
$$
\pi\varepsilon^2 ||P||_K^2 e^{-(m+2)\sup_{K_\varepsilon} \phi_K}  \leq \int_{\mathbb C} |P|^2 e^{-(m+2)\phi_K} \leq ||P||_K^2 \int_{\mathbb C} e^{-2\phi_K}, \ \ \forall \ P\in \mathcal P_m, \ \varepsilon>0,
$$
where $K_\varepsilon:=K+\{|z|\leq\varepsilon\}$. Since $\phi_K$ is continuous and $=0$ on $K$, the above inequality and \eqref{eq:FFS}  together give that
\begin{equation}\label{eq:FFS-L2}
\lim_{m\to \infty} \inf_{P\in \mathcal P_m} \left(\int_{\mathbb C} |P|^2 e^{-(m+2)\phi_K} \right)^{\frac1m} =C(K).
\end{equation}	
In case $K=\{|z|\leq R\}$, we have $\phi_K(z)=\max\{\log|z/R|^2, 0\}$ and $C(K)=R^2$, thus
$$
\inf_{P\in \mathcal P_m} \int_{\mathbb C} |P|^2 e^{-(m+2)\phi_K}  =\pi C(K)^{m+1} \frac{m+2}{m+1}.
$$
The following result is a sharp estimate for general connected $K$.

\medskip

\noindent
\textbf{Theorem C.} \emph{Let $K$ be a connected compact set in $\mathbb C$ that contains at least two points. Then
	\begin{equation}\label{eq:thC-new}
1\leq 	\frac{\inf_{P\in \mathcal P_m} \int_{\mathbb C} |P|^2 e^{-(m+2)\phi_K}}{\pi C(K)^{m+1}}  \leq  \frac{m+2}{m+1},  \ \ \forall \ m\geq 0.
	\end{equation}}

\medskip

\noindent
\textbf{Remark.} \emph{The above result is compatible with a classical result of Faber-Widom \cite{Fab19, Wid69}: 
		\begin{equation}\label{eq:thC-FW}
	\lim_{m\to \infty} \frac{\inf_{P\in \mathcal P_m}||P||^2_K}{C(K)^{m}}  =1,
	\end{equation}
where $K$ is the closure of a Jordan domain with smooth boundary. When $K$ is disconnected, our proof still gives the upper bound in \eqref{eq:thC-new}, but not the lower bound since the proof needs the unbounded component of $\mathbb P^1\setminus K$ to be simply connected. The disconnected case, along with further generalizations of Theorem C, will be investigated in a forthcoming sequel to this paper.}

\bigskip

The proof of Theorem A, B', C depends on our two main theorems in section 2 and a generalization of Witt Nystr\"om's symmetrization theory for psh weights in section 3. The whole paper is organized as follows.

\bigskip

\subsection{Acknowledgements}

It is a pleasure to thank Bo Berndtsson for
his suggestions for a better presentation of our results and David Witt Nystr\"om for his comment on a possible simplification of our First Main Theorem which leads us to find the Second Main Theorem. 

\medskip

We also thank Tam\'as Darvas for his valuable comments (leading term in Theorem \ref{th:BKA-sub}) on our canonical growth condition lower bound of the Bergman kernel, and Mingchen Xia for finding several inaccuracies in an early version of this paper. 

\medskip

Thanks are also due to Robert Berman, Zbigniew Blocki and Mihai P\u aun for several useful discussions.

\subsection{On the use of AI}

GPT-6 Astra was used to find several typos and the following:

\medskip

(1): In the definition of the Okounkov body in section 1.2, the coordinates $z$ should be admissible (in general, the hypersurface $z_j=0$ is not algebraic, thus can not extend to an (algebraic) subvariety of our projective manifold $X$). However, for infinitesimal Okounkov bodies, non-admissible coordinates are also allowed;

\medskip

(2): The normal bundle $NZ_k$ in \eqref{eq:rho-alphaY-new-XYZ1} is isomorphic to
$\oplus_{j=1}^k E_j|_{Y_k}$, which each $E_j$ denotes the normal bundle for the embedding of $Y_j$ to $Y_{j-1}$. But in general, $NZ_k$ is not  isomorphic to the normal bundle for the embedding of $Y_k$ to $X$.

\medskip

(3): The proof of Theorem B' also works for $\phi$ with non-analytic singularity if we use $\Delta_\phi$ in \eqref{eq:partial-OK} to replace Xia's partial Okounkov body in \cite{Xia25}.

\medskip

(4): Theorem B implies Theorem \ref{th:E} directly (our proof in the early arXiv version of this paper is longer).

\tableofcontents

\bigskip

		\noindent
		\textbf{List of notations.}
		
		\noindent
		
		\medskip
		
		\begin{itemize}
			
			\item[(0)] $\mathbb Z$ (resp. $\mathbb N$): set of (resp. non-negative) integers;  $\mathbb H$: right half plane in $\mathbb C$;
			
			\item[(1)] psh or PSH: plurisubharmonic;
			
			\item[(2)] usc: upper semi-continuous; 
			
			\item[(3)] $(L,e^{-\phi})$ positive (resp. pseudoeffective): $\phi$ is smooth and $dd^c\phi:=\frac{i\partial\dbar
				\phi}{2
				\pi}>0$ (resp. $\phi$ is psh);
			
			\item[(4)] $\phi\in {\rm PSH}(L)$: $\phi$ is a psh metric potential of $L$;
			
			\item[(5)] ${\rm SR}_\phi$: a collection of subgeodesic rays emanating from $\phi$ --- Definition \ref{de:sg};

			\item[(6)] $\lambda_{\phi_t}$: critical value of $\phi_t\in {\rm SR}_\phi$ --- Definition \ref{de:sg};
			
			\item[(7)] $K_X$: canonical line bundle of $X$;
			
			\item[(8)] $H^0(X, K_X+L)$ or $H^0(X, \mathcal O(K_X+L))$: space of holomorphic sections of $K_X+L$; 
			
			\item[(9)] $\mathcal I(\psi)$: sheaf of germs of holomorphic functions $f$ with locally integrable $|f|^2e^{-\psi}$;

			\item[(10)] $NY$: normal bundle of a submanifold $Y$ ---\eqref{eq:NY-def};
			
			\item[(11)] $F_{Y,0}$: tautological extension of $F_Y$ ---Definition \ref{de:te};

			\item[(12)] $\nu_x(\psi)$: Lelong number of $\psi$ at $x$ --- \eqref{eq:lelong-new};

			\item[(13)] $\phi_{Y}$: $S^1$-symmetrization on $NY$  of the metric potential $\phi$ --- \eqref{eq:S1-symm};
			
			\item[(14)] Upper semi-continuous regularizations of the supremum  and the limit are written as $$\sup^{\star}\, \{\cdots\}, \ \ \lim ^{\star}\, \{\cdots\};$$
			
			\item[(15)] $dz:=dz_1\wedge \cdots\wedge dz_n$.
			
		\end{itemize}
		
		\section{Berndtsson-Lempert method}

\subsection{Ohsawa-Takegoshi extension theorem}

The first breakthrough for the (psh) weighted $L^2$ extension of holomorphic functions is the following:

\begin{theorem}[Ohsawa-Takegoshi \cite{OT87}]\label{th:OT} Let $\phi$ be a psh function on a bounded pseudoconvex domain $\Omega\subset\mathbb C^n$. Then for any holomorphic function $f$ defined on $\Omega\cap H$, where $H$ is a hyperplane, there exists a holomorphic function $F$ in $\Omega$ such that $F=f$ on $\Omega\cap H$ and 
	$$
	\int_{\Omega} |F|^2 e^{-\phi} \leq C \int_{\Omega\cap H} |f|^2 e^{-\phi}, 
	$$
	where $C$ depends only on the diameter of $\Omega$ and we omit the Lebesgue measure in the integrals.
\end{theorem}

The original proof in \cite{OT87} uses the twisted H\"ormander $L^2$-estimate with respect to a certain singular weight (see \cite[section 2]{OT87}). This singular weight approach is further developed in \cite{Bern96, Bern05, BCP24, Blo13, Blo15, CDM17, CWW15, Dem92, Dem18, DK01, Guan19, GZ15a, GZ15b, Kim10, Kim21, KS23, Lem17, Ma18, MV07, MV19, MV17, Ohs88, Ohs95, Ohs01, Siu98, ZZ19}, to cite just a few. Inspired by \cite{Bern96}, Chen \cite{Chen11} found a simple proof using the original (non-twisted) H\"ormander $L^2$-estimate \cite[Theorem 2.2.1]{Hor65}. A few years later, a new approach was introduced by Berndtsson-Lempert \cite{BL16}. Our aim is to further extend the Berndtsson-Lempert method.

\subsection{Classical Berndtsson-Lempert method}

		Let us recall the Berndtsson-Lempert proof \cite{BL16} of Theorem \ref{th:OT}. For simplicity's sake, let us assume that the hyperplane $H$ in Theorem \ref{th:OT} is defined by $z_n=0$ and $\Omega$ lies inside the $|z_n|<1$ neighborhood of $H$ so that
\begin{equation}\label{eq:G-B0}
	G(z):=\log|z_n^2|
\end{equation}
is negative on $\Omega$ and has logarithmic pole at $H$. The main idea in \cite{BL16} is to consider a family of domains,
$$
\Omega_t:=\{z\in \Omega: G(z)<-t\}, \ \  \  t\geq 0,
$$
and to look at the optimal extension $F_t$ of $f$ to $\Omega_t$ (optimal means the square norm 
\begin{equation}\label{eq:Norm-B0}
	||F_t||_t^2:=\int_{\Omega_t} |F_t|^2 e^{-\phi}
\end{equation}
is minimal among all extensions of $f$ to $\Omega_t$). One then proves that 
\begin{equation}\label{eq:Norm-B1}
	\text{$e^{t} ||F_t||_t^2$ is increasing in $t\geq0$},
\end{equation} 
which shows that we can estimate the optimal extension to $\Omega_0=\Omega$ by the optimal extension to $\Omega_t$ for a very large $t$. The latter is easier to estimate since the norm of the extension can be estimated by $f$ as $t$ tends to infinity (see \cite[Lemma 3.7]{BL16}). Thus one obtains an extension $F$ of $f$ to $\Omega$ with
\begin{equation}\label{eq:norm-Bo-0}
	\int_{\Omega} |F|^2 e^{-\phi}  \leq \lim_{t\to \infty} e^{t} ||F_t||_t^2 = \pi \int_{\Omega\cap H} |f|^2 e^{-\phi}.
\end{equation}
For technical reasons (see \cite[the argument after (3.4)]{BL16}), however, instead of looking at norms over $\Omega_t$, Berndtsson and Lempert consider weighted $L^2$-norms
\begin{equation}\label{eq:norm-Bo}
	\int_{\Omega} |F|^2 e^{-\phi_t}
\end{equation}
over $\Omega$ (fixed) with weights ($\lambda>0$ is a constant)
\begin{equation}\label{eq:utlambda}
	\phi_t:=\phi+\lambda \max\{G+t, 0\}.
\end{equation}
Note that $e^{-\phi_t}$ equals $e^{-\phi}$ close to $H$ where $G\leq-t$, but tends to zero when $\lambda$ tends to infinity where $G>-t$. These $L^2$-norms satisfy the same monotonicity property as before and are good substitutes for the $L^2$-norms over $\Omega_t$. In particular, since $\Omega$ is fixed, they can apply the Berndtsson convexity theorem \cite{Bern06, Bern09} to prove a monotonicity property (see \cite{Al24, Ngu23, NW24, Wang23} for related results).

A basic observation in this paper is that (for fixed $\lambda$) $\phi_t$ is a ``subgeodesic ray'' emanating from $\phi$  (see the definition below) since $\lim_{t\to 0} \phi_t =\phi$ and $\phi_{{\rm Re} \,\tau}(z)$ is psh in $(z,\tau)\in \Omega\times \mathbb H$, where
$$
\mathbb H:=\{\tau\in\mathbb C: {\rm Re}\,\tau>0\}
$$
denotes the right half plane.

\begin{definition}\label{de:sg} Let $(L, e^{-\phi})$ be a pseudoeffective (means $\phi$ is psh) line bundle over a complex manifold $X$. Denote by ${\rm PSH}(L)$ the space of psh metric potentials of $L$. A map 
	\begin{align*}
		\mathbb R_{>0} & \to {\rm PSH}(L) \\
		t & \mapsto \phi_t
	\end{align*}
	is called a subgeodesic ray emanating from $\phi$ if $\lim_{t\to 0} \phi_t(z)=\phi(z)$ and $\phi_{{\rm Re}\, \tau}(z)$ is psh in $(z, \tau)\in X\times \mathbb H$. We shall denote by ${\rm SR}_\phi$ the space of subgeodesic rays $\phi_t$ emanating from $\phi$ with $\phi\leq \phi_t\leq \phi+Ct$ for some constant $C>0$ and call the sharp constant $C$, denoted by $\lambda_{\phi_t}$, the critical value of $\phi_t$. 
\end{definition}

This leads us to consider the same kind of arguments for general subgeodesic rays, and we shall see that this in some cases is more efficient. One such example is a proof of Theorem A, in which case, this generalization is the key ingredient.

\subsection{A pre-extension theorem of Berndtsson for general subgeodesic rays}

Let $(L, e^{-\phi})$ be a pseudoeffective line bundle over a complex manifold $X$. Let $\phi_t$ be a subgeodesic ray in ${\rm SR}_\phi$ with critical value $\lambda_{\phi_t}$ (see Definition \ref{de:sg}). In order to generalize \eqref{eq:norm-Bo-0}, Berndtsson introduced the following subspaces
		\begin{equation}\label{eq:salpha-intro}
			S_\alpha:=\left\lbrace F\in H^0(X, K_X+L): \int_{0}^\infty ||F||_t^2 \,e^{t\alpha} \,dt<\infty  \right\rbrace,  \ \  	||F||^2_t:=\int_{X} i^{n^2} F\wedge \bar F \, e^{-\phi_t},
		\end{equation}
		of $H^0(X, K_X+L)$ and proved the following theorem (see the proof of \cite[Theorem 1.1]{Bern20}):
		
		\begin{theorem}\label{th:Bern20} Assume that $X$ is compact K\"ahler.  Fix $F\in H^0(X, K_X+L)$ with 
			$$
			||F||^2_0:=\int_{X} i^{n^2} F\wedge \bar F \, e^{-\phi}<\infty.
			$$
			Then there exists an $F_\alpha\in  H^0(X, K_X+L)$ such that $F_\alpha-F\in S_\alpha$ and
			\begin{equation}\label{eq:Bern20}
				||F_\alpha||_0^2 \leq \liminf_{t\to\infty} e^{\alpha  t} ||F||_t^2.
			\end{equation}
		\end{theorem} 
		
		Denoting by $Y_\alpha$ the common zero locus of sections in $S_\alpha$ and by $f_\alpha$ the restriction to $Y_\alpha$ of $F$, the above theorem implies that $f_\alpha$ has an extension $F_\alpha$ satisfying \eqref{eq:Bern20}. Since $F$ itself is already an extension of $f_\alpha$, we call the above theorem the pre-extension theorem of Berndtsson. The \emph{two new ideas} (both transforming the above theorem into certain Ohsawa-Takegoshi type theorems) in this paper are:
		
		\smallskip
		
		(1) one may use Ross-Witt Nystr\"om correspondence to describe $S_\alpha$ and $Y_\alpha$;
		
		(2) the deformation to the normal bundle construction can be used to decode the right hand side of \eqref{eq:Bern20} and generalize the above theorem to certain non-compact cases.
		
		\subsection{A variation of Darvas-Xia-Zhang's result}
		
			Our starting point is the following result, which is based on the proof of the asymptotic slope formula of Ding type functionals in \cite[Lemma 4.3]{DX22}  and \cite[Theorem 4.1]{DZ22}:
		
\begin{theorem}\label{th:DXZ} Denote by
			\begin{equation}\label{eq:RWnew1}
				\psi_\alpha:=\inf_{t>0}\{\phi_t-t\alpha\},
			\end{equation}	
			the Kiselman partial Legendre transform \cite{Ki78} of the subgeodesic ray $\phi_t$. Put
			\begin{equation}\label{eq:RWnew2}
				H_\alpha:=\left\lbrace F\in H^0(X, K_X+L): \int_{X} i^{n^2} F\wedge \bar F \, e^{-\psi_\alpha}<\infty  \right\rbrace.
			\end{equation}
			Then 
			\begin{equation}\label{eq:RWnew3}
				H_{\alpha+\varepsilon}	\subset S_\alpha\subset H_\alpha, \ \forall \ \varepsilon>0.
			\end{equation}		
			In particular, $S_\alpha=H_{\alpha}$ for $\alpha<\lambda_{\phi_t}$ in case $X$ is compact.
\end{theorem}

\begin{proof} 	We shall follow the proof of Lemma 4.3 in \cite{DX22} (see also Theorem 4.1 in \cite{DZ22}). By \eqref{eq:RWnew1},  if $\psi_\alpha(x)>-\infty$ then there exists $t_0>0$ such that
	$
	\psi_\alpha(x)+1\geq \phi_{t_0}(x) -t_0 \alpha,
	$
	i.e.
	$$
	\psi_\alpha(x)+1 -t_0(\lambda_{\phi_t}-\alpha)\geq \phi_{t_0}(x) -t_0 \lambda_{\phi_t}.
	$$
	Since $\phi_t -t\lambda_{\phi_t}$ is decreasing in $t$, we get
	$$
	\psi_\alpha(x)+1 -t_0(\lambda_{\phi_t}-\alpha)\geq \phi_{t}(x) -t \lambda_{\phi_t}, \ \ \ \forall \ t\geq t_0,
	$$
	which gives
	$$
	-\phi_t(x) + t \alpha \geq -\psi_\alpha(x) -1-(t-t_0)(\lambda_{\phi_t}-\alpha), \  \ \ \ \forall \ t\geq t_0.
	$$
	If $\alpha\geq \lambda_{\phi_t}$ then $S_\alpha=\{0\}$ and we always have $S_\alpha \subset H_\alpha$. If $\alpha<\lambda_{\phi_t}$ then the Lebesgue measure of $\{\psi_\alpha=-\infty\}$ is zero. Thus for $F\in S_\alpha$ we have 
	$$
	\infty > \int_{\psi_\alpha>-\infty} i^{n^2} F\wedge \bar F \, \int_0^\infty e^{-\phi_t+t\alpha} \,dt \geq \frac{e^{-1}}{\lambda_{\phi_t}-\alpha} \int_X i^{n^2} F\wedge \bar F \, e^{-\psi_{\alpha}}. 
	$$
	Hence $S_\alpha \subset H_\alpha$.  $H_{\alpha+\varepsilon} \subset S_\alpha$ follows from 
	$$
	\int_{0}^\infty ||F||_t^2 e^{t\alpha} \,dt = \int_X i^{n^2} F\wedge \bar F \,  \int_0^\infty e^{-\phi_t+t\alpha} \,dt \leq \frac1\varepsilon\int_X i^{n^2} F\wedge \bar F \, e^{-\psi_{\alpha+\varepsilon}}.
	$$
	In case $X$ is compact, the strong openness theorem in \cite[(3.9)]{GZ15b} implies that $H_{\alpha+\varepsilon}=H_\alpha$ for $\varepsilon>0$ sufficiently small. Thus $H_\alpha=S_\alpha$ in the compact case. 
\end{proof}
		
		\subsection{Ross-Witt Nystr\"om correspondence}		
		In \cite{RW14} Ross and Witt Nystr\"om  observed that the partial Legendre transform of the subgeodesic ray space ${\rm SR}_\phi$  is precisely the test curve space ${\rm TC}_\phi$ defined below.
		
		\begin{definition}\label{de:btc} Let $(L, e^{-\phi})$ be a pseudoeffective  line bundle over a complex manifold $X$. A map 
			\begin{align*}
				\mathbb R & \to {\rm PSH}(L) \\
				\alpha & \mapsto \psi_\alpha
			\end{align*}
			is called a $\phi$-test curve if for every fixed $z$, $\lim_{\alpha\to-\infty} \psi_\alpha(z)=\phi(z)$ and $\alpha\mapsto \psi_\alpha(z)$ is concave, decreasing and usc. We shall denote by  ${\rm TC}_\phi$ the space of $\phi$-test curves $\psi_\alpha$ with
			\begin{equation}\label{eq:btc}
				\text{critical value}\, \lambda_{\psi_\alpha}:= \inf\{\alpha\in\mathbb R: \psi_\alpha\equiv -\infty\}<\infty  \   \text{and  $\psi_\alpha\equiv \phi$ for all $\alpha\leq 0$}.
			\end{equation}	
		\end{definition}
		
		\noindent
		\textbf{What does a test curve look like?} We will illustrate here with an example. In case $$X\subset \{|z_n|<1\}$$ is a pseudoconvex domain and $L$ is the trivial line bundle equipped with the psh weight $\phi$, then for each positive constant $\lambda$, $\phi_t(z):=\phi(z)+ \lambda \max\{\log|z_n^2|+t, 0\}$ is a subgeodesic ray in ${\rm SR}_\phi$. A direct computation gives the following partial Legendre transform formula
		$$ 
		\inf_{t>0}\{\phi_t-t\alpha\}=
		\begin{cases}
			\phi & \alpha\leq 0; \\
			
			\phi+\alpha\log|z_n^2| & 0<\alpha\leq \lambda; \\ 
			
			-\infty & \alpha>\lambda.
		\end{cases}
		$$	
		In this case, the test curve $\psi_\alpha:=		\inf_{t>0}\{\phi_t-t\alpha\}$ is just a linear family of singular weights $\phi+\alpha\log|z_n^2|$ with critical value $\lambda_{\psi_{\alpha}}=\lambda$. For a general pseudoeffective  line bundle $(L, e^{-\phi})$,  if $F$ is a holomorphic section of $L$ with $|F|^2e^{-\phi}\leq 1$, then 
		$$ 
		\psi_\alpha:=
		\begin{cases}
			\phi & \alpha\leq 0; \\
			
			\phi+\alpha\log
			\left(|F|^2e^{-\phi}\right) & 0<\alpha\leq 1; \\
			
			-\infty & \alpha>1.
		\end{cases}
		$$
		is a test curve in ${\rm TC}_\phi$ with $\lambda_{\psi_\alpha}=1$. In general, if we fix $\sigma\in {\rm PSH}(L)$ with $\sigma\leq \phi$, then
		\begin{equation}\label{eq:lineartc0}
			\psi_\alpha:=
			\begin{cases}
				\phi & \alpha\leq 0; \\
				
				\alpha \sigma+(1-\alpha)\phi & 0<\alpha\leq 1; \\
				
				-\infty & \alpha>1.
			\end{cases}
		\end{equation}
		gives a test curve in ${\rm TC}_\phi$ with $\lambda_{\psi_\alpha}=1$.
		
		\medskip
		
		The following Legendre duality result (see \cite[Proposition 3.6]{DZ22}) is a preliminary version of the \emph{Ross-Witt Nystr\"om correspondence} \cite{RW14} (see Theorem \ref{th:RW} for the complete version).	
		
		\begin{theorem}\label{th:RW0} The partial Legendre transforms
			\begin{equation}\label{eq:plt}
				\hat \psi_t:=\sup_{\alpha\in \mathbb R}\, \{\psi_\alpha+\alpha t\}, \ \ \check{\phi}_\alpha:=\inf_{t>0} \{\phi_t-t\alpha\} ,
			\end{equation}
			give a bijection between ${\rm TC}_\phi$ and ${\rm SR}_\phi$. Moreover, we have $\check{\hat \psi}_\alpha=\psi_\alpha$, $\hat{\check \phi}_t=\phi_t$ and $\lambda_{\psi_\alpha}=\lambda_{\hat \psi_t}$.
		\end{theorem}

\subsection{First Main Theorem}

		Our First Main Theorem is the following Ohsawa-Takegoshi type theorem based on Theorem \ref{th:Bern20}, Theorem \ref{th:DXZ}, Theorem \ref{th:RW0} and the main theorem in \cite{CDM17, ZZ19, GZ15b}. To obtain an honest extension from a submanifold $Y$ to $X$, one may choose a $\phi$-test curve $\psi_\alpha$ such that $\mathcal I(\psi_\alpha)= \mathcal I_Y$ for some $\alpha$ (to construct such a test curve, one may use \eqref{eq:lineartc0}; see also \eqref{eq:valphaY-k} for a canonical envelope construction) and define the subgeodesic ray $\phi_t$ to be $\hat \psi_t$. Then we have $f\in H^0(Y, (K_X+L)|_Y)$ and for that $\phi_t$, the theorem below gives an extension of $f$ from $Y$ to $X$ (the right hand side of \eqref{eq:FMT3} can be seen as a ``generalized Ohsawa-norm'' by the remark for \cite[Theorem 2.5]{NW24}).

		\medskip
		\noindent
		\textbf{First Main Theorem.} \emph{Let $(L, e^{-\phi})$ be a pseudoeffective line bundle over an $n$-dimensional compact K\"ahler manifold $X$. For every subgeodesic ray $\phi_t\in {\rm SR}_\phi$, we look at the test curve
			\begin{equation}\label{eq:FMT0}
				\psi_\alpha:=\check \phi_\alpha:=\inf_{t>0}\{\phi_t-t\alpha\}.
			\end{equation}	
			Then for $0<\alpha<\lambda_{\phi_t}$, the following restriction mapping
			\begin{equation}\label{eq:FMT1}
				\iota: H^0(X, \mathcal O(K_X+L)\otimes \mathcal I(\phi)) \to H^0(X, \mathcal O(K_X+L) \otimes \mathcal I(\phi) /\mathcal I(\psi_\alpha))
			\end{equation}
			is surjective, moreover, every section $f\in H^0(X, \mathcal O(K_X+L)\otimes \mathcal I(\phi) /\mathcal I(\psi_\alpha))$ has an extension $F_X\in  H^0(X, \mathcal O(K_X+L)\otimes \mathcal I(\phi))$ (in the sense that $\iota(F_X)=f$) with
			\begin{equation}\label{eq:FMT2}
				\int_X i^{n^2} F_X\wedge \overline{F_X} \, e^{-\phi} \leq ||f||^2,
			\end{equation}
			where 
			\begin{equation}\label{eq:FMT3}
				||f||^2:=\liminf_{t\to \infty} e^{\alpha t} \int_X i^{n^2} F\wedge \bar F \, e^{-\phi_t},
			\end{equation}
			for some $F$ with $\iota(F)=f$ ($||f||^2$ does not depend on the choice of $F$, see Proposition \ref{pr:FMT} below).}
		
\begin{proof} 	By the strong openness theorem, we can choose $\alpha<\beta<\lambda_{\phi_t}$ such that
	\begin{equation}\label{eq:beta-new}
		\mathcal I(\psi_\beta)=\mathcal I(\psi_\alpha).
	\end{equation}
	Put
	$
	\rho:=\frac{\alpha}{\beta} \psi_\beta+(1-\alpha/\beta) \phi
	$
	the concavity of $\alpha\to \psi_\alpha$ implies that $\rho\leq \psi_\alpha$.  Now $\phi\geq \psi_\beta$ implies that $\rho\geq \psi_\beta$, thus \eqref{eq:beta-new} gives $\mathcal I(\rho)=\mathcal I(\psi_\alpha)$. Note that $\psi_\beta\in {\rm PSH}(L)$ implies that
	$$
	\frac{\beta}\alpha (\rho-\phi)+ \phi\in {\rm PSH}(L).
	$$
	Since $\frac{\beta}\alpha>1$, the main theorem in \cite{CDM17} (see also \cite{ZZ19}) gives the surjectivity of 
	$$
	H^0(X, \mathcal O(K_X+L)\otimes \mathcal I(\phi)) \to H^0(X, \mathcal O(K_X+L) \otimes \mathcal I(\phi) /\mathcal I(\rho)).
	$$
	Thus \eqref{eq:FMT1} is surjective.
	Take  $F\in H^0(X, \mathcal O(K_X+L)\otimes \mathcal I(\phi))$ such that $\iota(F)=f$. By Theorem \ref{th:DXZ} and Theorem \ref{th:Bern20}, there exist $F_\alpha\in H^0(X, \mathcal O(K_X+L)\otimes \mathcal I(\phi))$ such that $F_\alpha-F \in H_\alpha$ (hence $\iota(F_\alpha)=\iota(F)=f$) and
	$$
	||F_\alpha||_{0}^2  \leq \liminf_{t\to\infty}e^{\alpha t} ||F||_{t}^2.
	$$
	Thus our first main theorem follows by taking $F_X=F_\alpha$. 
\end{proof}

		\begin{proposition}\label{pr:FMT}  $||f||^2$ in \eqref{eq:FMT3} does not depend on the choice of $F$.	
\end{proposition}

\begin{proof} If both $F_1$ and $F_2$ satisfies $\iota(F_1)=\iota(F_2)=f$ then
	$$
	\int_{X} i^{n^2} (F_1-F_2) \wedge \overline{(F_1-F_2)}\, e^{-\psi_\alpha}<\infty.
	$$
	Thus the strong openness theorem implies that 
	$$
	\int_{X} i^{n^2} (F_1-F_2) \wedge \overline{(F_1-F_2)}\, e^{-\psi_{\alpha+\varepsilon}}<\infty
	$$
	for some $\varepsilon>0$. Notice that
	$$
	||F_1-F_2||_t^2 e^{t\alpha} \leq e^{-t\varepsilon}  \int_{X} i^{n^2} (F_1-F_2) \wedge \overline{(F_1-F_2)}\, e^{-\psi_{\alpha+\varepsilon}}.
	$$
	Hence
	$
	\liminf_{t\to \infty} ||F_1||_t^2 e^{t\alpha} =  	\liminf_{t\to \infty} ||F_2||_t^2 e^{t\alpha}.
	$
\end{proof}
		
		\subsection{Deformation to the Normal bundle}
		
		Our First Main Theorem has two drawbacks. First, $X$ is assumed to be compact (in the non-compact case,  $H^0(X, K_X+L)$ can be infinite dimensional, thus one can not use \cite[Theorem 1.1]{Bern20} to prove the non-compact version of Theorem \ref{th:Bern20}). However, in many non-compact cases, one can directly generalize the method in \cite{NW24} (see Theorem \ref{th:noncompact}). The real issue is about the limit in \eqref{eq:FMT3}, which is very hard to compute (or estimate) for a general ray $\phi_t$, even in the compact case. Our main observation (which originates from a discussion of the third author with David Witt Nystr\"om) is that: if $\phi_t$ comes from the deformation to the normal bundle construction, then \eqref{eq:FMT3} reduces to an integral over the normal bundle under certain extra assumptions (i.e. $(L, e^{-\phi})$ is positive and $\epsilon_Y(L)>m$ in \cite[Theorem B]{HTW24}). It is natural to ask whether there is an alternative approach that bypasses these assumptions. Our starting point is the following theorem of Ohsawa \cite{Ohs17} (see \cite{NW22, YLZ22} for different proofs and generalizations).

		\begin{theorem}\label{th:Ohsawa} Let $\phi$ be a psh function on  a pseudoconvex domain  $\Omega\subset\mathbb C^n$. Assume that $0\in \Omega$. Then for every toric psh function $\psi$ (toric means that $\psi$ depends only on $|z_1|, \cdots, |z_n|$) on $\mathbb C^n$, there exists a holomorphic function $f$ on $\Omega$ with 	\begin{equation}\label{eq:Ohsawa}
				f(0)=1, \ \  \int_{\Omega}|f|^2 e^{-\phi-\psi} \leq  \int_{\mathbb C^n} e^{-\phi(0)-\psi} .
			\end{equation}						
		\end{theorem}
		
		\noindent
		\textbf{Remark.} \emph{It is known that the toric assumption on $\psi$ can not be removed, see \cite{Guan20, YLZ22} for counterexamples in the $n=1$ case. But later we shall prove that for $n>1$, we only need an $S^1$-invariant assumption, see Theorem \ref{th:Ohsawa1} below.}   
		
		\medskip
		
		Our observation is: one may  translate (and generalize) the above theorem to the following version, which reduces to Theorem \ref{th:Ohsawa} in case 
		\begin{equation}\label{eq:trivial-phiC}
			\phi_{\mathbb C}(\xi,s)=\phi(s\xi)+\psi(\xi).
		\end{equation}

		\begin{theorem}\label{th:Ohsawa1} Let $\Omega$ be a pseudoconvex domain in $\mathbb C^n$ with $0\in \Omega$. Let $\phi_{\mathbb C}$ be a psh function on 
			\begin{equation}\label{eq:B1}
				\Omega_{\mathbb C}:=\{(\xi, s)\in \mathbb C^n\times\mathbb C: s\xi\in \Omega\}.
			\end{equation} 
			Assume that 
			\begin{equation}\label{eq:B2}
				\phi_{\mathbb C}(e^{i\theta}\xi, e^{-i\theta} s)=\phi_{\mathbb C}(\xi, s), \ \forall \ (\xi, s)\in\Omega_{\mathbb C}, \ \theta\in\mathbb R.
			\end{equation}
			Put $\phi_s(\xi):=\phi_{\mathbb C}(\xi,s)$. Then there exists a holomorphic function $f$ on $\Omega$ with 	\begin{equation}\label{eq:B3}
				f(0)=1, \ \  \int_{\Omega}|f|^2 e^{-\phi_1} \leq  \int_{\mathbb C^n} e^{-\phi_0}.
			\end{equation}
		\end{theorem}

		\begin{proof} If $\int_{\mathbb C^n} e^{-\phi_0}=\infty$ then it suffices to take $f=1$. Thus one may assume that $\int_{\mathbb C^n} e^{-\phi_0}<\infty$. The main observation is that $\Omega_{\mathbb C}$ is pseudoconvex (if $E$ is a psh exhaustion function on $\Omega$ then $E(s\xi)+|s|^2+|\xi|^2$ defines a psh exhaustion function on $\Omega_{\mathbb C}$). Denote by 
			$$
			K_s(0):=\sup_{f\in\mathcal O(\Omega_s)} \frac{|f(0)|^2}{\int_{\Omega_s} |f|^2 e^{-\phi_s}}
			$$ the Bergman kernel at $0\in \Omega_s:=\Omega_{\mathbb C} \cap \{\mathbb C^n\times s\}$ with respect to the weight $\phi_s$. Then \cite[Theorem 1.1]{Bern06} implies that $\log K_s(0)$ is subharmonic in $s\in\mathbb C$. By \eqref{eq:B2}, we know that $\log K_s(0)$ depends only on $|s|$ (thus must be increasing with respect to $|s|$ by the submean inequality) and 
			$
			K_0(0)= \frac1{\int_{\mathbb C^n} e^{-\phi_0}}
			$ (since $\phi_0$ is $S^1$-invariant on $\mathbb C^n$),
			from which \eqref{eq:B3} follows.	
		\end{proof}
		
		\noindent
		\textbf{Remark.} \emph{Both Theorem \ref{th:Ohsawa} and Theorem \ref{th:Ohsawa1} give a general estimate of the Bergman kernel. One example is the following result proved in \cite[Theorem 3.2]{BD26}: if $\varphi$ is psh on $\Omega$ and 
			\begin{equation}\label{eq:Bergman-condition-Bo}
				i\partial\dbar\varphi \geq C\, i\partial \dbar |z|^2,
			\end{equation}
			then one may simply choose $\psi=C|z|^2$ and take $\phi:=\varphi-\psi$. Theorem \ref{th:Ohsawa} then implies that
			\begin{equation}\label{eq:Bergman-Ohsawa-Bo}
				K_{\Omega, \varphi}(0):=\sup_{f\in\mathcal O(\Omega)} \frac{|f(0)|^2}{\int_{\Omega} |f|^2 e^{-\varphi}} \geq  \frac{e^{\varphi(0)} }{\int_{\mathbb C^n} e^{-C|z|^2}} =\frac{C^n \,e^{\varphi(0)}}{\pi^n}.
			\end{equation}
			In case $\Omega$ is the polydisk $\mathbb D^n$, one may further choose a family $\psi_j$ such that $\psi_j=C|z|^2$ in $\mathbb D^n$ and tends to infinity outside $\mathbb D^n$ as $j\to \infty$, then we recover the sharp estimate 
			\begin{equation}\label{eq:Bergman-Ohsawa-Bo-1}
				K_{\mathbb D^n, \varphi}(0)\geq  \frac{e^{\varphi(0)} }{\int_{\mathbb D^n} e^{-C|z|^2}}.
			\end{equation}
			in \cite[Corollary 3.4]{BD26}. The above proof of Theorem \ref{th:Ohsawa1} is much shorter than the one in \cite{NW22}. The main difference is that we  apply the non-product (i.e. $\Omega_{\mathbb C}$ is not a product domain in case $\Omega\neq \mathbb C^n$) version of Berndtsson's theorem directly to $\Omega_{\mathbb C}$, which contains the zero fiber. In this way we can directly compare the zero fiber $\Omega_0$ and $\Omega_1=\Omega$ (without taking limit as in \eqref{eq:FMT3}).  We call this new approach the ``normal bundle approach'' since the natural mapping
		\begin{equation}\label{eq:B-mu}
			\mu: \Omega_{\mathbb C} \to \Omega \times \mathbb C 
		\end{equation}
		defined by $\mu(\xi,s):=(s\xi, s)$ induces a deformation of $\Omega$ to the normal bundle  $\mathbb C^n$ of the submanifold $\{0\}\subset\Omega$.} 
		
		\medskip
		
		In the manifold case, the deformation to the normal bundle is defined via a gluing process. Let $Y$ be a codimension $k$ ($k\geq 1$) closed complex submanifold of $X$. Fix an open covering 
		$$
		Y\subset U:=\bigcup U_l
		$$ 
		with holomorphic coordinate chart $z^l$ such that
		$
		z^l(U_l)=\mathbb B_{k} \times \mathbb B_{n-k}, \ \ z^l(U_l\cap Y)=\{0\}\times \mathbb B_{n-k} ,
		$
		where $\mathbb B_k$ denotes the unit ball in $\mathbb C^k$. The total space  $X_{\mathbb C}$ is defined by gluing $X\times \mathbb C^*$ and each
		$$
		\mathcal B_l:=\{(\xi^l, w^l, s)\in \mathbb C^k \times \mathbb B_{n-k} \times  \mathbb C: s\xi^l \in \mathbb B_k\}
		$$
		via the change of variables (comparing with \eqref{eq:B-mu})
		\begin{equation}\label{eq:COV}
			\mu:	(\xi^l,w^l, s) \mapsto (z^l,s):=(s\xi^l, w^l, s).
		\end{equation}
		The trivial $\mathbb C^*$-action $a\cdot (z^l, s):=(z^l, a^{-1}s)$ on $X\times\mathbb C^*$ naturally extends to a (non-trivial) $\mathbb C^*$-action on $X_{\mathbb C}$ such that
		\begin{equation}\label{eq:Cstaraction-new}
			a\cdot (\xi^l, w^l,s):=(a\xi^l, w^l, a^{-1}s), \ \ \forall \ a\in \mathbb C^*.
		\end{equation}
		Similarly, the identity mapping on $X\times\mathbb C^*$ naturally extends to a holomorphic mapping
		\begin{equation}\label{eq:mu-new}
			\mu: X_{\mathbb C} \to X\times \mathbb C,
		\end{equation}
		such that $\mu(\xi^l,w^l, s)=(s\xi^l, w^l, s)$.

		\begin{definition}\label{de:dtn} We call $ \mu: X_{\mathbb C} \to X\times \mathbb C $ the deformation of $X$ to the normal bundle $NY$, where $NY$ is defined by
			\begin{equation}\label{eq:NY-def}
				NY:=(TX)|_Y / TY,
			\end{equation}	
			(thus it is a holomorphic quotient bundle over $Y$). The fiber at $s$ of the associated projection
			$$
			\mu_{\mathbb C}: X_{\mathbb C} \to \mathbb C
			$$
			is denoted by $X_s$, $\mu_s:=\mu_X|_{X_s}$, where $\mu_X: X_{\mathbb C} \to X$ means the projection to $X$.
		\end{definition}
		
		The zero fiber $X_0$ is isomorphic to $NY$ via
		\begin{equation}\label{eq:NY-def-1}
			\xi^l \mapsto \xi^l_1 \frac{\partial}{\partial z^l_1}+ \cdots + \xi^l_k \frac{\partial}{\partial z^l_k},
		\end{equation}	
		where we think of $\{\partial/\partial z^l_j\}_{1\leq j\leq k}$ as a local frame on $U_l\cap Y$ for $NY$. Via this isomorphism, $\mu_0$ equals the bundle projection mapping $NY \to Y\subset X$. Let $L$ be a holomorphic line bundle over $X$ such that $L|_{U_l}$ has a holomorphic frame $e_l$. We shall write
		\begin{equation}\label{eq:LC-def-new}
			L_{\mathbb C}:=\mu^*(L\times\mathbb C), \ \ \  L_s:=(L_{\mathbb C})|_{X_s},
		\end{equation}	
		where $L\times \mathbb C$ denotes the pull back of $L$ from $X$ to $X\times\mathbb C$.
		In particular, $L_1=L$ and $L_0=\mu_0^*(L|_{Y})$ has frame $\mu_0^* e_l$ on each domain $\mu_0^*(U_l\cap Y)$. In this way, every holomorphic section $F_Y\in H^0(Y, (K_X+L)|_Y)$ naturally extends to a section
		\begin{equation}\label{eq:FY-def-new}
			F_{Y, 0} \in H^0(X_0, K_{X_0}+L_0)
		\end{equation}	
		as follows: if $F_Y=f_l(z^l_{k+1}, \cdots, z^l_n) \,dz_1^l\wedge \cdots \wedge dz_n^l \otimes e_l$ on $U_l\cap Y$, then 
		\begin{equation}\label{eq:FY-def-new1}
			F_{Y, 0} := f_l(w^l_{k+1}, \cdots, w^l_n)\, d\xi_1^l\wedge \cdots \wedge d\xi_k^l \wedge dw^l_{k+1} \wedge \cdots \wedge dw_n^l \otimes \mu_0^* e_l
		\end{equation}	
		over $\mu_0^*(U_l\cap Y)$, in particular, $F_{Y, 0}$ has constant coefficient along the fiber of $\mu_0:NY\to Y$.		
		
		\begin{definition}\label{de:te} We call $F_{Y, 0}\in H^0(X_0, K_{X_0}+L_0)$ the tautological extension of $F_Y$ to $X_0=NY$.
		\end{definition}
	
	\subsection{Second Main Theorem}
		
		Now we are ready to state the manifold version of Theorem \ref{th:Ohsawa1} (in particular, if $\phi_{\mathbb C}$ is given as in \eqref{eq:trivial-phiC} then the following theorem reduces to Theorem \ref{th:Ohsawa}).
		
		\medskip
		\noindent
		\textbf{Second Main Theorem.} \emph{Let $L$ be a holomorphic line bundle over an $n$-dimensional complex manifold $X$. Let $Y$ be a closed complex submanifold of $X$. Assume that $X$ is $Y$-Stein in the sense that there exists a divisor $Z$ of $X$ such that $X\setminus Z$ is Stein and $Z\cap Y$ is a divisor of $Y$ (when $Y$ is a point we mean that $Z\cap Y$ is empty). Assume further that $(L_{\mathbb C}, e^{-\phi_{\mathbb C}})$ is $S^1$-pseudoeffective in the sense that $\phi_{\mathbb C}$ is psh and $S^1$-invariant with respect to the $\mathbb C^*$-action \eqref{eq:Cstaraction-new}. Then every  $F_Y\in H^0(Y, (K_X+L)|_Y)$ with
			\begin{equation}
				||F_Y||^2:=\int_{X_0} i^{n^2} F_{Y,0} \wedge \overline{F_{Y,0}} \, e^{-\phi_0} <\infty,  \ \ \phi_s:=\phi_{\mathbb C}|_{X_s},
			\end{equation}
			extends to a holomorphic section $F_X$ of $K_{X}+L$ on $X$ with 
			\begin{equation}\label{eq:SMT-estmate}
				\int_{X} i^{n^2} F_X \wedge \overline{F_X} \, e^{-\phi_1} \leq	||F_Y||^2. 
		\end{equation}}
	
	\begin{proof} 		The proof is similar to the proof of Theorem \ref{th:Ohsawa1}. Since $L^2$-holomorphic sections extend over divisors, one may assume that $X$ is Stein. Then $X_{\mathbb C}$ is Stein. Fix a volume form $dV$ on $Y$, every smooth section 
		$f\in C_c^{\infty}(Y, -(K_X+L)|_Y)$ with compact support defines a linear functional
		$$
		\xi_f: g\mapsto \int_Y f\otimes g\, dV, \ \ \ g\in  C^{\infty}(Y, (K_X+L)|_Y),
		$$
		satisfying that for every $g\in H^0(X_{\mathbb C}, K_{X_{\mathbb C}/\mathbb C}+L_{\mathbb C})$
		$$
		s\mapsto \xi_f(g_s), \ \ g_s:= g|_{Y\times\{s\}}\in H^0(Y, (K_X+L)|_Y), 
		$$
		is holomorphic. Thus $\xi_f$ defines a holomorphic section of the dual of the direct image bundle (see \cite[Definition 2.14]{Wang17}) and Berndtsson--P\u aun's positivity of the direct image bundle implies that (see \cite[section 2]{BP08})
		$$
		s\mapsto \log(||\xi_f||^2_s)
		$$
		is subharmonic on $\mathbb C$, where
		$$
		||\xi_f||^2_s:=\sup_{F\in H^0(X_s, K_{X_s}+L_s)} \frac{|\xi_f(F|_{Y})|^2}{||F||_s^2}, \ \ ||F||_s^2:=\int_{X_s} i^{n^2} F \wedge \overline{F} \, e^{-\phi_s}.
		$$
		Since $\phi_{\mathbb C}$ is $S^1$-invariant, we know that $||\xi_f||^2_s$ is a convex increasing function of $\log|s|$. Denote by $||F_Y||_s$ the minimal $||\cdot||_s$-norm for the space of holomorphic extensions of $F_Y$ to $X_s$, then \cite[proof of Theorem 3.1]{BL16} implies that
		$$
		||F_Y||_s^2=\sup_{f\in C_c^{\infty}(Y, -(K_X+L)|_Y)} \frac{|\xi_f(F_Y)|^2}{||\xi_f||_s^2}.
		$$
		In particular, $||F_Y||_s^2$ is decreasing in $|s|$ and the Second Main Theorem follows.	
	\end{proof}

\section{Symmetrization of a pseudoeffective line bundle}

\subsection{Ross-Witt Nystr\"om correspondence between maximal test curves and geodesic rays}

\begin{definition}\label{de:maxtc}  Let $(L, e^{-\phi})$ be a pseudoeffective line bundle over a complex manifold $X$. We call $\psi_\alpha\in {\rm TC}_\phi$ a maximal test curve  if $P[\psi_\alpha]=\psi_\alpha$ for every $\alpha\in\mathbb R$,
where  the maximal envelope $P[\psi_\alpha]$ of $\psi_\alpha$ is defined by ($P[-\infty]:=-\infty$)
$$
P[\psi_\alpha]:=\sup^\star \left\lbrace \varphi\in {\rm PSH}(L): \varphi\leq \phi, \ \varphi\leq \psi_\alpha+O(1) \ \text{on} \ X \right\rbrace.
$$
--- Here $\star$ denotes the usc regularization and $O(1)$ means that there exists a constant $C\geq 0$ such that $\varphi \leq \psi_\alpha+C$ on $X$.  The space of  maximal test curves is denoted by ${\rm C}_\phi$ .
\end{definition}

\noindent
\textbf{Remark.} \emph{The main example is the Hele-Shaw type maximal test curve \eqref{eq:valphaY-k} (see \cite[section 5]{RW18} for the background).
In general,  \cite[Theorem 3.14]{DDNL25} implies that $P[\psi_\alpha]$ is maximal (i.e. $P[P[\psi_\alpha]]=P[\psi_\alpha]$) for all $\psi_\alpha\in {\rm TC}_\phi$ in case $X$ is compact K\"ahler and $i\partial\dbar \phi \geq \omega$ for some K\"ahler form $\omega$ on $X$. Thus it is fairly easy to construct maximal test curves.}

\medskip

The corresponding maximal subgeodesic ray is usually called a \emph{geodesic ray}. The following is a definition for possibly non-compact $X$ from \cite{Tes24} (see \cite{RW14} for the compact case).

\begin{definition}\label{de:max}  Let $(L, e^{-\phi})$ be a pseudoeffective line bundle over a complex manifold $X$. We call $\phi_t\in {\rm SR}_\phi$ a geodesic ray if $P[\Phi]=\Phi$, where
	$$
	\Phi(z, \tau):=\phi_{{\rm Re}\, \tau} (z)
	$$
and the maximal envelope $P[\Phi]$ is defined by 
$$
P[\Phi]:=\sup^\star\left\lbrace \varphi_{{\rm Re}\, \tau} (z): \varphi_t\in {\rm SR}_\phi,  \lambda_{\varphi_t}\leq \lambda_{\phi_t}, \ \varphi_{{\rm Re}\, \tau} (z) \leq \Phi(z, \tau) +O(1) \ \text{on} \ X\times \mathbb H  \right\rbrace.
$$
The space of geodesic rays is denoted by $ {\rm R}_\phi$.
\end{definition}

The main theorem in the Ross-Witt Nystr\"om correspondence theory is the following result (see \cite[Theorem 1.1]{RW14}, \cite[Theorem 2.9]{Dar17} and \cite[Theorem 1.2]{DDNL18} for the proof of the compact case, the non-compact case can be found in \cite{Tes24}):

\begin{theorem}\label{th:RW} Let $(L, e^{-\phi})$ be a pseudoeffective line bundle over a complex manifold $X$. Then the partial Legendre transforms $\psi_\alpha\mapsto \hat \psi_t$ and $\phi_t\mapsto \check \phi_\alpha$ in \eqref{eq:plt} give two bijections:
$$
{\rm TC}_\phi 	\to {\rm SR}_\phi, \ \ {\rm C}_\phi \to {\rm R}_\phi,
$$
with 
$$
\lambda_{\psi_\alpha} =\lambda_{\hat \psi_t}, \ \lambda_{\phi_t}=\lambda_{\check \phi_\alpha}, \ \check{\hat \psi}_\alpha =\psi_\alpha,\ \text{and} \ \hat{\check \phi}_t=\phi_t.
$$
Sometimes we write
$
\mathcal L(\psi_\alpha) =\hat \psi_t, \  \mathcal L^{-1}(\phi_t) =\check \phi_\alpha.
$
\end{theorem}

\subsection{Homogeneous parts of $S^1$-invariant psh functions}

Let $\phi$ be $S^1$-invariant psh on $\mathbb C^n$, i.e. $\phi(e^{i\theta}z)=\phi(z)$ for all $\theta\in\mathbb R$, $z\in \mathbb C^n$. By \cite[Lemma 2.21]{WN19}, we have $
	\phi=\sup_{\alpha\in \mathbb R} \rho_\alpha,
$
where
		\begin{equation}\label{eq:rho-alpha0}
\rho_\alpha(z):=\inf_{\tau \in \mathbb C^*} \{\phi(\tau z)-\alpha \log |\tau^2|\}.
\end{equation}
Note that each $\rho_\alpha$ is $\alpha$-$\log$ homogeneous in the sense that 
$$
\rho_\alpha(s\xi)=\alpha\log|s^2| +\rho_\alpha(\xi).
$$ 

\begin{definition}\label{de:rho-alpha}
We call $\rho_\alpha$ the $\alpha$-$\log$ homogeneous part of $\phi$.
\end{definition}

From \eqref{eq:rho-alpha0}, we know that $\rho_\alpha$ is concave in $\alpha$, the Kiselman minimum principle \cite{Ki78} further implies that each $\rho_\alpha$ is psh on $\mathbb C^n$. The proof of \cite[Lemma 2.21]{WN19} also gives the following:

\begin{proposition}\label{pr:rho-alpha} Let $\phi$ be $S^1$-invariant psh on $\mathbb C^n$ with $\alpha$-$\log$ homogeneous part $\rho_\alpha$. 
Put
	\begin{equation}\label{eq:rho-alpha1}
	\psi_\alpha:=\sup\{\psi\in {\rm PSH}(\mathbb C^n): \psi\leq \phi, \ \nu_0(\psi) \geq \alpha\},
\end{equation}
where $\nu_0(\psi):=\liminf_{z\to 0} \frac{\psi(z)}{\log|z^2|}$ denotes the Lelong number of $\psi$ at $0$, then 
		\begin{equation}\label{eq:rho-alpha2}
\rho_\alpha(z)=\limsup_{s\to 0} \ \{\psi_\alpha (sz)-\alpha \log|s|^2\} , \ \ \psi_\alpha=\sup_{\beta\geq\alpha} \rho_\beta.
\end{equation}
\end{proposition} 

\noindent
\textbf{Remark.} \emph{If $\phi$ is not $S^1$-invariant then in general $\rho_\alpha$ defined in \eqref{eq:rho-alpha0} is not psh. But the envelope $\psi_\alpha$ in  \eqref{eq:rho-alpha1} is always psh, thus \eqref{eq:rho-alpha2} suggests that one may use the envelope in \eqref{eq:rho-alpha1} to define the $\alpha$-$\log$ homogeneous part of a non-$S^1$-invariant psh function. In the next section we shall generalize this idea to psh metric potentials of a pseudoeffective line bundle.}
		
\subsection{$S^1$-symmetrization of psh weights} Let $L$ be a holomorphic line bundle over a complex manifold $X$. Recall that the Lelong number of $\psi\in {\rm PSH}(L)$ at $y\in X$ is defined by
		\begin{equation}\label{eq:lelong-new}
			\nu_y(\psi):=\liminf_{x\to y} \frac{\psi(x)}{\log |z(x)^2|},
		\end{equation} 
		where $z$ is an arbitrary holomorphic coordinate chart around $y$ such that $z(y)=0$ (we know that $\nu_y(\psi)$ does not depend on the choice of local frame and local coordinate $z$). The (generic) Lelong number along a closed codimension $k$ complex submanifold $Y$ is defined by
		\begin{equation}\label{eq:lelong-new1}
			\nu_Y(\psi):=\inf_{y\in Y} \nu_y(\psi).
		\end{equation} 
		
		\begin{definition}\label{de:pseudo-thre-L} We call 
			$$
			\lambda_Y(L):=\sup_{\text{$\psi$ is a psh metric potential of $L$}}\nu_Y(\psi)  \ \text{(later we simply write $\lambda_Y(L)$ as $\lambda_Y$)},
			$$
			the pseudoeffective threshold of $L$ along $Y$ (it can be infinite when $X$ is non-compact).
		\end{definition} 
	Similar to  \eqref{eq:rho-alpha1}, for $\phi\in{\rm PSH}(L)$ we look at the maximal envelope
		\begin{equation}\label{eq:valphaY-k}
			\psi_\alpha:=\sup\left\lbrace \psi\in {\rm PSH}(L):  \psi\leq \phi \, \text{and $\nu_{Y}(\psi)\geq \alpha$}\right\rbrace
		\end{equation}
	associated to $Y$ ($\psi_\alpha=\phi$ for $\alpha\leq 0$). To define $\rho_\alpha$ in this general case, we have to pull back $\psi_\alpha$ along the deformation $\mu$ (see \eqref{eq:mu-new}) to the normal bundle. Recall that $\mu$ is defined by $(\xi, w, s)\mapsto (s\xi, w, s)$ around $Y\times\{0\}$, so similar to \eqref{eq:rho-alpha2}, we define
		\begin{equation}\label{eq:rho-alphaY}
		\rho_\alpha(\xi, w):= \limsup_{s\to 0}\ \{\psi_\alpha(s\xi, w)-\alpha\log|s^2|\}.
	\end{equation}
Notice that each $\rho_\alpha$ is psh and $\alpha$-$\log$ homogeneous on $NY$ in the sense that
$$
\rho_\alpha(s\xi, w)= \rho_\alpha(\xi, w) +\alpha\log |s^2|,
$$
moreover, $\rho_0(\xi, w)= \phi(0,w)$  and $\rho_\alpha\equiv -\infty$ for $\alpha<0$.

\begin{definition}\label{de:S1-symm} We call $\rho_\alpha$ the $\alpha$-$\log$ homogeneous part of $\phi\in {\rm PSH}(L)$ and the function
	\begin{equation}\label{eq:S1-symm}
		\phi_Y:=\sup_{\alpha\in\mathbb R} \rho_\alpha
	\end{equation} 
the $S^1$-symmetrization of $\phi$ on $NY$. For $\lambda>0$, we call
\begin{equation}\label{eq:S1-symm-lambda}
	\phi_{Y,\lambda}:=\sup_{\alpha\leq \lambda}\, \rho_\alpha,
\end{equation} 
the level-$\lambda$ $S^1$-symmetrization of $\phi$ on $NY$.
\end{definition} 

\noindent
\textbf{Remark 1.} \emph{Similar to Proposition \ref{pr:rho-alpha}, we know that if $X=NY$ and $\phi$ itself is $S^1$-invariant with respect to the  action $a\cdot(\xi, w)=(a\xi, w)$ along the fiber of $NY$ then $\phi=\phi_Y$.}

\medskip
\noindent
\textbf{Remark 2.} \emph{It is clear that $\phi_Y=\lim_{\lambda\to\infty} \phi_{Y,\lambda}$. We shall prove that each $\phi_{Y, \lambda}$ is $S^1$-invariant psh on $NY$ (see Proposition \ref{pr:phi-C-lambda} below). By the definition of the pseudoeffective threshold $\lambda_Y$, we know that $\psi_{\alpha} \equiv -\infty$ for all $\alpha>\lambda_Y$, thus $\phi_{Y,\lambda}=\phi_Y$ for all $\lambda\geq\lambda_Y$ in case $\lambda_Y<\infty$. In case $\phi$ is smooth, $i\partial\dbar\phi>0$, $X$ is compact and $Y=\{x\}$ is a single point (so that $\lambda_Y<\infty$ and $\phi_Y(\xi,w)$ reduces to $\phi_x(\xi)$), Witt-Nystr\"om \cite[section 8]{WN18} proved  that
	\begin{equation}\label{eq:G-new12}
		\phi_x(\xi)=\sup^\star \left\lbrace \frac2m \log |\lim_{s\to 0}  \frac{f(s\xi)}{s^{{\rm ord}_x f}}|: F \in H^0(X, mL),  \text{$\sup_X |F|^2 e^{-m\phi} \leq 1$, $m\in\mathbb N$}\right\rbrace,
	\end{equation} 
	where $f$ is defined by $F=f \bm e^{\otimes m}$ around $x$ and the holomorphic frame $\bm e$ of $L$ is chosen such that $|\bm e|^2(x)=e^{-\phi(x)}$. $\phi_x+O(1)$ is also called the canonical growth condition at $x$ of $L$ in \cite{WN18}.}

\subsection{Deformation of a psh weight to its $S^1$-symmetrization} In this section we shall follow \cite{WN21} to construct a psh weight $\phi_{\mathbb C,\lambda}$ (see \eqref{eq:G-new01} below) on $L_{\mathbb C}$ (see \eqref{eq:LC-def-new})  connecting $\phi$ and its level-$\lambda$ $S^1$-symmetrization $\phi_{Y, \lambda}.$ The main idea is to prove the following "Legendre duality" relation between the maximal test curve (see Definition \ref{de:maxtc})
\begin{equation}\label{eq:psi-lambda-test curve}
\psi^{\leq \lambda}_\alpha:=
\begin{cases}
	\psi_\alpha & \alpha\leq \lambda;\\
	-\infty & \alpha>\lambda,
\end{cases}
\end{equation}
and the total Green function
		\begin{equation}\label{eq:G-lambdanew}
			G_\lambda:=\sup\left\lbrace \sigma\in {\rm PSH}(L\times\mathbb D): \sigma\leq \phi, \ \nu_{Y\times\{0\}}(\sigma) \geq \lambda \right\rbrace,
		\end{equation}
		where $\mathbb D$ denotes the unit disk and we think of $\phi$ as a psh metric potential for $L\times \mathbb D$.
		
		\begin{proposition}\label{pr:G-lambdanew} We have
			\begin{equation}\label{eq:G-lambdanew1}
				G_\lambda(z,s)=\lambda\log |s^2| +\sup_{\alpha\leq \lambda} \left\lbrace \psi_\alpha(z)-\alpha\log |s^2| \right\rbrace =\lambda\log |s^2|+\hat\psi^{\leq\lambda}_{-\log|s^2|}(z)
			\end{equation}
			on $X\times \mathbb D$, in particular, $G_\lambda-\lambda\log |s^2|$ is increasing in $\lambda$. Moreover, if we extend $G_\lambda$ to 
			\begin{equation}\label{eq:G-lambdanew2}
				\tilde G_\lambda(z,s):=
				\begin{cases}
					G_\lambda(z,s) &  |s|<1, \\
					\phi(z)+\lambda \log|s^2| & |s|\geq 1,
				\end{cases}
			\end{equation}
			and extend $\phi$ to 
			\begin{equation}\label{eq:phi-tilde}
				\tilde\phi_\lambda(z,s):=\phi(z)+\lambda \max\{0, \log|s^2|\},
			\end{equation}
			then 
			\begin{equation}\label{eq:G-new}
				\tilde G_\lambda:=\sup\left\lbrace \sigma\in {\rm PSH}(L\times\mathbb C): \sigma\leq \tilde\phi_\lambda, \ \nu_{Y\times\{0\}}(\sigma) \geq \lambda \right\rbrace.
			\end{equation} 
		\end{proposition}
		
		\begin{proof} We follow the proof in \cite[Theorem 6.3]{RW18}. Since $\psi_\alpha\in {\rm PSH}(L)$, we know that each $\lambda\log |s^2| + \psi_\alpha(z)-\alpha\log |s^2|$ is a candidate in the definition of $G_\lambda$, thus
			$$
			G_\lambda(z,s)\geq \lambda\log |s^2| +\sup_{\alpha\leq \lambda} \left\lbrace \psi_\alpha(z)-\alpha\log |s^2| \right\rbrace.
			$$
			On the other hand, by  comparing with the local envelope, we have $$G_\lambda \leq 2\lambda \log d$$ around $Y\times\{0\}$, where $d$ denotes the distance to $Y\times\{0\}$ with respect to a sufficiently large Riemannian metric on $X$. Then by a direct computation, we know that for every $\alpha\leq \lambda$,  $\inf_{t>0} \{G_\lambda(z, e^{-t/2}) +(\lambda-\alpha) t\} $ is a candidate in the definition of $\psi_\alpha$, thus
			$$
			\inf_{t>0} \{G_\lambda(z, e^{-t/2}) +(\lambda-\alpha) t\} \leq \psi_\alpha.
			$$
			Hence the Legendre duality gives
			$$
			G_\lambda(z,s)\leq \lambda\log |s^2| +\sup_{\alpha\leq \lambda} \left\lbrace \psi_\alpha(z)-\alpha\log |s^2| \right\rbrace.
			$$
			A similar argument gives \eqref{eq:G-new}.
		\end{proof}
		
		\noindent
		\textbf{Remark.} \emph{Note that the Lelong number of $\mu^* (\tilde G_\lambda)$ along the zero fiber $X_0=NY$ is no less than (in fact, equal to) $\lambda$ and the preimage of $Y\times\{0\}$ under the pull back of  $\mu(\xi, w,s)=(s\xi, w,s)$ is the divisor $\{s=0\}$, thus one may write
			\begin{equation}\label{eq:G-new0}
				\mu^* \tilde G_\lambda =\phi_{\mathbb C, \lambda}+\lambda\log|s^2|,
			\end{equation}
			where $\phi_{\mathbb C, \lambda}$ is  $S^1$-invariant psh  on $X_{\mathbb C}$. The above proposition implies that
			\begin{equation}\label{eq:G-new01}
				\phi_{\mathbb C, \lambda} (\xi,w,s):= 
				\begin{cases}
					\limsup_{s\to 0} \hat\psi^{\leq\lambda}_{-\log|s^2|}(s\xi, w) & s=0; \\
					\hat\psi^{\leq\lambda}_{-\log|s^2|}(s\xi, w) & 0<|s| < 1; \\
					\phi(s\xi, w) &  |s|\geq 1.	
				\end{cases}
			\end{equation}}
		
		\begin{definition}\label{de:cwkd} Following \cite{WN21} we call $(L_{\mathbb C}, e^{-\phi_{\mathbb C, \lambda}})$  (for $\phi_{\mathbb C,\lambda}$ defined in \eqref{eq:G-new0}) the level-$\lambda$ canonical weak K\"ahler deformation of $(L,e^{-\phi})$ to the normal bundle $NY$.  If $\lambda_Y<\infty$ then $\phi_{\mathbb C, \lambda}$ does not depend on $\lambda\geq \lambda_Y$ and we simply write $\phi_{\mathbb C, \lambda_Y}$ as $\phi_{\mathbb C}$. In that case, we call $(L_{\mathbb C}, e^{-\phi_{\mathbb C}})$ the  canonical weak K\"ahler deformation of $(L,e^{-\phi})$ to the normal bundle $NY$.
		\end{definition}

	It is clear that $\phi_{\mathbb C, \lambda}$ restricts to $\phi$ at the $s=1$ fiber. The proposition below tells that $ \phi_{\mathbb C, \lambda}$ restricts to $\phi_{Y, \lambda}$ at the $s=0$ fiber.

	\begin{proposition}\label{pr:phi-C-lambda} $\phi_{Y, \lambda}=\phi_{\mathbb C,\lambda}|_{NY}$, in particular $\phi_{Y, \lambda}$ is $S^1$-invariant psh on $NY$.		
	\end{proposition}

\begin{proof} Recall that $$ \hat\psi^{\leq\lambda}_{-\log|s^2|}(s\xi, w) = \sup_{\alpha\in\mathbb R} \left\lbrace \psi^{\leq\lambda}_\alpha(s\xi, w)-\alpha\log |s^2| \right\rbrace = \sup_{\alpha\leq \lambda} \left\lbrace \psi_\alpha(s\xi, w)-\alpha\log |s^2| \right\rbrace,$$
	thus \eqref{eq:G-new01} and \eqref{eq:rho-alphaY} imply that
	$$
	\phi_{\mathbb C, \lambda} (\xi,w,0)=\limsup_{s\to 0}  \sup_{\alpha\leq \lambda} \left\lbrace \psi_\alpha(s\xi, w)-\alpha\log |s^2| \right\rbrace \geq \sup_{\alpha\leq \lambda} \rho_\alpha(\xi,w)= \phi_{Y,\lambda}(\xi, w),
	$$
	which gives $\phi_{\mathbb C,\lambda}|_{NY} \geq \phi_{Y, \lambda}$. Now it suffices to prove $\phi_{\mathbb C,\lambda}|_{NY} \leq \phi_{Y, \lambda}$. Since $\phi_{\mathbb C,\lambda}$ is $S^1$-invariant psh with respect to the $\mathbb C^*$-action \eqref{eq:Cstaraction-new}, the proof of \cite[Lemma 2.21]{WN19} implies that $\phi_{\mathbb C,\lambda} =\sup_{\alpha\in\mathbb R} \sigma_\alpha$ is the supremum over $\alpha$ of its $\alpha$-$\log$  homogeneous part
	$$
\sigma_\alpha(\xi, w, s):=	\inf_{\tau\in \mathbb C^*} \{\phi_{\mathbb C, \lambda} (\tau\xi,w,\tau^{-1}s) -\alpha\log |\tau|^2\}.
	$$
	Thus it suffices to show that 
			\begin{equation}\label{eq:sigma-alpha-key}
\sigma_\alpha(\xi, w, s) \leq \psi_\alpha(s\xi, w)-\alpha\log |s^2|, \ \forall \ \alpha\in \mathbb R, \ \ 0<|s|<1.
\end{equation}
By \eqref{eq:G-new0}, for $0<|s|<1$, we have $ \tilde G_\lambda= G_\lambda$ and
	$$
	\sigma_\alpha(\xi, w, s)= \inf_{\tau\in \mathbb C^*} \left\lbrace  G_\lambda(s\xi, w, \tau^{-1}s)-\lambda\log |\tau^{-1}s|^2  -\alpha\log |\tau|^2\right\rbrace.
	$$
Thus \eqref{eq:sigma-alpha-key} is equivalent to that
$$
\inf_{\tau\in \mathbb C^*} \left\lbrace  G_\lambda(s\xi, w, \tau^{-1}) +(\lambda-\alpha)\log |\tau|^2 \right\rbrace  \leq \psi_\alpha(s\xi,w),
$$
which follows directly from the fact that $\inf_{t>0} \{G_\lambda(\cdot, e^{-t/2}) +(\lambda-\alpha) t\} $ is a candidate in the envelope definition \eqref{eq:valphaY-k} of $\psi_\alpha$. The proof is now complete.
\end{proof}


\subsection{$S^1$-symmetrization version of the Ohsawa-Takegoshi extension theorem} Applying our Second Main Theorem to each $\phi_{\mathbb C,\lambda}$ and letting $\lambda\to \infty$, we obtain the following $S^1$-symmetrization version of the Ohsawa-Takegoshi extension theorem.

\begin{theorem}\label{th:oldtheoremC} Let $(L, e^{-\phi})$ be a pseudoeffective line bundle over an $n$-dimensional complex manifold $X$. Let $Y$ be a closed complex submanifold. Assume that $X$ is $Y$-Stein (see the Second Main Theorem). Let $F_Y\in H^0(Y, (K_X+L)|_Y)$, if 
			\begin{equation}\label{eq:ThC-new}
				||F_Y||^2:=\lim_{\lambda\to\infty} \int_{NY} i^{n^2} F_{Y,0} \wedge \overline{F_{Y,0}} \, e^{-\phi_{Y,\lambda}} <\infty,  
			\end{equation}
then
			\begin{equation}\label{eq:ThC-new1}
||F_Y||^2 = \int_{NY} i^{n^2} F_{Y,0} \wedge \overline{F_{Y,0}} \, e^{-\phi_{Y}}
\end{equation} 
and $F_Y$			extends to a holomorphic section $F_X$ of $K_{X}+L$ on $X$ with 
			\begin{equation}\label{eq:coB}
				\int_{X} i^{n^2} F_X \wedge \overline{F_X} \, e^{-\phi} \leq	||F_Y||^2.
			\end{equation}
			Moreover, equality holds in \eqref{eq:coB} in case $NY$ is a dense Zariski open set of $X$ and $\phi$ is  $S^1$-invariant with respect to the $S^1$-action $a\cdot(\xi, w)=(a\xi, w)$ on $NY$.
\end{theorem}

\begin{proof} For each $\lambda>0$, one may apply the Second Main Theorem to  $\phi_{\mathbb C,\lambda}$ and  obtain a holomorphic section $F_{\lambda}$ of $K_X+L$ on $X$ with  
	$$
		\int_{X} i^{n^2} F_{\lambda} \wedge \overline{F_{\lambda}} \, e^{-\phi} \leq	\int_{NY} i^{n^2} F_{Y,0} \wedge \overline{F_{Y,0}} \, e^{-\phi_{\mathbb C, \lambda}|_{NY}}.
	$$ 
By Proposition \ref{pr:phi-C-lambda}, we have $\phi_{\mathbb C, \lambda}|_{NY}=\phi_{Y, \lambda}$, thus it suffices to take $F_X$ as a weak limit of $F_\lambda$. \eqref{eq:ThC-new1} follows from the fact that $\phi_Y$ is the $\lambda\to\infty$ increasing limit of $\phi_{Y, \lambda}$.
In case $NY$ is a dense Zariski open set of $X$ and $\phi$ is  $S^1$-invariant, we know that $\phi=\phi_Y$ on $NY$ by Remark 1 after Definition \ref{de:S1-symm}. The Taylor expansion of $F_X$ along the fiber of $NY \to Y$ gives
$$
\int_{X} i^{n^2} F_X \wedge \overline{F_X} \, e^{-\phi} \geq \int_{NY} i^{n^2} F_{Y,0} \wedge \overline{F_{Y,0}} \, e^{-\phi_Y} =||F_Y||^2,
$$   
hence equality holds in \eqref{eq:coB}.
\end{proof}
		
\noindent
\textbf{Remark.} \emph{In case $\lambda_Y<\infty$ (for example when $X$ is compact), we have
$\phi_Y=\phi_{Y,\lambda}$ for all $\lambda\geq \lambda_Y$ and \eqref{eq:ThC-new} is equivalent to that $\int_{NY} i^{n^2} F_{Y,0} \wedge \overline{F_{Y,0}} \, e^{-\phi_{Y}}<\infty$. Assume further that $Y=\{x\}$ is a single point, choose coordinates $\xi, z$ such that the deformation mapping $\mu$ in Definition \ref{de:dtn} is given by
	\begin{equation}\label{eq:oldtheoremC-xi-z}
\mu(\xi,s)=(z,s)=(s\xi, s)
\end{equation}
around $z(x)=0$, then the above theorem gives:}

\begin{theorem}\label{th:oldtheoremC-x} Let $(L, e^{-\phi})$ be a pseudoeffective line bundle over a projective manifold $X$. Fix $x\in X$ and a local frame $\bm e$ of $L$ around $x$ such that $|\bm e|^2(x)=e^{-\phi(x)}$. Let us choose coordinates $\xi, z$ satisfying \eqref{eq:oldtheoremC-xi-z} and write the associated $S^1$-symmetrization $\phi_x$ of $\phi$ on $T_xX$ as a function $\phi_x(\xi)$ in $\xi\in\mathbb C^n$. Assume that for some $m\geq 1$,
	\begin{equation}\label{eq:oldtheoremC-x1}
		I:=\int_{\mathbb C^n} i^{n^2} d\xi\wedge \overline{d\xi}\, e^{-m\phi_x(\xi)} <\infty,  
	\end{equation}
	then there exists $F\in H^0(X, K_X+mL)$ with $F(x)=dz\otimes\bm e^{\otimes m}$ and
	\begin{equation}\label{eq:oldtheoremC-x2}
			\int_{X} i^{n^2} F \wedge \overline{F} \, e^{-m\phi} \leq	I.
	\end{equation} 
\end{theorem}

\noindent
\textbf{Remark.} \emph{The above theorem gives the following lower bound
	\begin{equation}\label{eq:oldtheoremC-x3}
	K_{m\phi}(x) \geq \frac{i^{n^2} dz \wedge \overline{dz}\, e^{-m\phi(x)}}{\int_{\mathbb C^n} i^{n^2} d\xi\wedge \overline{d\xi}\, e^{-m\phi_x(\xi)}}
\end{equation} 
for the Bergman kernel form, in the next subsection we shall show that the asymptotic expansion of right hand side is equal to the left hand side up to the sub-leading term, which leads us to prove Theorem A using the above theorem.}

\subsection{Bergman kernel asymptotic is symmetric up to the sub-leading term}

The following theorem tells us that the $S^1$-symmetrization (or the canonical growth condition construction) preserves the first two coefficients in the Bergman asymptotic. 

\begin{theorem}\label{th:BKA-sub} Let $(L, e^{-\phi})$ be a positive line bundle over a compact complex manifold. Put $c_1=\omega/(2\pi)$, $\omega:=i\partial\dbar\phi$ and denote by 
	$$
	S:=\frac{n (i\dbar\partial \log c_1^n) \wedge (i\partial\dbar\phi)^{n-1}}{ (i\partial\dbar\phi)^n}
	$$
the scalar curvature of $(X, \omega)$. Then
	\begin{equation}\label{eq:BKA-sub}
\frac{i^{n^2} dz \wedge \overline{dz}\, e^{-m\phi(x)}}{\int_{\mathbb C^n} i^{n^2} d\xi\wedge \overline{d\xi}\, e^{-m\phi_x(\xi)}}=\frac{c_1(x)^n}{n!} \left(m^n - \frac{S(x)}2 m^{n-1} + o(m^{n-1}) \right), \  x\in X,
\end{equation}
for large $m$, where $\xi, z, \phi_x$ are the same as in Theorem \ref{th:oldtheoremC-x}.
\end{theorem}

\begin{proof} Observe that the left hand side of \eqref{eq:BKA-sub} does not depend on the choice of $z$ and local frame $\bm e$ of $L$ around $x$. Let us choose Bochner coordinate \cite{Boch47} $z$ and $\bm e$ such that
	\begin{equation}\label{eq:boch-phi}
\phi(z)=|z|^2+\sum_{j, k, l, m =1}^n c_{j\bar kl\bar m} z_j\bar z_k z_l \bar z_m + O(|z|^5).
\end{equation}
Then Lemma \ref{le:boch} below implies that 
$$
\phi_x(\xi) \geq \phi_1(\xi):=|\xi|^2+\sum_{j, k, l, m =1}^n c_{j\bar kl\bar m} \xi_j\bar \xi_k \xi_l \bar \xi_m -C|\xi|^5
$$
on a neighborhood of the origin. The standard Bergman asymptotic theory (see \cite{BBS08}) gives
$$
\frac{1}{\int_{\mathbb C^n} i^{n^2} d\xi\wedge \overline{d\xi}\, e^{-m\phi_x(\xi)}} \geq \frac1{(2\pi)^n} \left(m^n - \frac{S(x)}2 m^{n-1} + o(m^{n-1}) \right).
$$
Thus \eqref{eq:BKA-sub} follows from \eqref{eq:oldtheoremC-x3} and the Bergman asymptotic formula of $K_{m\phi}$. 
 \end{proof}

\begin{lemma}\label{le:boch} There exist $\varepsilon, C>0$ such that
	\begin{equation}\label{eq:boch-phi-x}
	\phi_x(\xi) \geq |\xi|^2+\sum_{j, k, l, m =1}^n c_{j\bar kl\bar m} \xi_j\bar \xi_k \xi_l \bar \xi_m -C|\xi|^5
\end{equation}
on $|\xi|<\varepsilon$.
\end{lemma}

\begin{proof} By \eqref{eq:boch-phi}, there exist $\varepsilon_1, C>0$ such that
	\begin{equation}\label{eq:boch-phi-1}
	\phi(z) \geq \phi_1(z):=|z|^2+\sum_{j, k, l, m =1}^n c_{j\bar kl\bar m} z_j\bar z_k z_l \bar z_m -C |z|^5
\end{equation}	
on $|z|<\varepsilon_1$ and $\phi-\phi_1$ is psh on $|z|<\varepsilon_1$.  Put 
	\begin{equation}\label{eq:boch-phi-2}
\phi_2(z):=\log (1+|z|^2/2+C_1 |z|^4)
\end{equation}	
and choose large constant $C_1>0$ such that
$$
\phi_3:=\max\{\phi_1, \phi_2\} =\phi_2
$$
near $|z|=\varepsilon_1$. Then we know that $\phi_3$ extends to an $S^1$-invariant psh function on $\mathbb C^n$ such that $\phi_3=\phi_2$ on $|z|\geq \varepsilon_1$ and $\phi_3=\phi_1$ on $|z|<\varepsilon_2<\varepsilon_1$ for some sufficiently small $\varepsilon_2>0$. Put
$$
\psi_{\alpha}:=\sup\{\psi\in {\rm PSH}(L): \psi\leq \phi, \ \nu_0(\psi) \geq \alpha\}
, \ \ \rho_\alpha(\xi):= \limsup_{s\to 0}\ \{\psi_\alpha(s\xi)-\alpha\log|s^2|\}
$$
and similarly define
$$
\psi_{3,\alpha}:=\sup\{\psi\in {\rm PSH}(\mathbb C^n): \psi\leq \phi_3, \ \nu_0(\psi) \geq \alpha\}
$$
and
$$
\rho_{3,\alpha}(\xi):= \limsup_{s\to 0}\ \{\psi_{3,\alpha}(s\xi)-\alpha\log|s^2|\}.
$$
We know that there exists $\alpha_3>0$ such that
$$
\overline{\{\psi_{3,\alpha} < \phi_3\}} \subset \{|z|<\varepsilon_2\},
$$
for every $\alpha\leq \alpha_3$. Thus we have
$$
\psi_\alpha\geq \phi-\phi_1+\psi_{3,\alpha}
$$
on $|z|<\varepsilon_2$, which implies that
$$
\rho_\alpha \geq \rho_{3,\alpha}.
$$
for $\alpha\leq \alpha_3$. Take $0<\alpha_4<\alpha_3$ and $\varepsilon_2>\varepsilon>0$ such that
$$
\rho_{3,\alpha} \geq \rho_{3,\beta}
$$
on $|\xi|<\varepsilon$ for all  $\alpha \leq \alpha_4$ and all $\beta\geq \alpha_3$ (we can do this because $\rho_{3,\alpha}$ has isolated $\alpha\log$-pole at $\xi=0$ for small positive $\alpha$ and $\rho_{3,\beta}=-\infty$ for all $\beta\geq 2$ by \eqref{eq:boch-phi-2}). Since $\phi_3$ is $S^1$-invariant psh, Proposition \ref{pr:rho-alpha} implies that
$$
\phi_3=\sup_{\alpha\in\mathbb R} \rho_{3,\alpha}.
$$
Thus we have
$$
\phi_x=\sup_{\alpha\in\mathbb R} \rho_\alpha \geq \phi_3 = \phi_1
$$
on $|\xi|<\varepsilon$. The proof is complete.
\end{proof}

\subsection{$(S^1)^k$-symmetrization with respect to a complete flag and Chebyshev transform}

Let $(L, e^{-\phi})$ be a pseudoeffective line bundle over an $n$-dimensional complex manifold $X$. Given a closed submanifold $Y$ of $X$ and $\lambda>0$, we have defined the maximal test curve
\begin{equation}\label{eq:psi-lambda-test curve-new}
	\psi^{\leq \lambda}_\alpha:=
	\begin{cases}
		\psi_\alpha & \alpha\leq \lambda;\\
		-\infty & \alpha>\lambda,
	\end{cases}
\end{equation}
using the maximal envelopes
		\begin{equation}\label{eq:valphaY-k-new}
	\psi_\alpha:=\sup\left\lbrace \psi\in {\rm PSH}(L):  \psi\leq \phi \, \text{and $\nu_{Y}(\psi)\geq \alpha$}\right\rbrace.
\end{equation}
Denote by $\mu:X_{\mathbb C} \to X\times\mathbb C$ the deformation of $X$ to $NY$ and choose coordinate $z, \xi, w$ such that $\mu$ is defined by
$$
\mu(\xi, w ,s) =(z,s)=(s\xi, w, s)
$$
around $Y\times\{0\}$, then the $\alpha$-$\log$ homogeneous part of $\phi$ is defined by
		\begin{equation}\label{eq:rho-alphaY-new}
	\rho_\alpha(\xi, w):= \limsup_{s\to 0}\ \{\psi_\alpha(s\xi, w)-\alpha\log|s^2|\}
\end{equation}
and we call
\begin{equation}\label{eq:S1-symm-lambda-new}
	\phi_{Y,\lambda}:=\sup_{\alpha\leq \lambda}\, \rho_\alpha,
\end{equation} 
the level-$\lambda$ $S^1$-symmetrization of $\phi$ on $NY$. In this section, we shall generalize this construction and define a natural $(S^1)^n$-symmetrization of a psh weight using a \emph{complete flag $Y_{\bullet}$ of submanifolds}, i.e. each $Y_j$ is a codimension-$j$ closed complex submanifold of $X$ with
$$
X=Y_0\supset Y_1\supset \cdots Y_{n-1}\supset Y_n=\{x\}.
$$
Note that $\phi_{Y_1, \lambda}$ is a psh weight on $\pi_{Y_1}^*L$, where $\pi_{Y_1}: NY_1 \to Y_1$ denotes the bundle projection. To continue the process, one may simply replace $Y$ by 
\begin{equation}\label{eq:Z2}
Z_2:=\pi_{Y_1}^*(Y_2)
\end{equation} 
in \eqref{eq:valphaY-k-new} and define
		\begin{equation}\label{eq:valphaY-k-new-2}
	\psi^2_\alpha:=\sup\left\lbrace \psi\in {\rm PSH}(\pi_{Y_1}^*L):  \psi\leq \phi_{Y_1,\lambda} \, \text{and $\nu_{Z_2}(\psi)\geq \alpha$}\right\rbrace.
\end{equation}
Since $Z_2$ is a closed submanifold of $NY_1$, we denote its normal bundle by $NZ_2$. Choose suitable coordinates such that the deformation of $NY_1$ to $NZ_2$ is defined by
$$
(\xi_1, \xi_2, z_3,\cdots, z_n, s) \mapsto (\xi_1, s\xi_2, z_3,\cdots, z_n, s)
$$
around $Z_2\times\{0\}$, then the $\alpha$-$\log$ homogeneous part of $\phi_{Y_1,\lambda}$ is defined by
\begin{equation}\label{eq:rho-alphaY-new-2}
	\rho_\alpha^2(\xi_1, \xi_2, z_3,\cdots, z_n):= \limsup_{s\to 0}\ \{\psi^2_\alpha(\xi_1, s\xi_2, z_3,\cdots, z_n)-\alpha\log|s^2|\},
\end{equation}
we call
\begin{equation}\label{eq:S1-symm-lambda-new-2}
	\phi^2_{Y_{\bullet},\lambda}:=\sup_{\alpha\leq \lambda}\, \rho^2_\alpha,
\end{equation}
the level-$\lambda$ $(S^1)^2$-symmetrization of $\phi$ on $NZ_2$. In general, we shall inductively define
\begin{equation}\label{eq:rho-alphaY-new-XYZ1}
Z_1:=Y_1, \ \ \pi_{k-1}: NZ_{k-1} \to Z_{k-1}, \ \ Z_k:=\pi_{Y_{k-1}}^{-1}(Y_k), \ \ k\geq 3,
\end{equation}
where $\pi_{k-1}$ denotes the bundle projection,  $NZ_{k-1}$ denotes the normal bundle with respect to the embedding of $Z_{k-1}$ in $ NZ_{k-2}$ ($N Z_0:=X$) and $\pi_{Y_{k-1}}: NZ_{k-1} \to Y_{k-1}$ denotes the composition of $\pi_{k-1}$ with the bundle projection $Z_{k-1}\to Y_{k-1}$. Now we can define 
\begin{equation}\label{eq:valphaY-k-new-k}
	\psi^{k}_\alpha:=\sup\left\lbrace \psi\in {\rm PSH}(\pi_{Y_{k-1}}^*L):  \psi\leq \phi^{k-1}_{Y_{\bullet},\lambda} \, \text{and $\nu_{Z_{k}}(\psi)\geq \alpha$}\right\rbrace, \ \ k\geq 3,
\end{equation}
the  $\alpha$-$\log$ homogeneous part 
$$
	\rho_\alpha^{k}(\xi_1, \cdots, \xi_{k}, z_{k+1},\cdots, z_n):= \limsup_{s\to 0}\ \{\psi^k_\alpha(\xi_1, \cdots, \xi_{k-1}, s\xi_{k}, z_{k+1},\cdots, z_n)-\alpha\log|s^2|\}
$$
 of $\phi^{k-1}_{Y_{\bullet},\lambda}$ and
\begin{equation}\label{eq:S1-symm-lambda-new-k}
	\phi^k_{Y_{\bullet},\lambda}:=\sup_{\alpha\leq \lambda}\, \rho^k_\alpha.
\end{equation}
Put $
\phi^1_{Y_{\bullet},\lambda}:= \phi_{Y_1,\lambda}.
$ We know that each $\phi^k_{Y_{\bullet},\lambda}$ is $(S^1)^k$-invariant in the sense that
$$
	\phi^k_{Y_{\bullet},\lambda} (\xi_1, \cdots, \xi_{k}, z_{k+1},\cdots, z_n)=	\phi^k_{Y_{\bullet},\lambda} (|\xi_1|, \cdots, |\xi_{k}|, z_{k+1},\cdots, z_n).
$$

\begin{definition}\label{de:S1-symm-k} We call $\phi^k_{Y_{\bullet},\lambda}$, $1\leq k\leq n$ the level-$\lambda$ $(S^1)^k$-symmetrization of $\phi$ on $NZ_k$. 
\end{definition} 

Assume that $X$ is compact, $\phi$ is smooth and $i\partial\dbar\phi>0$ on $X$, then we know that $\phi^k_{Y_{\bullet},\lambda}$ does not depend on $\lambda$ when $\lambda$ is sufficiently large. We call
\begin{equation}\label{eq:S1-symm-lambda-new-infty}
	\phi^k_{Y_{\bullet}}:=\lim_{\lambda\to\infty} \phi^k_{Y_{\bullet},\lambda} 
\end{equation}
the $(S^1)^k$-symmetrization of $\phi$ on $NZ_k$. Similar to \eqref{eq:G-new12}, we know (see \cite{Tes24}) that $\phi^k_{Y_{\bullet}}$ can also be defined using sections of $mL$. The main idea is to define the homogenization of a section. First, let us choose a holomorphic coordinate system $z$ around $x$ such that $Y_j$ is defined by $z_1=\cdots=z_j=0$ around $x$ (in particular $z(x)=0$). Fix a holomorphic frame $\bm e$ of $L$ around $x$ such that $|\bm e|^2=e^{-\phi}$ around $x$, one may write each $F\in H^0(X, mL)$ as 
$
F=F(z)\bm e^{\otimes m}
$ 
near $x$, where $F(z)$ is holomorphic near the origin of $\mathbb C^n$. Then one may inductively  define
\begin{equation}\label{eq:new-hom1}
	F_{\rm hom,1}(\xi_1, z_2, \cdots, z_n):=\lim_{s\to 0} \frac{F(s\xi_1, z_2, \cdots, z_n)}{s^{{\rm ord}_{z_1=0} F}}	
\end{equation}
and the $k$-th homogenization of $F$
\begin{equation}\label{eq:new-homj}
	F_{\rm hom,k}(\xi_1,\cdots, \xi_k, z_{k+1}, \cdots, z_n):=\lim_{s\to 0} \frac{F_{\rm hom,k-1}(\xi_1,\cdots, \xi_{k-1}, s\xi_k, z_{k+1}, \cdots, z_n)}{s^{{\rm ord}_{z_k=0} F_{\rm hom,k-1}}}	
\end{equation}
on $NZ_k$ for $2\leq k\leq n$. Then Testorf \cite{Tes24} proved that:

\begin{theorem}[Theorem 1.2 in \cite{Tes24}]\label{th:Tes01} Let $(L, e^{-\phi})$ be a positive line bundle over an $n$-dimensional compact complex manifold, then
	\begin{equation}\label{eq:Tes01}
		\phi^k_{Y_{\bullet}}=\sup^*\left\lbrace \frac1m \log|F_{\rm hom,k}|^2: F\in H^0(X, mL), \, \sup_{X}\,  |F|^2e^{-m\phi}=1, \, m\in \mathbb N \right\rbrace
	\end{equation}	
for  $1\leq k\leq n$.
\end{theorem}

Since $\phi^n_{Y_{\bullet}}$ is $(S^1)^n$-invariant on $NZ_n=\mathbb C^n$, we can further write $\phi^n_{Y_{\bullet}}$ as
\begin{equation}\label{eq:RW-C2-new}
	\phi^n_{Y_{\bullet}}(\xi)=\phi_{\mathbb R}(\log|\xi_1^2|,\cdots, \log|\xi_n^2|)
\end{equation}
for some convex increasing function $\phi_{\mathbb R}$ on $\mathbb R^n$ (we call such $\phi^n_{Y_{\bullet}}$ a toric psh function). This convex function $\phi_{\mathbb R}$ is continuous (and should be $C^{1,1}$, but we do not need it in this paper) on $\mathbb R^n$, we define its gradient image $\nabla \phi_{\mathbb R}(\mathbb R^n)$ as
\begin{equation}\label{eq:RW-C2-new-1}
	\nabla \phi_{\mathbb R}(\mathbb R^n):=\bigcup_{x\in\mathbb R^n} \left\lbrace \alpha\in\mathbb R^n: \phi_{\mathbb R}(y)\geq \phi_{\mathbb R}(x)+\alpha\cdot (y-x) \ \text{for all}\ y\in \mathbb R^n\right\rbrace.
\end{equation}	
From this definition, we have 
\begin{equation}\label{eq:Legendre0}
	\nabla \phi_{\mathbb R}(\mathbb R^n)\subset\{\phi_{\mathbb R}^*<\infty\},
\end{equation} 
where
\begin{equation}\label{eq:RW-C2-new-2}
	\phi_{\mathbb R}^*(\alpha):=\sup_{y\in\mathbb R^n} \{y\cdot\alpha-\phi_{\mathbb R} (y)\}
\end{equation}
denotes the \emph{Legendre transform} of $\phi_{\mathbb R}$. The following lemma is well known.

\begin{lemma}\label{le:Legendre} The interior of $\{\phi_{\mathbb R}^*<\infty\}$ is non-empty and lies in $\nabla \phi_{\mathbb R}(\mathbb R^n)$, moreover
	\begin{equation}\label{eq:Legendre}
		\overline{\nabla \phi_{\mathbb R}(\mathbb R^n)} =  \overline{\{\phi_{\mathbb R}^*<\infty\}}.
	\end{equation}		
\end{lemma}

\begin{proof}
	By the Legendre duality, we have
	\begin{equation}\label{eq:RW-C2-new-3}
		\phi_{\mathbb R}(y)=\sup_{\alpha\in\mathbb R^n} \{y\cdot\alpha-\phi_{\mathbb R}^* (\alpha)\}.
	\end{equation}
	Assume that $\phi^*_{\mathbb R}$ is finite on the ball $\{|\alpha-\alpha_0|<\varepsilon\}$, taking a smaller $\varepsilon$ if necessary, one may assume that (note that a finite  convex function on an open set is always continuous) 
	$$
	C_\varepsilon:=\sup_{|\alpha-\alpha_0|<\varepsilon} \phi^*_{\mathbb R}(\alpha) <\infty,
	$$
	thus
	\begin{equation}\label{eq:RW-C2-new-4}
		\phi_{\mathbb R}(y)  \geq \sup_{|\alpha-\alpha_0|<\varepsilon} \{y\cdot\alpha-\phi_{\mathbb R}^* (\alpha)\} \geq y\cdot\alpha_0+\varepsilon |y|-C_\varepsilon.
	\end{equation}
	Hence the convex function $y\mapsto \phi_{\mathbb R}(y) -y\cdot\alpha_0$ tends to infinity as $|y|\to\infty$, thus it has a minimum point, say $x_0$, i.e.
	$$
	\phi_{\mathbb R}(y) \geq \phi_{\mathbb R}(x_0)+\alpha_0\cdot (y-x_0) \ \text{for all}\ y\in \mathbb R^n,
	$$
	which gives $\alpha_0\in \nabla \phi_{\mathbb R}(\mathbb R^n)$. Thus the interior of $\{\phi_{\mathbb R}^*<\infty\}$ lies in $\nabla \phi_{\mathbb R}(\mathbb R^n)$. The interior of $\{\phi_{\mathbb R}^*<\infty\}$ is non-empty because it contains the interior of the Okounkov body 
	$$
	\Delta_{Y_{\bullet}}:=\overline{\{\nu(F)/m: F\in H^0(X, mL)\setminus \{0\}, \,m\in \mathbb N\}},
	$$
	of $L$ associated to $Y_{\bullet}$ by \eqref{eq:Tes01} and \eqref{eq:Legendre0}, where 
\begin{equation}\label{eq:RW-C2-new-nuF}
	\nu(F):=({\rm ord}_{z_1=0}F, {\rm ord}_{z_2=0} F_{\rm hom, 1}, \cdots, {\rm ord}_{z_n=0} F_{\rm hom, n-1}).
		\end{equation}  
	\eqref{eq:Legendre} follows from \eqref{eq:Legendre0} and the fact that the closure of a convex set with non-empty interior is equal to the closure of its interior.
\end{proof}

\begin{theorem}\label{th:Legendre} Put 
	\begin{equation}\label{eq:RW-C4-new}
		v(\alpha):=\begin{cases}
			\liminf_{\frac{\nu(F)}m \to \alpha, \, m\to\infty}   \frac1m \log \left(\sup_X|F|^2e^{-m\phi}\right) & \alpha\in \Delta_{Y_{\bullet}} \\ 
			\infty & \alpha\notin  \Delta_{Y_{\bullet}},
			\end{cases}
	\end{equation}
where the $\liminf$ goes through all sections $F\in H^0(X, mL)$ with $F_{\rm hom,n} =\xi^{\nu(F)}$ (see \eqref{eq:RW-C2-new-nuF}, such $F$ can be seen as a generalized monic polynomial). Then we have
	\begin{equation}\label{eq:RW-C4-new1}
		v=\phi_{\mathbb R}^*, \ \  \ \Delta_{Y_{\bullet}}=\overline{\nabla \phi_{\mathbb R}(\mathbb R^n)} =  \overline{\{\phi_{\mathbb R}^*<\infty\}}.
	\end{equation}
\end{theorem}

\begin{proof} From the definition of $v$, we know that $\phi_{\mathbb R}=v^*$ and $v$ is convex, lower semi-continuous. Taking the Legendre transform, we thus obtain
	$$
	v=(v^*)^*=\phi_{\mathbb R}^*.
	$$
	Note that $\{v<\infty\}$ contains the interior of $\Delta_{Y_{\bullet}}$ and lies inside $\Delta_{Y_{\bullet}}$, thus Lemma \ref{le:Legendre} gives 
	$$
	\Delta_{Y_{\bullet}}=\overline{\nabla \phi_{\mathbb R}(\mathbb R^n)} =  \overline{\{\phi_{\mathbb R}^*<\infty\}}.
	$$	
\end{proof}

\noindent
\textbf{Remark.} \emph{By \cite[Theorem 1.1]{WN14}, the restriction of $v=\phi_{\mathbb R}^*$ to the interior of $\Delta_{Y_{\bullet}}$ is precisely the Chebyshev transform of $\phi$ with respect to $Y_{\bullet}$.}

\begin{definition}[See \cite{Mu20, DRWXZ}]\label{de:toric-deg-new} Let $(L, e^{-\phi})$ be a positive line bundle over an $n$-dimensional compact complex manifold. We call (see \eqref{eq:rho-alphaY-new-XYZ1} for the definition of $Z_k$) 
	$$
	(X,L,\phi) \to (NZ_1, \pi^*_{Y_1}L, \phi^1_{Y_{\bullet}}) \to \cdots \to (NZ_n, \pi^*_{Y_n}L, \phi^n_{Y_{\bullet}})
	$$
	the toric degeneration of $(X,L,\phi)$ with respect to $Y_{\bullet}$.
\end{definition} 

The following theorem is a special case of Theorem B (when $\Delta(L)=\Delta_{Y_\bullet}$, $m=1$, $\nu(F)=0$).

\begin{theorem}\label{th:toric-deg-new1} Let $(L, e^{-\phi})$ be a positive line bundle over a compact complex manifold $X$. Fix $x\in X$, assume that there exists a complete flag of submanifolds
	$$
	X=Y_0 \supset Y_1\supset \cdots Y_{n-1}\supset Y_n=\{x\}
	$$
	such that the associated Okounkov body $\Delta_{Y_\bullet}$ of $L$ contains  $(1,\cdots, 1)$ as an interior point. Then there exist $F\in H^0(X, K_X+L)$ with $F(x) \neq 0$.
\end{theorem}

\begin{proof} 
	By \eqref{eq:RW-C2-new}, we have 
	\begin{equation}\label{eq: toric-deg-new2}
		\int_{\mathbb C^n} i^{n^2} d\xi \wedge \overline{d\xi} \, e^{-\phi_{Y_\bullet}^n} = (2\pi)^n \int_{\mathbb R^n} e^{-\phi_{\mathbb R}(y)+y_1+\cdots+y_n} \, dy_1\cdots dy_n.	 
	\end{equation}
	Assume that $(1\,\cdots, 1)$ lies in the interior of $ \Delta$, then Theorem \ref{th:Legendre} implies that $(1\,\cdots, 1)$ lies in the interior of $\nabla \phi_{\mathbb R}(\mathbb R^n)$. Hence, similar to the proof of \eqref{eq:RW-C2-new-4}, we have
	$$
	\phi_{\mathbb R}(y)-y_1-\cdots-y_n \geq \varepsilon|y|-C_\varepsilon,
	$$	
	for some constants $\varepsilon>0, C_\varepsilon>0$. Thus
	$$
	\int_{\mathbb R^n} e^{-\phi_{\mathbb R}(y)+y_1+\cdots+y_n} \, dy_1\cdots dy_n <\infty.
	$$
	Hence \eqref{eq: toric-deg-new2} gives
	\begin{equation}\label{eq: toric-deg-new2-finite}
	\int_{\mathbb C^n} i^{n^2} d\xi \wedge \overline{d\xi} \, e^{-\phi^n_{Y_\bullet}} <\infty.
		\end{equation}
	Thanks to \eqref{eq:Tes01}, if one introduce
	$$
	F_{\rm hom,n,s}(\xi):=\frac{F_{\rm hom,n-1}(\xi_1,\cdots, \xi_{n-1},s\xi_n)}{s^{{\rm ord}_{z_n=0} F_{\rm hom,n-1}}}
	$$
	then
		\begin{equation}\label{eq: toric-deg-new3}
		\phi_{\mathbb C}^n(\xi,s):=	 \sup^*\left\lbrace \frac1k \log|F_{\rm hom,n,s}(\xi)|^2: F\in H^0(X, kL), \sup_{X}\,  |F|^2e^{-k\phi}=1 \right\rbrace
	\end{equation} 
directly gives a psh weight on the total deformation space connecting $\phi^n_{Y_\bullet}$ and $\phi^{n-1}_{Y_\bullet}$. Similar to Definition \ref{de:te}, one may use $\pi_{Y_n}: NZ_n \to Y_n=\{x\}$ to define the tautological extension, say $F_n \in H^0(NZ_n, K_{NZ_n}+\pi_{Y_n}^*L))$, of the section $dz\otimes \bm e$ at $x$, then \eqref{eq: toric-deg-new2-finite}  and the Second Main Theorem for the degeneration $NZ_{n-1} \to NZ_n$ gives an extension of $F_n$ to $F_{n-1}$ with 
	$$
	\int_{NZ_{n-1}} i^{n^2}  F_{n-1} \wedge \overline{F_{n-1}} \, e^{-\phi_{Y_\bullet}^{n-1}} \leq 	\int_{NZ_{n}} i^{n^2}  F_{n} \wedge \overline{F_{n}} \, e^{-\phi_{Y_\bullet}^{n}} =  \int_{\mathbb C^n} i^{n^2} d\xi \wedge \overline{d\xi} \, e^{-\phi_{Y_\bullet}^n} <\infty.
	$$
	Apply the Second Main Theorem  $n$ times, we eventually obtain $F:=F_0\in H^0(X, K_X+L)$ with
	$
	F(x)=dz\otimes \bm e \neq 0
	$
	and
	$
	\int_{X} i^{n^2}  F \wedge \overline{F} \, e^{-\phi} \leq \int_{\mathbb C^n} i^{n^2} d\xi \wedge \overline{d\xi} \, e^{-\phi_{Y_\bullet}^n} <\infty.
	$
\end{proof}

\noindent
\textbf{Remark.} \emph{The proof above also gives a lower bound
\begin{equation}\label{eq: toric-deg-new1}
K_{\phi}(x)\geq 	\frac{i^{n^2} dz\wedge \overline{dz}\, e^{-\phi(x)}}{\int_{\mathbb C^n} i^{n^2} d\xi \wedge \overline{d\xi} \, e^{-\phi_{Y_\bullet}^n}} 
\end{equation} 
for the Bergman kernel.}

\subsection{Toric degeneration with respect to a weighted infinitesimal Okounkov body}

The proof of Theorem B below also suggests us to look at weighted version of the toric degeneration construction in Definition \ref{de:toric-deg-new}. 	For simplicity's sake, we still assume that $(L, e^{-\phi})$ is positive over a compact $X$. Fix a holomorphic coordinate system $z$ around $x$ such that  $z(x)=0$. As before, we fix a holomorphic frame $\bm e$ of $L$ around $x$ such that $|\bm e|^2=e^{-\phi}$ and write each $F\in H^0(X, kL)$ as 
$$
F=F(z)\bm e^{\otimes k}
$$ 
near $x$, where $F(z)$ is holomorphic near the origin of $\mathbb C^n$. Consider the Taylor series
\begin{equation}\label{eq:taylor512}
	F(z)=\sum_{\alpha \in \mathbb N^n} F_\alpha z^\alpha, \ \ \  \ \ \ z^\alpha:=z_1^{\alpha_1}\cdots z_n^{\alpha_n},
\end{equation} 
of $F$ around $0$, we call $F_\alpha$ the Taylor coefficient of $F$. We shall also use linear mapping
\begin{equation}\label{eq:T-beta}
	T(\alpha)=(T_1(\alpha), \cdots, T_n(\alpha)):=\left(\sum_{k=1}^n\beta_{k}\alpha_k, \alpha_1, \cdots, \alpha_{n-1}\right), 
\end{equation}
with respect to the weight $\beta\in\mathbb Z_{>0}^n$. Each weight $\beta$ defines an order on $\mathbb N^n$ as follows.

\begin{definition}\label{de:betaorder} We say that $\alpha \prec_\beta \gamma$ if there exists $1\leq j\leq n$ such that $T_l(\alpha)=T_l(\gamma)$ for $1\leq l\leq j-1$ and $T_{j}(\alpha)<T_j({\gamma})$. We shall denote by ${\rm min}_\beta$ the minimum with respect to the order $\prec_\beta$ and
	call
	\begin{equation}\label{eq:betaorder}
		\nu_\beta(F):={\rm min}_{\beta}\{\alpha\in \mathbb N^n: F_\alpha\neq 0 \ \ \text{in}\ \ \eqref{eq:taylor512}\}, \ \ F\in H^0(X, kL)\setminus\{0\},
	\end{equation}
	the $\beta$-valuation vector of $F$. The $\beta$-weighted infinitesimal Okounkov body of $L$ at $x$ is defined by
	\begin{equation}\label{eq:beta-Okoun}
		\Delta_{\beta}:=\overline{\{\nu_\beta(F)/k: F\in H^0(X, kL)\setminus \{0\}, \,k\in \mathbb Z_{>0}\}}.
	\end{equation}
\end{definition}

\noindent
\textbf{Remark.} \emph{$\Delta_\beta$ reduces to the "straightened up" infinitesimal Okounkov body  \cite[(1.1.1)]{FL25} in case $\beta =(1,1,\cdots,1)$. We look at the weighted generalizations since the non-trivial $\beta$ is crucial in the proof of Theorem B (see also \cite{ASQTW25, And13, Kav19, WN19} and references therein for other applications of the non-trivial weight).}

\medskip

Similar to the classical case (see \cite[Theorem 4.9]{KK12}), we know that each $\Delta_\beta$ is a convex body in $\mathbb R^n$ with
\begin{equation}\label{eq:Okoun-beta}
	|\Delta_\beta|=\int_X \frac{c_1(L)^n}{n!}.
\end{equation}

Similar to \eqref{eq:valphaY-k}, we define
\begin{equation}\label{eq:valpha1}
	\psi^1_\alpha=\sup\{\psi\in {\rm PSH}(L):  \psi\leq \phi \, \text{and $\nu_{x,\beta}(\psi)\geq \alpha$}\},
\end{equation}
where (one may verify that the limit below exists)
\begin{equation}\label{eq:nu1sigma}
	\nu_{x,\beta}(\psi):=\lim_{t\to -\infty} \frac{\sup\{\psi(z): |z_1|^{2/\beta_{1}} + \cdots+   |z_n|^{2/\beta_{n}} =e^t\}}{t},  
\end{equation}
denotes the $\beta$-directional Lelong number of $\psi$ at $x$.

\begin{definition}[Following \cite{WN18, WN21}]\label{de:wccG} The $\beta$-weighted $S^1$-symmetrization of $\phi$ on $\mathbb C^n$ is defined as
	$$
	\phi^{1,\beta}_x:=\sup_{\alpha\geq 0} \rho^1_\alpha, \ \ \rho^1_\alpha(\xi):=\limsup_{s\to 0}  \{ \psi^1_\alpha(s^{\beta_{1}}\xi_1, \cdots, s^{\beta_{n}}\xi_n) -\alpha \log|s^2|\}
	$$
\end{definition}

\noindent
\textbf{Remark.} \emph{	Similar to the canonical growth condition $\phi_x$ in \cite{WN18}, $\phi^{1,\beta}_{x}$ is also $S^1$-invariant:
	$$
	\phi^{1,\beta}_{x} (s^{\beta_{1}}\xi_1, \cdots, s^{\beta_{n}}\xi_n) = \phi^{1,\beta}_{x}(\xi), \ \ \forall \  \xi\in\mathbb C^n, \ \ |s|=1.
	$$
	In case $\beta_{1}=\cdots=\beta_{n}=1$, $\phi^{1,\beta}_x=\phi_x$ reduces to the canonical growth condition in \cite{WN18}. Using weighted deformation to the normal bundle, we may also construct an $S^1$-invariant psh weight connecting $\phi$ and its weighted $S^1$-symmetrization. Then the Second Main Theorem gives (compare with Theorem \ref{th:oldtheoremC-x}):}

\begin{theorem}\label{th:RWThC1} Let $(L, e^{-\phi})$ be a positive line bundle over an $n$-dimensional compact complex manifold $X$. Fix $x\in X$; with respect to the notations above, if
	\begin{equation}\label{eq:RWThC1-int-condition}
		I_1:=	\int_{\mathbb C^n} i^{n^2} d\xi \wedge \overline{d\xi} \, e^{-m \phi_{x}^{1,\beta}(\xi)} <\infty , 
	\end{equation}
	then there is a holomorphic section $F_X$ of $K_{X}+mL$ such that 
	\begin{equation}\label{eq:RWThC1-estmate}
		F_X(x)=dz\otimes \bm e^{\otimes m}, \ \ \int_{X} i^{n^2} F_X \wedge \overline{F_X} \, e^{-m\phi} \leq  I_1. 
	\end{equation}
\end{theorem}

To continue the process, one may introduce
	\begin{equation}\label{eq:2tc}
		\psi_\alpha^2:= \sup \left\lbrace \psi\in {\rm PSH}(\mathbb C^n):\psi\leq \phi_{x}^{1,\beta}, \ \nu_{\xi_1=0}(\psi)\geq \alpha\right\rbrace,
	\end{equation}
and define the $\beta$-weighted $(S^1)^2$-symmetrization of $\phi$ on $\mathbb C^n$ as
$$
\phi^{2,\beta}_x:=\sup_{\alpha\geq 0} \rho^2_\alpha, \ \ \rho^2_\alpha(\xi):=\limsup_{s\to 0}  \{ \psi^2_\alpha(s\xi_1, \xi_2, \cdots, \xi_n) -\alpha \log|s^2|\}.
$$
For $k\geq 3$, we inductively define
$$
	\psi_\alpha^k:= \sup \left\lbrace \psi\in {\rm PSH}(\mathbb C^n):\psi\leq \phi_{x}^{k-1,\beta}, \ \nu_{\xi_{k-1}=0}(\psi)\geq \alpha\right\rbrace,
$$
and call
	\begin{equation}\label{eq:beta-weighted-new}
\phi^{k,\beta}_x:=\sup_{\alpha\geq 0} \rho^k_\alpha, \ \ \rho^k_\alpha(\xi):=\limsup_{s\to 0}  \{ \psi^k_\alpha(\xi_1, \cdots, \xi_{k-2}, s\xi_{k-1}, \xi_k, \cdots, \xi_n) -\alpha \log|s^2|\}
	\end{equation}
the $\beta$-weighted $(S^1)^k$-symmetrization of $\phi$ around $x$

\begin{definition}\label{de:toric-deg-new-weighted} Let $(L, e^{-\phi})$ be a positive line bundle over an $n$-dimensional compact complex manifold. We call 
	$$
	(X,L,\phi) \to (\mathbb C^n, \mathbb C^{n+1}, \phi^{1,\beta}_x) \to \cdots \to (\mathbb C^n, \mathbb C^{n+1}, \phi^{n,\beta}_x)
	$$
	the toric degeneration of $(X,L,\phi)$ around $x\in X$ with respect to the weight $\beta\in\mathbb Z_{>0}^n$.
\end{definition}

The toric psh function $\phi^{n,\beta}_x$ decodes many properties of $(X, L, \phi)$. One example is the following result proved by the second author. 

\begin{theorem}[Theorem D and Corollary 1.1 in \cite{Tes24}]\label{th:RW-Tes} Let $(L, e^{-\phi})$ be a positive line bundle over an $n$-dimensional compact complex manifold $X$. Then the $\beta$-weighted $(S^1)^n$-symmetrization of $\phi$ satisfies
$$
		\phi_x^{n,\beta}(\xi)=\sup^*\left\lbrace \frac1m \log|F_{\nu_\beta(F)} \xi^{\nu_\beta(F)}|^2: F\in H^0(X, mL), \sup_{X}\,  |F|^2e^{-m\phi}=1, \, m\in \mathbb N \right\rbrace.
$$
	In particular, if we write
	\begin{equation}\label{eq:RW-Tes1}
		\phi^{n,\beta}_{x}(\xi)=u(\log|\xi_1^2|, \cdots, \log|\xi_n^2|),
	\end{equation}
	where $u$ is convex increasing on $\mathbb R^n$, then 
	\begin{equation}\label{eq:RW-Tes2}
		\overline{\nabla u(\mathbb R^n)}=\Delta_\beta.
	\end{equation}
\end{theorem}

\noindent
\textbf{Remark.} \emph{From the above construction and \eqref{eq:betaorder}, we have
	\begin{equation}\label{eq:RW-Tes3}
		\phi^{n,\beta}_{x}(\xi) \geq \frac1m \log \left(|F_{\nu_\beta(F)}\xi^{\nu_\beta(F)}|^2 \right)
	\end{equation}
	for every $F\in H^0(X, mL)\setminus\{0\}$ with $|F|^2e^{-m\phi}\leq 1$ on $X$, which implies that the convex function $u$ in \eqref{eq:RW-Tes1} satisfies that 
	\begin{equation}\label{eq:RW-Tes4}
		u(y) \geq \frac{\nu_\beta(F)}{k} \cdot y + \frac1k \log \left(|F_{\nu_\beta(F)}|^2 \right), \ \ y\in \mathbb R^n.
	\end{equation}
	Thus one may easily obtain
	\begin{equation}\label{eq:RW-Tes5}
		\overline{\nabla u(\mathbb R^n)} \supset \Delta_\beta
\end{equation}
without using the above theorem.}

\medskip

The following result is also a special case of Theorem B.

\begin{theorem}\label{th:C-weight} Let $L$ be a positive line bundle over a compact complex manifold $X$. Fix $x\in X$, assume that there exists a  weighted infinitesimal Okounkov body of $L$ at $x$ (see Definition \ref{de:betaorder}) containing $(1,\cdots, 1)$ as an interior point, then $K_X+L$ has a global holomorphic section non-vanishing at $x$.
\end{theorem}

\begin{proof}
	Assume that there exists $\beta\in\mathbb Z_{>0}^n$ such that the interior of $\Delta_\beta$ contains $(1,\cdots, 1)$. Then \eqref{eq:RW-Tes5} implies that $(1,\cdots, 1)$ is an  interior point of the convex set $\nabla u(\mathbb R^n)$. Hence we have
	$$
	\int_{\mathbb C^n} i^{n^2} d\xi \wedge \overline{d\xi} \, e^{-\phi_{x}^{n,\beta}(\xi)} = (2\pi)^n \int_{\mathbb R^n} e^{-u(y)+y_1+\cdots+y_n} \, dy_1\cdots dy_n <\infty.	 
	$$
	Thus the Second Main Theorem applies (same as the proof of Theorem \ref{th:toric-deg-new1}).
\end{proof}
		
		\subsection{Weighted Chebyshev transform and weighted Chebyshev body} Similar to \eqref{eq:RW-C4-new} in Theorem \ref{th:Legendre}, the convex function  $u$ in \eqref{eq:RW-Tes1} associated to the $\beta$-weighted $(S^1)^n$-symmetrization of $\phi$ equals the Legendre transform of the following convex function $v$ on $\mathbb R^n$ 
\begin{equation}\label{eq:RW-C4}
		v(\alpha):=\begin{cases}
		\liminf_{\frac{\nu(F)}m \to \alpha, \, m\to\infty}   \frac1m \log \left(\sup_X|F|^2e^{-m\phi}\right) & \alpha\in \Delta_{\beta} \\ 
		\infty & \alpha\notin  \Delta_{\beta},
	\end{cases}
\end{equation}
where the $\liminf$ goes through all sections $F\in H^0(X, mL)$ with Taylor coefficient $F_{\nu_\beta(F)} =1$ (generalizations of monic polynomials). The restriction of $v$ to the interior of $\Delta_{\beta}$ can be seen as a weighted version of the Chebyshev transform $c[\phi]$ in \cite{WN14}.

\begin{definition}\label{de:beta-cheby} We call $v$ in \eqref{eq:RW-C4} the $\beta$-Chebyshev transform of $\phi$.
\end{definition}

The local case of the above construction is also interesting. To simplify the discussions, we shall only consider bounded hyperconvex domains. Recall that a domain $X$ in $\mathbb C^n$ is called \emph{hyperconvex} if there is a negative smooth strictly psh function $\rho$ on $X$ such that the sublevel set $\{\rho<-\varepsilon\}$ is relatively compact in $X$ for every $\varepsilon>0$. Let $X$ be a bounded hyperconvex domain in $\mathbb C^n$ and $\phi$ be a bounded psh function on $X$. Fix $x\in X$ and a weight $\beta$, one may still use \eqref{eq:valpha1} to define the test curve
\begin{equation}\label{eq:valpha1-local}
	\psi^1_\alpha=\sup\{\psi\in {\rm PSH}(X):  \psi\leq \phi \, \text{and $\nu_{x,\beta}(\psi)\geq \alpha$}\}.
\end{equation}
Different from the campact case, the critical value $\lambda_{\psi_\alpha^1}=\infty$ now, but one can always fix a positive $\lambda$ (let $\lambda$ go to infinity later) and look at the test curve $\psi_\alpha$ such that $\psi_\alpha=\psi^1_\alpha$ for $\alpha\leq \lambda$ and $\psi_\alpha=-\infty$ for $\alpha>\lambda$. Then $\psi_\alpha$ has finite critical value $\lambda$. To simplify the argument, we shall look at $\psi^1_\alpha$ directly and define the $\beta$-weighted $(S^1)^k$-symmetrization  $\phi_x^{k, \beta}$ around $x$ as in \eqref{eq:beta-weighted-new}. Similar to Theorem \ref{th:C-weight}, our Second Main Theorem gives:

\begin{theorem}\label{th:RWThCn-domain} Let $X$ be a bounded hyperconvex domain in $\mathbb C^n$ and $\phi$ be a bounded psh function on $X$. Fix $x\in X$ and a weight $\beta\in \mathbb Z_{>0}^n$. Then there exists a holomorphic function $f$ on $X$ with $f(x)=1$ and
	\begin{equation}\label{eq:RWThCn-domain}
		\int_{X} |f|^2 e^{-\phi} \leq  \int_{\mathbb C^n}  e^{-\phi_{x}^{1, \beta}} \leq \cdots \leq  \int_{\mathbb C^n}  e^{-\phi_{x}^{n, \beta}}.
	\end{equation}
\end{theorem}

If $\phi=0$ then we have $\psi_\alpha^1=\alpha G$, $\alpha\geq 0$, where
\begin{equation}\label{eq:RWThCn-domain1}
	G:=\sup\{\sigma\in {\rm PSH}(X):  \sigma\leq 0 \, \text{and $\nu_{x,\beta}(\sigma)=1$}\}
\end{equation}
is the pluricomplex green function on $X$ with
$$
\log \left( |z_1-x_1|^{2/\beta_{1}} + \cdots+   |z_n-x_n|^{2/\beta_{n}} \right)
$$
singularity around $x$. Thus we have
\begin{equation}\label{eq:phix1-new}
	\phi_x^{1,\beta}=\lim_{\lambda\to\infty} \lambda\max\{0, A^1\},
\end{equation}
where
\begin{equation}\label{eq:AZU-beta}
	A^1(\xi):=\limsup_{s\to 0} \{ G(x_1+s^{\beta_{1}}\xi_1, \cdots, x_n+s^{\beta_{n}}\xi_n) - \log|s^2|\}
\end{equation}
is a generalized Azukawa pseudometric (see \cite[section 1]{Zwo00} and the original paper \cite{Azu86} for the classical case). In particular, we have 
$$
\text{$\phi_x^{1,\beta} =0$ on $\{A^1\leq 0\}$ 	and $\phi_x^{1,\beta} =\infty$ on $\{A^1> 0\}$.}
$$ 
Since $A^1$ is weighted 1-$\log$ homogeneous and continuous (see \cite[Theorem C]{Tes24}, see also \cite{Zwo00} for the unweighted case), we know that $\{A^1\leq 0\}$ is the closure of  
$$
X^1:=\{\xi\in \mathbb C^n: A^1(\xi)<0\},
$$
and $\partial X^1$ has zero Lebesgue measure.
Thus we have 
$$
e^{-\phi_{x}^{1,\beta}}=1_{\overline{X^1}} \  \ \text{and} \ \  \int_{\mathbb C^n}  e^{-\phi_{x}^{1,\beta}}=|X^1|,
$$
where $|X^1|$ denotes the Lebesgue measure of $X^1$. For general $1\leq k\leq n$, we have
$$
e^{-\phi_{x}^{k, \beta}}=1_{\overline{X^k}},
$$
where each $X^k$ is an $(S^1)^k$-invariant bounded hyperconvex domain in $\mathbb C^n$. In particular, $X^n$ is toric around $x$:
$$
x+\xi\in X^n \Leftrightarrow x+(|\xi_1|,\cdots, |\xi_n|)\in X^n.
$$ 
Similar to \eqref{eq:phix1-new}, we have		
\begin{equation}\label{eq:phixn-new}
	\phi_x^{n,\beta}=\lim_{\lambda\to\infty} \lambda\max\{0, A^n\},
\end{equation}
with (the proof is similar to Theorem \ref{th:RW-Tes})
\begin{equation}\label{eq:AZUn-beta}
	A^n(\xi):=\sup^*\left\lbrace \frac{\log \left( |f_{\nu_{\beta}(f)} \xi^{\nu_{\beta}(f)}|^2 \right)}{\nu_{\beta}(f)\cdot\beta} : f\in \mathcal O(X), \  \sup_X |f|=1, \  f(x)=0 \right\rbrace,
\end{equation}
where the $\beta$-valuation vector $\nu_{\beta}(f)$ is defined by (see Definition \ref{de:betaorder})
$$
\nu_\beta(f):={\rm min}_{\beta}\{\alpha\in \mathbb N^n: f_\alpha\neq 0\} \ \text{for the Taylor expansion} \ f(z)=\sum f_\alpha (z-x)^\alpha.
$$
Similar to \eqref{eq:RW-C4}, put
\begin{equation}\label{eq:Cheby-domain}
	v(\alpha):=\liminf_{\nu_\beta(f)/k \to \alpha, \, k\to\infty}  \left\lbrace \frac1k \log \left(\sup_X|f|^2\right): f\in\mathcal O(X), f_{\nu_\beta(f)} =1\right\rbrace,
\end{equation}
for $\alpha\in\mathbb R^n_{\geq 0}$ and $v(\alpha):=\infty$ for $\alpha\notin\mathbb R^n_{\geq 0}$, then 
\begin{equation}\label{eq:Cheby-domain1}
	\phi_x^{n,\beta}(\xi)=v^*(\log|\xi_1^2|, \cdots, \log|\xi_n^2|), \ \ \ A^n(\xi)=\sup_{\alpha\cdot\beta=1} \left\lbrace\sum_{j=1}^n \alpha_j \log |\xi_j^2| -v(\alpha)\right\rbrace,
\end{equation}
and
\begin{equation}\label{eq:Cheby-domain2}
	X^n=\{\xi\in\mathbb C^n: A^n(\xi)<0 \}.
\end{equation}		

\begin{definition}\label{de:cheby-body} We call $X^k$ the $k$-th $\beta$-weighted Chebyshev body of $X$ around $x$ and
	$$
	X\to X^1 \to \cdots\to X^n
	$$
	the toric degeneration of $X$ to $X^n$ around $x$ with respect to the weight $\beta\in \mathbb Z_{>0}^n$.
\end{definition}

The following is the $\phi=0$ case of Theorem \ref{th:RWThCn-domain}.

\begin{theorem}\label{th:RWThCn-domain0} Let $X$ be a bounded hyperconvex domain in $\mathbb C^n$. Fix $x\in X$ and a weight $\beta\in \mathbb Z_{>0}^n$. Then there exists a holomorphic function $f$ on $X$ with $f(x)=1$ and
	\begin{equation}\label{eq:RWThCn-domain0}
		\int_{X} |f|^2  \leq  |X^1| \leq \cdots \leq  |X^n|.
	\end{equation}
\end{theorem}

\noindent		
\textbf{Remark.} \emph{In case $\beta_{1}=\cdots =\beta_{n}=1$, the first inequality $\int_{X} |f|^2  \leq  |X^1| $ reduces to Theorem 7.5 in \cite{Blo14}. But the final estimate $|X^1|\leq |X^n|$ is new even in that case.}

\section{Applications and generalizations}

\subsection{Hessian comparison theorem}

We shall use the following Hessian comparison theorem (see \cite[Lemma 1.13 in page 14 and Theorem A in page 19]{GW79}, see also \cite[Theorem 1.2]{LW05} for a stronger comparison theorem when the curvature has a lower bound).

\begin{theorem}\label{th:comparison} Let $(X,\omega)$ be a Riemann surface. Fix $x\in X$, suppose that the exponential map ${\rm exp}_x: T_xX\to X$ is a diffeomorphism from 
	$$
	B:=\{v\in T_xX: |v|<\pi\}
	$$
to $B_x:=\{y\in X: d(x, y)<\pi\}$ and the scalar curvature $S\leq 1$ on $B_x$. If $f: (0, \pi) \to \mathbb R$ is $C^{1,1}$ increasing,  $f(2\arctan|z|)+2\log(1+|z|^2)$ is subharmonic in $z\in \mathbb C$ and $f(x)\to-\infty$ as $x\to0$ then 
	\begin{equation}\label{eq:comparison} 
			i\partial\dbar (f\circ d) + \omega\geq 0
	\end{equation}
on $B_x$, where $d:=d(\cdot, x)$ denotes the $x$ distance function on $X$. 
\end{theorem}

\noindent
\textbf{Remark.} \emph{The function $2\arctan |z|$ appears because the distance function on $(\mathbb P^1, \omega^{\rm FS})$, where $\omega^{\rm FS}:=2 i\partial\dbar \log (1+|z|^2))$, satisfies
$$
d^{\rm FS}(0,z)= 2\arctan|z|,
$$
one may check that $C^{1,1}$ increasing of $f$ is enough for \eqref{eq:comparison} on $B_x\setminus\{x\}$, $f(x)\to-\infty$ as $x\to 0$ further implies that \eqref{eq:comparison} holds around $x$. To apply the above theorem, we need the following lemma.}

\begin{lemma}\label{le:comparison1} For every $0< \alpha\leq 2$, the envelope
		\begin{equation}\label{eq:comparison1} 
		v_\alpha:=\sup\left\lbrace v\in  {\rm PSH}(\mathbb C): v(z) \leq 2\log(1+|z|^2) \ \text{on $\mathbb C$}, \nu_0(v)\geq \alpha\right\rbrace
	\end{equation}
satisfies
	\begin{equation}\label{eq:comparison11} 
v_\alpha(z)=
\begin{cases}
	\alpha \log|z^2|+2\log 2-\alpha\log\alpha-(2-\alpha)\log(2-\alpha) &  \text{for}\  |z|^2< \frac\alpha{2-\alpha} \\
	2\log(1+|z|^2)  &  \text{for}\  |z|^2\geq  \frac\alpha{2-\alpha}. 
\end{cases}
	\end{equation}
where $(2-\alpha)\log(2-\alpha) :=0$ when $\alpha=2$.
\end{lemma}

\begin{proof} Notice that each $v_\alpha$ is $S^1$-invariant psh, thus there exists a convex increasing function $u_\alpha$ on $\mathbb R$ such that 	\begin{equation}\label{eq:comparison12}  v_\alpha(z)=u_\alpha(\log|z^2|),
			\end{equation}
moreover
$$
u_\alpha(x)=\sup \left\lbrace u(x)\leq 2\log(1+e^x):  \text{$u$ is convex on $\mathbb R$ with $u'(-\infty) \geq \alpha$} \right\rbrace.
$$
Thus for $0<\alpha<2$, $u_\alpha$ is determined by the $\alpha$-slope line tangent to $2\log(1+e^x)$ and $u_2(x) =2x$. A direct computation gives
$$
u_\alpha=
\begin{cases}
	\alpha x+2\log 2-\alpha\log\alpha-(2-\alpha)\log(2-\alpha) &  \text{for}\ \frac{2 e^x}{1+e^x}<\alpha \\
	2\log(1+e^x) &  \text{for}\  \frac{2 e^x}{1+e^x}\geq \alpha.
\end{cases}
$$
for $0<\alpha<2$. Thus \eqref{eq:comparison12} gives 	\eqref{eq:comparison11}.
\end{proof}

\noindent
\textbf{Remark.} \emph{Another way of computing $v_\alpha$ is to use the geodeisc ray
\begin{equation}\label{eq:comparison13}
\phi^{\rm FS}_t(z):=2\log(1+e^t|z|^2)
	\end{equation}
emanating from the Fubini-Study potential $\phi^{\rm FS}_0=2\log (1+|z|^2)$. A direct computation gives 
\begin{equation}\label{eq:comparison14}
v_\alpha(z)=\inf_{t>0} \, \{\phi^{\rm FS}_t(z)-t\alpha\}.
	\end{equation}
One may observe that $v_\alpha(z)-2\log(1+|z|^2) $ is $C^{1,1}$ and increasing in $|z|$, thus the above lemma implies that:}

\begin{lemma}\label{le:comparison2} For every $0< \alpha\leq 2$, let us determine $f_\alpha$ such that
	\begin{equation}\label{eq:comparison2} 
	f_\alpha(2\arctan|z|)=v_\alpha(z)-2\log(1+|z|^2).
\end{equation}
Then each $f_\alpha$ 
	is $C^{1,1}$ increasing on $(0,
	\pi)$, $f_\alpha(2\arctan|z|)+2\log(1+|z|^2)$ is subharmonic in $z\in \mathbb C$ and $f(x) \to-\infty$ as $x\to0$.
\end{lemma}

\subsection{Proof of Theorem A} Let $d$ be the distance function on $(X, i\partial\dbar\phi)$, for $0< \alpha\leq 2$, put
\begin{equation}\label{eq:AD-Bo}
	\psi_\alpha(z):=
	\begin{cases}
		f_\alpha(d(z,x))+\phi(z) & d(z,x)< \arccos(1-\alpha); \\
		\phi(z) & d(z,x)\geq \arccos(1-\alpha).
	\end{cases}
\end{equation}
Extend $\psi_\alpha$ such that $\psi_\alpha:=\phi$ for $\alpha\leq 0$ and $\psi_\alpha:=-\infty$ for $\alpha> 2$.
Since $L_0\geq 2\pi$ and $S\leq 1$, we know that the exponential map ${\rm exp}_x$ on $T_xX$ is a diffeomorphism on
$
B:=\{v\in T_xX: |v|<\pi\}.
$
Thus we can apply Theorem \ref{th:comparison} to $\omega:=i\partial\dbar\phi$. By Lemma \ref{le:comparison2}, we know that $\psi_\alpha$ defines a test curve in ${\rm TC}_\phi$ with $\mathcal I(\psi_1)=\mathfrak{m}_x$. Thus our First Main Theorem gives $F\in H^0(X, K_X+L)$ with $F(x)\neq 0$ and
$$
\int_X i F\wedge\bar F \, e^{-\phi}  \leq \liminf_{t\to\infty} e^t\int_X i F\wedge \overline{F} \, e^{-\phi_t},
$$
where $\phi_t$ is the associated subgeodesic ray $\phi_t:=\sup_{\alpha\in\mathbb R} \{\psi_\alpha+t\alpha\}$.  By  \eqref{eq:comparison13} (or a direct computation), we have
\begin{equation}\label{eq:AD0}
	\phi_t=\phi+2\log\frac{1+e^t E}{1+E},
\end{equation}
where
$$
E(z):=
\begin{cases}
	\tan^2\left(\frac{d(z,x)}{2} \right) & d(z,x)<\pi; \\
	\infty & d(z,x)\geq \pi.
\end{cases}
$$
Hence we have 
$$
\int_X i F\wedge \overline{F} \, e^{-\phi} \leq \liminf_{t\to\infty} e^t \int_X i F\wedge \overline{F} \, e^{-\phi} \frac{(1+E)^2}{(1+e^tE)^2}.
$$
Choosing a local coordinate $z$ around $x$ with $z(x)=0$ and $d(z,x)=|z|+ \,\text{higher order terms}$, we know that
$E(z)=|z|^2/4+ \,\text{higher order terms}$. Then a change of coordinates $\xi=e^{t/2}z$ gives
$$
\liminf_{t\to\infty} e^t \int_X i F\wedge \overline{F} \, e^{-\phi} \frac{(1+E)^2}{(1+e^tE)^2}= \frac{i F(x)\wedge \overline{F(x)}e^{-\phi(x)}}{\omega(x)} \int_{\mathbb C} \frac{1}{(1+|\xi|^2/4)^2},
$$
from which Theorem A follows (note that $\int_{\mathbb C} \frac{1}{(1+|\xi|^2/4)^2}=4\pi$).

\subsection{Non-compact version of the First Main Theorem}

The main step in the proof of the First Main Theorem is to apply Theorem \ref{th:Bern20}, which is only proved when $X$ is compact. In particular, in case $X$ is non-compact, one has to find a method to bypass Theorem \ref{th:Bern20}. Our idea is to apply the main result in \cite{NW24}.

\begin{theorem}\label{th:noncompact} Let $(L, e^{-\phi})$ be a pseudoeffective line bundle over an $n$-dimensional qck (see \cite[section 2.1]{NW24}) manifold $X$. Fix $\psi_\alpha\in {\rm TC}_\phi$,  $0<\alpha<\lambda:=\lambda_{\psi_\alpha}$. Assume that $\psi_{\lambda}$ is not identically $-\infty$  and
	\begin{equation}\label{eq:quasilinear}
		\sup_{0\leq \beta\leq \lambda, \,\psi_{\lambda}(x)>-\infty } \big|\psi_\beta - \left(\frac\beta\lambda\psi_\lambda+\left(1-\frac\beta\lambda\right)\phi\right)|<\infty.
	\end{equation}
	Then for every $F\in H^0(X, \mathcal O(K_X+L)\otimes \mathcal I(\phi))$, there exists $ F_\alpha$ such that
	$$
	F_\alpha- F \in H^0(X, \mathcal O(K_X+L)\otimes \mathcal I(\psi_\alpha))
	$$
	and
	$$
	\int_X i^{n^2} F_\alpha\wedge \overline{F_\alpha} \, e^{-\phi} \leq \liminf_{t\to \infty} e^{\alpha t} \int_X i^{n^2} F\wedge \bar F \, e^{-\hat\psi_t}.
	$$
\end{theorem}

\begin{proof} In case $\psi_\beta = \frac\beta\lambda\psi_\lambda+\left(1-\frac\beta\lambda\right)\phi$ for every $0\leq \beta \leq \lambda$, we have
	$$
	\hat\psi_t=\max\{\psi_{\lambda}+\lambda t, \phi\},
	$$
	which is precisely the test curve used in \cite[Theorem 2.4]{NW24}. We call such test curves linear test curves. The main observation is that \cite[Theorem 2.4]{NW24} also applies to test curves that are uniformly close (in the sense of \eqref{eq:quasilinear}) to linear test curves. Thus the above theorem follows from the proof of Theorem C in \cite{NW24}. 
\end{proof}

\begin{definition}\label{de:quasilinear} We call $\psi_\alpha\in {\rm TC}_\phi$ a quasi-linear test curve if \eqref{eq:quasilinear} holds.			
\end{definition}

We believe that the First Main Theorem is still true for non-quasi-linear test curves on pseudoconvex domains, but we do not know the proof. 

\subsection{Ohsawa-Takegoshi extension for jets}

Let $(L, e^{-\phi})$ be a pseudoeffective line bundle over a compact K\"ahler manifold $X$. Fix $\psi_\alpha\in {\rm TC}_\phi$ with critical value $\lambda$; the strong openness theorem implies that
$$
\alpha \mapsto \dim H^0(X, \mathcal O(K_X+L) \otimes \mathcal I(\psi_\alpha))
$$
is an integer valued lower semi-continuous function for $\alpha<\lambda$. Denote by
$$
0<\alpha_1<\cdots <\alpha_N <\lambda
$$
its jumping numbers (non-continuous points) and put
$$
\mathcal F_{\alpha_j}:= H^0(X, \mathcal O(K_X+L) \otimes \mathcal I(\psi_{\alpha_j})).
$$
We have
$$
H^0(X, \mathcal O(K_X+L)\otimes\mathcal I(\phi) )  \supset \mathcal F_{\alpha_1} \supset\cdots \supset \mathcal F_{\alpha_N}.
$$
For each $1\leq k \leq N$, the image of the following restriction mapping
$$
\iota_k: H^0(X, \mathcal O(K_X+L)\otimes \mathcal I(\phi)  ) \to  H^0(X, \mathcal O(K_X+L) \otimes \mathcal I(\phi) /\mathcal I(\psi_{\alpha_k}))
$$
is isomorphic to $H^0(X, \mathcal O(K_X+L)\otimes\mathcal I(\phi)  ) / \mathcal F_{\alpha_k}$. Fix
$$
f\in {\rm Im}\, \iota_k \simeq H^0(X, \mathcal O(K_X+L)\otimes\mathcal I(\phi) ) / \mathcal F_{\alpha_k}
$$
and choose $F\in H^0(X, \mathcal O(K_X+L) \otimes\mathcal I(\phi))$ such that $\iota_k(F)= f$. Then $F$ has the following orthogonal decomposition
$$
F= F_1+\cdots+F_k +R,  
$$
where $F_1 \in \mathcal F_{\alpha_1}^{\bot}$, $F_2 \in  \mathcal F_{\alpha_1} \cap \mathcal F_{\alpha_2}^{\bot}, \cdots, F_k \in  \mathcal F_{\alpha_{k-1}} \cap \mathcal F_{\alpha_k}^{\bot}$  and $R\in \mathcal F_{\alpha_k}$. Since different choices of $F$ differ only by a term in $\mathcal F_{\alpha_k}$, we know that $F_1, \cdots, F_k$ are independent of the choice of $F$. We call
\begin{equation}\label{eq:jet}
	||f||_{\rm jet}^2:= \sum_{j=1}^k \liminf_{t\to \infty} e^{\alpha_j t}\int_{X} i^{n^2} F_j\wedge \bar F_j e^{-\hat\psi_t}
\end{equation}
the (squared) jet norm of $f\in {\rm Im}\, \iota_k $. We call it the jet norm since it contains the norm of higher jets (derivatives) in case the variety associated to $\mathcal I(\psi_\alpha)$ is not reduced. The following result is a jet extension version of the First Main Theorem.

\begin{theorem}\label{th:jet} Let $(L, e^{-\phi})$ be a pseudoeffective line bundle over a compact K\"ahler manifold $X$. Let  $\psi_\alpha\in {\rm TC}_\phi$. Fix $1\leq k \leq N$, with respect to the notations above, then the following restriction mapping
	$$
	\iota_k: H^0(X, \mathcal O(K_X+L) \otimes\mathcal I(\phi) ) \to  H^0(X, \mathcal O(K_X+L) \otimes\mathcal I(\phi)/\mathcal I(\psi_{\alpha_k}))
	$$
	is surjective. Moreover, every section $f\in H^0(X, \mathcal O(K_X+L) \otimes\mathcal I(\phi)/\mathcal I(\psi_{\alpha_k}))$ has an extension $F_X \in H^0(X, \mathcal O(K_X+L)\otimes\mathcal I(\phi))$ (here "extension" means $\iota_k(F_X)=f$) with
	$$
	\int_X i^{n^2} F_X\wedge \overline{ F_X} e^{-\phi}  \leq ||f||^2_{\rm jet}.
	$$
\end{theorem}

\begin{proof} By the strong openness theorem, we can choose $\alpha_k<\alpha<\lambda$ such that 
	\begin{equation}\label{eq:SOT}
		\mathcal I(\psi_{\alpha_k})=\mathcal I(\psi_{\alpha}).
	\end{equation}
	Put 
	$$
	G:=\frac{\alpha_k}{\alpha}  \psi_{\alpha} + \left(1-\frac{\alpha_k}{\alpha} \right) \phi.
	$$
	The decreasing concavity of $\alpha\mapsto \psi_\alpha$ implies that
	$$
	\psi_\alpha\leq G\leq \psi_{\alpha_k}.
	$$
	Thus \eqref{eq:SOT} gives $\mathcal I(G) =\mathcal I(\psi_{\alpha_k})$. Hence the main theorem in \cite{CDM17} (see also \cite{ZZ19}) gives the surjectivity of 
	$$
	H^0(X, \mathcal O(K_X+L) \otimes\mathcal I(\phi)) \to  H^0(X, \mathcal O(K_X+L) \otimes\mathcal I(\phi)/\mathcal I(G)).
	$$
	From $\mathcal I(G) =\mathcal I(\psi_{\alpha_k})$ we obtain  the surjectivity of $\iota_k$. A similar argument also gives
	$$
	\mathcal F_{\alpha_j} / \mathcal F_{\alpha_{j+1}} \simeq  H^0(X, \mathcal O(K_X+L) \otimes \mathcal I(\psi_{\alpha_{j}})/\mathcal I(\psi_{\alpha_{j+1}})).
	$$
	Then the estimate part follows from the final section of \cite{Bern20} and Theorem \ref{th:DXZ}.
\end{proof}

 One may also generalize our Second Main Theorem to jets. The following is the domain case (see Theorem \ref{th:OT-jet-2} for the general case).

 \begin{theorem}\label{th:OT-jet-1} Let $X$ be a pseudoconvex domain in $\mathbb C^n$ with $0\in X$. Let $\phi_{\mathbb C}$ be a psh function on
 	$$
 	X_{\mathbb C}:=\{(\xi,s)\in\mathbb C^n\times\mathbb C: (s^{\beta_1}\xi_1, \cdots, s^{\beta_n}\xi_n)\in X\},
 	$$
 	where $\beta:=(\beta_1,\cdots, \beta_n) \in \mathbb Z_{>0}^n$. Assume that
 	\begin{equation}\label{eq:phiC-beta}
 		\phi_{\mathbb C}(e^{i\beta_1\theta} \xi_1,\cdots, e^{i\beta_n\theta}\xi_n,e^{-i\theta}s)=\phi_{\mathbb C}(\xi,s), \ \forall \ (\xi,s)\in X_{\mathbb C}, \ \theta\in\mathbb R.
 	\end{equation}
 	Put $\phi_s(\xi)=\phi_{\mathbb C}(\xi,s)$. Then for every non-negative integer $m$ and $\beta$-weighted degree $m$ polynomial
 	$$
 	P(\xi):=\sum_{\nu\cdot\beta=m} c_\nu \xi^\nu,
 	$$ 
 	there exists $f\in\mathcal O(X)$ with 
 	\begin{equation}\label{eq:OT-jet-1}
 		{\rm ord}_0^\beta (f-P) >m, \ \ \int_X |f|^2 e^{-\phi_1} \leq \int_{\mathbb C^n} |P|^2 e^{-\phi_0},
 	\end{equation}
 	where ${\rm ord}_0^\beta F$ denotes the vanishing order of $s\mapsto F(s^{\beta_1}z_1, \cdots, s^{\beta_n}z_n)$ at $s=0$.
 \end{theorem}

 \begin{proof} We modify the proof of Theorem \ref{th:Ohsawa1} so that it fits for jets. Denote by
 	$$
 	X_s:=\{\xi\in\mathbb C^n: (s^{\beta_1}\xi_1, \cdots, s^{\beta_n}\xi_n)\in X\},
 	$$
 	and put
 	$$
 	A(s):= \min\left\lbrace||F||_s^2: F\in\mathcal O(X_s), \, {\rm ord}_0^\beta (F-P) >m\right\rbrace,
 	$$
 	where
 	$$
 	||F||_s^2:=\int_{X_s} |F|^2\, e^{-\phi_s}.
 	$$
 	Then \cite[proof of Theorem 3.1]{BL16} implies that 
 	$$
 	A(s)=\sup_{\xi_s} \frac{|\xi_s(P)|^2}{||\xi_s||_s^2},
 	$$
 	where the supremum is taken over all functionals $\xi_s$ of the type 
 	$$
 	\xi_s(F):=\sum_{\nu\cdot\beta\leq m} s^{m-\nu\cdot\beta} b_\nu D^\nu F(0), \ \ b_\nu\in \mathbb C, \ \ D^\nu F:=\frac{\partial^{\nu_1+\cdots+\nu_n} F}{\partial \xi_1^{\nu_1}\cdots \partial \xi_n^{\nu_n}},
 	$$ 
 	and
 	$$
 	||\xi_s||_s^2:=\sup_{F\in \mathcal O(X_s)} \frac{ |\xi_s(F)|^2}{||F||_s^2}.
 	$$
 	Here the coefficients $s^{m-\nu\cdot\beta}$ are chosen to make sure that $||\xi_s||_s$ depends only on $|s|$ and $\xi_s(P)$ does not depend on $s$. Since $\xi_s$ is holomorphic in $s$, it defines a holomorphic section of the dual of the direct image bundle (see \cite[Definition 2.14]{Wang17}) and Berndtsson's convexity theorem (see \cite{Bern09}, the full proof follows from \cite[Corollary 2.15, Proposition 5.4]{Wang17} and the approximation method in \cite{Bern06}) implies that $||\xi_s||_s$ is subharmonic in $s$ (thus convex increasing in $\log|s|$ since it depends only on $|s|$).  Hence $A(s)$ is decreasing in $|s|$, which gives
 	$$
 	A(1)\leq A(0)= \int_{\mathbb C^n} |P|^2 e^{-\phi_0}.
 	$$
 	The proof is now complete.
 \end{proof}  

\subsection{A generalization of the Second Main Theorem}

Let $M$ be a complex manifold. Recall that a holomorphic mapping 
$$
\mathbb C^*\times M \to M
$$
defined by $(a,m) \to a\cdot m$ is called a $\mathbb C^*$-action if it satisfies that $1\cdot m=m$ and $a\cdot (b\cdot m)=(ab)\cdot m$, for every $a,b\in\mathbb C^*$ and $m\in M$. In case $M=N\times\mathbb C$ is a product, we shall always use the trivial $\mathbb C^*$-action 
$
a\cdot (n,s):=(n,a^{-1}s).
$

\begin{definition}\label{de:cstar-eq} Let $M, N$ be two complex manifolds with $\mathbb C^*$-action. A holomorphic mapping $p:M\to N$ is called $\mathbb C^*$-equivariant if $p$ commutes with the $\mathbb C^*$-actions, i.e. $p(a\cdot m)=a\cdot p(m)$ for every $a\in\mathbb C^*$ and $m\in M$.	In case $M$ is a  submanifold of $N$ and the inclusion mapping $p$ is $\mathbb C^*$-equivariant we say that $M$ is a $\mathbb C^*$-invariant submanifold of $N$.
\end{definition}

\begin{definition}\label{de:cstar-deg-Y}  Let $Y$ (resp. $Y\times \mathbb C$) be a (resp. $\mathbb C^*$-invariant) closed complex submanifold of a complex manifold $X$  (resp. $X_{\mathbb C}$ with $\mathbb C^*$-action). We call  a $\mathbb C^*$-equivariant mapping
	\begin{equation}\label{eq:cstar-deg-Y}
		\mu: X_{\mathbb C} \to X\times\mathbb C
	\end{equation}
	a $\mathbb C^*$-degeneration of $X$ to $Y$ if 
	$\mu$ is biholomorphic over $X\times \mathbb C^*$, 
	\begin{equation}\label{eq:cstar-deg-Y-0}
		\mu_{\mathbb C} : X_{\mathbb C} \to \mathbb C 
	\end{equation}
	is a submersion (means both $\mu_{\mathbb C}$ and $d\mu_{\mathbb C}$ are surjective) and
	\begin{equation}\label{eq:cstar-deg-Y-1}
		\mu_0(X_0) =Y, \ \ \ \mu|_{Y\times \mathbb C} ={\rm Id},
	\end{equation}
	where $X_s$ denotes the $s$-fiber and $\mu_s:=\mu_X|_{X_s}$, where $\mu_X: X_{\mathbb C} \to X$ means the projection to $X$. 		
\end{definition}

\begin{definition}\label{de:cstar-deg-Y-adm}  A $\mathbb C^*$-degeneration $\mu: X_{\mathbb C} \to X\times\mathbb C$ of $X$ to $Y$  is said to be admissible if 
	\begin{equation}\label{eq:cstar-deg-Y-2}
		K_{X_0}\simeq \mu_0^*(K_{X_0}|_{Y}), \ \ K_X|_Y \times \mathbb C\simeq K_{X_{\mathbb C}/\mathbb C}|_{Y\times\mathbb C}.
	\end{equation}								
\end{definition}

\noindent
\textbf{Remark 1.} \emph{All deformations to the normal bundles are admissible. In general, one may also introduce a weight $\beta_j\in\mathbb Z_{>0}$ for each variable $z_j$ , then the associated weighted construction is called deformation to the $\beta$-weighted normal bundle (cone) in \cite[section 2.4]{Abra25}, which has found a nice application \cite{ASQTW25} in the Hironaka theorem very recently. Let us take a closer look for the domain case (see Theorem \ref{th:OT-jet-1}). Different from \eqref{eq:B1}, we put
	\begin{equation}\label{eq:B1-beta}
		\Omega_{\mathbb C}:=\{(\xi, s)\in \mathbb C^n\times\mathbb C: (s^{\beta_1}\xi_1,\cdots, s^{\beta_n} \xi_n)\in \Omega\},
	\end{equation} 
	with $\mathbb C^*$-action
	\begin{equation}\label{eq:B-action-beta}
		a\cdot (\xi,s):=(a^{\beta_1}\xi_1,\cdots, a^{\beta_n}\xi_n, a^{-1}s), \ \ a\in\mathbb C^*, \ \ (\xi,s)\in\Omega_{\mathbb C}.
	\end{equation}
	All deformations to the weighted normal bundles are admissible.}

\medskip

\noindent
\textbf{Remark 2.} \emph{The admissible condition can be used to construct the generalized tautological extension of $F_Y\in H^0(Y, (K_X+L)|_Y)$ (comparing with Definition \ref{de:te}). In fact, as before, for a holomorphic line bundle $L$ over $X$, one may introduce
	\begin{equation}\label{eq:LC-smt}
		L_{\mathbb C}:=\mu_X^*L, \ \ L_s:=L_{\mathbb C}|_{X_s}.
	\end{equation}
	In particular, \eqref{eq:cstar-deg-Y-2} implies that
	\begin{equation}\label{eq:cstar-deg-Y-3}
		K_{X_0}+L_0\simeq \mu_0^*((K_{X_0}+L_0)|_{Y}), \ \ (K_X+L)|_Y \times \mathbb C\simeq (K_{X_{\mathbb C}/\mathbb C}+L_{\mathbb C})|_{Y\times\mathbb C}.
	\end{equation}
	The restriction of the second isomorphism to the zero fiber gives
	\begin{equation}\label{eq:cstar-deg-Y-4}
		(K_X+L)|_Y \simeq (K_{X_0}+L_{0})|_{Y}.
	\end{equation}
	Think of $F_Y\in H^0(Y, (K_X+L)|_Y)$ as a holomorphic section of $(K_{X_0}+L_{0})|_{Y}$, then by \eqref{eq:cstar-deg-Y-3}, 
	\begin{equation}\label{eq:SMT}
		F_{Y,0}:=\mu_0^*(F_Y) \in H^0(X_0, K_{X_0}+L_0)
	\end{equation}
	will be our generalized tautological extension $F_Y$ to $X_0$. The proof of our Second Main Theorem gives the following:}

\begin{theorem}\label{th:OT-main} Let $L$ be a holomorphic line bundle over a complex manifold $X$. Let $\mu: X_{\mathbb C} \to X\times\mathbb C$ be an 
	admissible degeneration of $X$ to its closed complex submanifold $Y$. Assume that $X_{\mathbb C}$ is Stein  (or almost Stein in the sense that $X_{\mathbb C}\setminus Z_{\mathbb C}$ is Stein for some smooth $\mathbb C^*$-invariant divisor $Z_{\mathbb C}$ which intersects $Y\times\mathbb C$ properly) and $(L_{\mathbb C}, e^{-\phi_{\mathbb C}})$ is $S^1$-pseudoeffective in the sense that the local metric potential $\phi_{\mathbb C}$ is psh and $S^1$-invariant with respect to the $\mathbb C^*$-action. Fix $F_Y\in H^0(Y, (K_X+L)|_Y)$. If
	\begin{equation}
		||F_Y||^2:=\int_{X_0} i^{n^2} F_{Y,0} \wedge \overline{F_{Y,0}} \, e^{-\phi_0} <\infty, 
	\end{equation}
	where $\phi_s:=\phi_{\mathbb C}|_{X_s}$, then there exists a holomorphic section $F_X$ of $K_{X}+L$ on $X$ with 
	\begin{equation}\label{eq:C'-estmate}
		F_X|_Y=F_Y, \ \ \int_{X} i^{n^2} F_X \wedge \overline{F_X} \, e^{-\phi_1} \leq	||F_Y||^2. 
	\end{equation}
\end{theorem}

One may also look at jet extension versions of the above theorem. One example is a manifold version of Theorem \ref{th:OT-jet-1} for deformation to the weighted normal bundles. Fix a coordinate $z$ around a given point $x$ in  a complex manifold $X$ with $z(x)=0$ and a weight 
$$
\beta=(\beta_1,\cdots, \beta_n)\in\mathbb Z_{>0}^n.
$$
Let $\mu:X_{\mathbb C} \to X\times\mathbb C$ be the deformation to $\beta$-weighted normal bundle of $x$, i.e. $\mu$ is given by
$$
(\xi,s)\mapsto (z,s):=(s^\beta\xi, s):=(s^{\beta_1}\xi_1, \cdots, s^{\beta_n}\xi_n, s)
$$
around $\{0\}\times \mathbb C$ and the $\mathbb C^*$-action on $X_{\mathbb C}$ is given by \eqref{eq:B-action-beta}. Let $L$ be a holomorphic line bundle over $X$, its degeneration $L_0$ to the central fiber $X_0=\mathbb C^n$ defined in \eqref{eq:LC-smt} is trivial. Thus one may think of $f\in H^0(X_0, K_{X_0}+L_0)$ as a holomorphic $n$-form $f(\xi)d\xi$. Then the $\mathbb C^*$-action on $X_{\mathbb C}$ induces a $\mathbb C^*$-action on $H^0(X_0, K_{X_0}+L_0)$:
$$
a\cdot f(\xi)d\xi:= f(a^\beta\xi) d(a\cdot \xi)= f(a^\beta\xi) a^{|\beta|} d\xi,
$$
where $d\xi:=d\xi_1\wedge\cdots\wedge d\xi_n, \,|\beta|:=\beta_1+\cdots+\beta_n$.

\begin{definition}\label{de:beta-jet} Let $m$ be a non-negative integer, we call $P\in H^0(X_0, K_{X_0}+L_0)$ a $\beta$-weighted degree $m$ section if 
	$a\cdot P=a^{m+|\beta|} P
	$
	for every $a\in\mathbb C^*$.  Such $P$ is called the $\beta$-weighted degree $m$-jet of $F\in H^0(X, K_X+L)$ (we write $P={\rm jet}^m_\beta F$) if $\tilde F:=s^{-m-|\beta|} \mu_X^* F$ satisfies
	\begin{equation}\label{eq:beta-jet}
		\tilde F|_{X_0} =P,
	\end{equation} 
	i.e. $\tilde F$ extends to a holomorphic section of $K_{X_{\mathbb C}/\mathbb C}+L_{\mathbb C}$  over $X_{\mathbb C}$ whose restriction to $X_0$ is $P$. 
\end{definition}

The proof of Theorem \ref{th:OT-jet-1} gives:

\begin{theorem}\label{th:OT-jet-2} With the notations above, assume that $L_{\mathbb C}$ has a psh metric potential $\phi_{\mathbb C}$, which is $S^1$-invariant with respect to the $\mathbb C^*$-action on $X_{\mathbb C}$. Assume further that $X$ is $x$-Stein in the sense that there exists a divisor $Z$ of $X$ such that $X\setminus Z$ is Stein and $x\notin Z$. Then for every $\beta$-weighted degree $m$ section $P\in H^0(X_0, K_{X_0}+L_0)$ with  
	\begin{equation}\label{eq:OT-jet-21}
		||P||^2:=  \int_{X_0} i^{n^2} P\wedge \bar P\, e^{-\phi_0}<\infty,
	\end{equation} 
	where $\phi_s:=\phi_{\mathbb C}|_{X_s}$, there exists $F_X\in H^0(X, K_X+L)$ with ${\rm jet}^m_\beta F_X =P$ and
	\begin{equation}\label{eq:OT-jet-22}
		\int_X i^{n^2} F_X \wedge \overline{F_X}\, e^{-\phi_1}\leq ||P||^2.
	\end{equation}
\end{theorem}

\subsection{Proof of Theorem B'} We shall use the idea in the proof of \cite[Theorem 4.2]{WN19}. Put
		$$
		A_m:= \left\lbrace \nu(F): F\in H^0(X, mL), \,  \sup_X |F|^2e^{-m\phi}= 1 \right\rbrace.
		$$
	If $v+(1,\cdots, 1) \in \Delta_\phi^0$, then there exist finitely many bounded-norm sections $ s_j\in H^0(X,k_jL)$ such that $v+(1,\cdots, 1) $ lies in the interior of the convex hull of $\nu(s_j)/k_j$. Choose a common multiple, say $m_0$ of the $k_j$, then we know that 
		$$
		\Delta_{m_0}:=\text{convex hull of $A_{m_0}/m_0$} 
		$$
		also contains $v+(1,\cdots, 1)$ as an interior point. Let us look at the finite set 
		$$
		A:=A_{m_0} \cup \{\nu(F): F\in H^0(X, K_X+L)\setminus\{0\}\}.
		$$
		For each $\alpha\in A$, put
		$$
		M_\alpha:=\{\gamma\in \mathbb N^n: \alpha\prec\gamma\}.
		$$
		By Dickson's Lemma (see \cite[Lemma 1.2.6]{GP08}), there exists a finite set $B_\alpha\subset M_\alpha$ such that for every $\gamma\in M_\alpha$ there exists $\gamma'\in B_\alpha$ with
		\begin{equation}\label{eq:nat}
			\gamma'_j \leq\gamma_j, \ \forall \   \ 1\leq j\leq n.
		\end{equation}
		Put
		$$
		M:=A \cup \bigcup_{\alpha\in A} B_\alpha.
		$$ 
		Since $M$ is a finite set, \cite[Lemma 1.2.11]{GP08} gives a $\beta\in \mathbb Z_{>0}^n$ such that for all $\alpha, \gamma' \in M$, 
		\begin{equation}\label{eq:nat1}
			\alpha \prec \gamma' \Rightarrow \alpha\cdot\beta <\gamma'\cdot \beta.
		\end{equation}
		Now if $\alpha\in A$ and $\alpha\prec \gamma$, then we can choose $\gamma'\in B_\alpha$ satisfying \eqref{eq:nat}, thus \eqref{eq:nat1} implies that
		\begin{equation}\label{eq:nat2}
			\alpha\cdot\beta <\gamma'\cdot \beta \leq \gamma \cdot\beta.
		\end{equation}
		Consider the deformation to the $\beta$-weighted normal bundle (see Remark 1 after Definition \ref{de:cstar-deg-Y-adm}) and pick a basis $\{F_\alpha\}_{\alpha\in A_{m_0}}$ of 
		$
		\left\lbrace F\in H^0(X, m_0 L): \sup_X |F|^2e^{-m_0 \phi} <\infty \right\rbrace
		$
		such that $$F_\alpha(z)=c_\alpha z^\alpha+\sum_{\alpha\prec\gamma} c_\gamma z^\gamma \ \text{with}\ c_\alpha\neq 0 \ \text{near $z(x)=0$ and} \  \sup_X |F_\alpha|^2e^{-m_0\phi}=1.
		$$ 
		Then \eqref{eq:nat2} implies that
		$$
		F_\alpha(s^{\beta_1}\xi_1, \cdots, s^{\beta_n} \xi_n) =c_\alpha s^{\alpha \cdot\beta} (\xi^\alpha +O(|s|))
		$$
		near $s=0$. Hence
		$$
		\phi_{\mathbb C}(\xi,s):=\frac1{m_0} \log \left( \max_{\alpha\in A_{m_0}} \big|\frac{F_\alpha(s^{\beta_1}\xi_1, \cdots, s^{\beta_n} \xi_n)}{s^{\alpha \cdot\beta} }\big|^2 \right)
		$$
		is $S^1$-psh on the total space $X_{\mathbb C}$ of the deformation to the $\beta$-weighted normal bundle with
		$$
		\phi_0(\xi)=\lim_{s\to 0} \phi_{\mathbb C}(\xi,s) =\frac1{m_0} \log  \left( \max_{\alpha\in A_{m_0}} \big|c_\alpha \xi^\alpha\big|^2 \right)
		$$
		being toric. Since convex hull of $A_{m_0}/m_0$ contains $v+(1\,\cdots, 1)$ as an interior point, we have
		$$
		\int_{\mathbb C^n=X_0} i^{n^2} |\xi^v|^2 d\xi \wedge \overline{d\xi}\, e^{-\phi_0(\xi)}<\infty.
		$$
		Apply Theorem \ref{th:OT-jet-2}, we obtain $F_X\in H^0(X, K_X+L)$ with 
		\begin{equation}\label{eq:nat3}
		F_X(z)=\left(z^v + 
	\sum_{v\cdot\beta < \gamma\cdot \beta} c_\gamma z^\gamma\right)dz\otimes\bm e 
		\end{equation}
		around $x$ and
		$$
		\int_X i^{n^2} F_{X} \wedge \overline{F_X}\, e^{-\phi_1} \leq \int_{\mathbb C^n=X_0} i^{n^2}  |\xi^v|^2 d\xi \wedge \overline{d\xi}\, e^{-\phi_0(\xi)}<\infty.
		$$
		We claim that $\nu(F_X)=v$: otherwise $\nu(F_x)\prec v$ implies that $\nu(F_x)\cdot \beta < v\cdot\beta$ by \eqref{eq:nat2}, which contradicts \eqref{eq:nat3}.
		Note that our choice of $F_\alpha$ implies that $\phi_1=\phi(\cdot,1)\leq \phi$, thus we have
		$$
		\int_X i^{n^2} F_{X} \wedge \overline{F_X}\, e^{-\phi} <\infty.
		$$
		Now it suffices to take $F$ as $F_X$.

\subsection{A remark on Demailly's Seshadri-type jet-separation theorem}

If we apply Theorem B to the infinitesimal Okounkov bodies then a recent result  of Fulger and Lozovanu gives:

\begin{theorem}\label{th:E} Let $L$ be a positive line bundle over an $n$-dimensional compact complex manifold $X$. For $x\in X$, $1\leq j\leq n$, let (see \cite{FL25}) 
	\begin{equation}\label{eq:Seshadri-jth}
		\epsilon_j(L;x):=\min\left\lbrace t>0: {\rm dim} (B_+(\mu^*L-tE) \cap E) \geq j-1\right\rbrace, \ \dim\emptyset:=-\infty,
	\end{equation} 
	be the $j$-th infinitesimal successive minima of $L$ at $x$, where $\mu$ denotes the blowup of $x$ with exceptional divisor $E$ and $B_+$ denotes the augmented base locus (or non-K\"ahler locus by \cite[Theorem 6.3]{Tos18}). If
	\begin{equation}\label{eq:ThmD}
		\frac{k+1}{\epsilon_1(L;x)}+ \frac1{\epsilon_2(L;x)}+\cdots + \frac1{\epsilon_n(L;x)}<1
	\end{equation}
	then global holomorphic sections of $K_X+L$ generate all $k$-jets at $x$.
\end{theorem}

\medskip

\noindent
\textbf{Remark.} \emph{From \eqref{eq:Seshadri-jth}, we know that 
	$$
	\epsilon_1(L;x) \leq \cdots \leq \epsilon_n(L;x)
	$$
	and $\epsilon_1(L;x)$ is equal to the classical Seshadri constant of $L$ at $x$. Hence the above theorem is better than the classical Seshadri type jet-separation result \cite[Proposition 6.8 (a)]{Dem92}.}

\begin{proof}[Proof of Theorem \ref{th:E}] Put $\epsilon_j:=\epsilon_j(L;x)$, by \cite[Theorem 1.2]{FL25}, the generic straightened up infinitesimal Okounkov body $\Delta$ (i.e. the $\beta=(1,\cdots, 1)$ case, see the remark after Definition \ref{de:betaorder}) at $x$ contains the simplex 
	$$
	\left\lbrace \alpha\in \mathbb R_{\geq 0}^n: \frac{\alpha_1}{\epsilon_1} +\cdots+\frac{\alpha_n}{\epsilon_n} \leq 1 \right\rbrace.
	$$
Since $\epsilon_1\leq \cdots\leq \epsilon_n$,  \eqref{eq:ThmD} implies that $v+(1,\cdots, 1)$ lies in the interior of $\Delta$ for all $v$ with $|v|\leq k$. Thus the theorem follows from Theorem B.
\end{proof}

\subsection{Proof of Theorem C}

The main idea is to apply our Second Main Theorem to the deformation to $T_\infty \mathbb P^1$. With respect to the non-homogeneous coordinate, the deformation mapping is given by
$$
\mu(\xi,s)=(z,s)=(\xi/s, s)
$$
so that $\mu$ fixes $(\infty, s)$. The mapping $\mu$ suggests us to look at the geodesic ray	
\begin{equation}\label{eq:ThC-phiC}
\phi_{\mathbb C}(\xi,s):=\max\left\lbrace \phi_K\left(\frac\xi s\right)+\log|s^2|, 0\right\rbrace,
\end{equation}
where $\phi_K$ is the equilibrium potential \eqref{eq:equi-po}, so that $\phi_1(\xi):=\phi_{\mathbb C}(\xi,1)= \phi_K(\xi)$ and
$$
\phi_0(\xi):=\limsup_{s\to 0}\phi_{\mathbb C}(\xi,s):=\max\left\lbrace0, \log\frac{|\xi|^2}{C(K)}\right\rbrace.
$$
The Second Main Theorem  implies that (see Lemma \ref{le:ThC} below)
	\begin{equation}\label{eq:ThC-Bergman}
\inf_{P\in \mathcal P_m} \int_{\mathbb C} |P|^2 e^{-(m+2)\phi_K}  \leq  \inf_{P\in \mathcal P_m} \int_{\mathbb C} |P|^2 e^{-(m+2)\phi_0}.
\end{equation}
Since $\phi_0$ is rotation invariant, we have
$$
\inf_{P\in \mathcal P_m} \int_{\mathbb C} |P|^2 e^{-(m+2)\phi_0}= \int_{\mathbb C} |\xi^m|^2 e^{-(m+2)\phi_0} =\frac{\pi (m+2)C(K)^{m+1} }{m+1}, 
$$
which gives the upper bound in \eqref{eq:thC-new}. To prove the lower bound, note that $K$ is connected, thus the unbounded component, say $\Omega$, of $\mathbb P^1\setminus K$ is simply connected and there exists a conformal mapping
$$
F:\{|z|^2>C(K)\} \to \Omega
$$
such that $F(\infty)=\infty$, $F'(\infty)=1$ and $\phi_K(F(z))=\log \frac{|z|^2}{C(K)}$. Hence for $P\in \mathcal P_m$, we have
$$
\int_{\mathbb C} |P|^2 e^{-(m+2)\phi_K} \geq \int_K |P|^2+C(K)^m\int_{|z|^2>C(K)} \frac{|P(F(z)) F'(z)|^2}{|z^m|^2} e^{-2 \log \frac{|z|^2}{C(K)}}.
$$
Apply the submean inequality for the circles $|z|=r$, $r^2>C(K)$, around $\infty$ to 
$$
p(z):=\frac{P(F(z)) F'(z)}{z^m}, 
$$
we obtain (note that $p(\infty)=1$)
$$
\int_{|z|^2>C(K)} \frac{|P(F(z)) F'(z)|^2}{|z^m|^2} e^{-2 \log \frac{|z|^2}{C(K)}} \geq |p(\infty)|^2 \int_{|z|^2>C(K)}e^{-2 \log \frac{|z|^2}{C(K)}} = \pi C(K).
$$
Hence the lower bound follows.

\begin{lemma}\label{le:ThC} \eqref{eq:ThC-Bergman} holds.
\end{lemma}

\begin{proof} The potential $\phi_{\mathbb C}$ in \eqref{eq:ThC-phiC} defines a psh weight for the line bundle $\mathcal O(1)\times\mathbb C$ over $\mathbb P^1\times\mathbb C$ and is $S^1$-invariant in the sense that
	$$
	\phi_{\mathbb C} (\xi,s)=\phi_{\mathbb C} (e^{i\theta}\xi, e^{i\theta}s).
	$$
Since $K_{\mathbb P^1}=-2\mathcal O(1)$, one may identify $H^0(\mathbb P^1, K_{\mathbb P^1}+\mathcal O(m+2))$ with the space $L_m$ of degree $\leq m$ polynomials on $\mathbb C$. Then for the metric potential $\phi_{s}:=\phi_{\mathbb C} (\cdot,s)$, we have
$$
K_{(m+2)\phi_s} = \sup_{f\in L_m} \frac{|f|^2e^{-m\phi_s} \omega_s}{\int_{\mathbb C}|f|^2e^{-m\phi_s} \omega_s}, \ \  \omega_s(\xi):=e^{-2\phi_s(\xi)}\, \frac i2 \, d\xi\wedge d\bar \xi, 
$$
for the Bergman kernel of $H^0(\mathbb P^1, K_{\mathbb P^1}+\mathcal O(m+2))$ at $\infty\in\mathbb P^1$. Let
$$
\omega_{\rm FS}:=e^{-2 \log(1+|\xi|^2) }\, \frac i2 \, d\xi\wedge d\bar \xi
$$
be the Fubini-Study volume form on $\mathbb P^1$. Then we have
$$
\frac{K_{(m+2)\phi_s}(\infty)}{\omega_{\rm FS}(\infty)}=C(K)^{m+2} \frac1{\inf_{P\in \mathcal P_m} \int_{\mathbb C} |P|^2 e^{-(m+2)\phi_s}}.
$$
Berndtsson's theorem in \cite{Bern09} implies that $
K_{(m+2)\phi_s}(\infty)$ is subharmonic in $s\in \mathbb C$. Since it is also $S^1$-invariant, we know that $K_{(m+2)\phi_s}(\infty)$ is increasing in $|s|$, thus $\inf_{P\in \mathcal P_m} \int_{\mathbb C} |P|^2 e^{-(m+2)\phi_s}$ is decreasing in $|s|$. Hence we have
$$
\inf_{P\in \mathcal P_m} \int_{\mathbb C} |P|^2 e^{-(m+2)\phi_1}  \leq  \inf_{P\in \mathcal P_m} \int_{\mathbb C} |P|^2 e^{-(m+2)\phi_0}
$$
and \eqref{eq:ThC-Bergman} follows from $\phi_1=\phi_K$.
\end{proof}

\subsection{Effective polynomial approximation for the Riemann mapping}

Let $K$ be a connected compact set in $\mathbb C$ that contains at least two points. Then there exists a unique monic degree $m$ polynomial, say $P_m$, such that
	\begin{equation}\label{eq:ThC-Pm}
	\inf_{P\in \mathcal P_m} \int_{\mathbb C} |P|^2 e^{-(m+2)\phi_K}  = \int_{\mathbb C} |P_m|^2 e^{-(m+2)\phi_K}.
\end{equation}

\begin{definition}\label{de:extreaml-m} We call $P_m$ the $m$-th Bergman-extremal polynomial associated to $K$.
\end{definition}

\begin{theorem}\label{th:extreaml-m} Let $\Omega$ be the unbounded component of $\mathbb P^1\setminus K$. Then for every $m\geq 0$, we have
	\begin{equation}\label{eq:extreaml-m}
I_m:=	\int_{\Omega}  \big|  \frac{P_m}{f_K^m f'_K} -1\big|^2 e^{-2\phi_K}  \, \frac i2 \,df_K \wedge d\bar f_K \leq \frac{\pi C(K)}{m+1},
\end{equation}
	where $f_K$ denotes the unique conformal mapping from $\Omega$ onto $\{|z|^2>C(K)\}$ with $f_K(\infty)=\infty$ and $f'_K(\infty)=1$. In particular, 
	\begin{equation}\label{eq:extreaml-m1}
	\frac{P_{m+1}}{P_m} \to f_K
\end{equation}
locally uniformly on $\Omega$ as $m\to \infty$.
\end{theorem}

\begin{proof} Since $\phi_K=\log|(f^2_K)/C(K)|$, we have
	$$
	\int_{\Omega}  \big|  \frac{P_m}{f_K^m f'_K} \big|^2 e^{-2\phi_K}  \, \frac i2 \,df_K \wedge d\bar f_K = \frac{\int_{\Omega} |P_m|^2 e^{-(m+2)\phi_K}}{C(K)^m} \leq \frac{\pi C(K)(m+2)}{m+1}
	$$
by Theorem C and
$$
\int_{\Omega}  e^{-2\phi_K}  \, \frac i2 \,df_K \wedge d\bar f_K =\int_{|z|^2>C(K)} \frac{C(K)^2}{|z|^4} =\pi C(K). 
$$
Thus
$$
I_m \leq \frac{\pi C(K)(m+2)}{m+1}+ \pi C(K) -2\,{\rm Re}\, 	\int_{\Omega}  \frac{P_m}{f_K^m f'_K} e^{-2\phi_K}  \, \frac i2 \,df_K \wedge d\bar f_K 
$$
Let $F$ be the inverse of $f_K$, then  $\frac{P_m(F)F' }{z^m} = 1+ \sum_{j\geq 1} c_j/z^j$ gives
$$
	\int_{\Omega}  \frac{P_m}{f_K^m f'_K} e^{-2\phi_K}  \, \frac i2 \,df_K \wedge d\bar f_K =\int_{|z|^2>C(K)} \frac{P_m(F)F'\, C(K)^2}{z^m |z|^4} = \pi C(K).
$$
Thus we have
$$
I_m\leq  \frac{\pi C(K)(m+2)}{m+1}- \pi C(K) = \frac{\pi C(K)}{m+1}.
$$
The proof is now complete.
\end{proof}
		
		\section{Appendix: around a theorem of Berndtsson}
		
		\subsection{A remark on Berndtsson's theorem}

		We shall give a slightly different proof of Theorem \ref{th:Bern20} of Berndtsson. The main motivation is to avoid the induction process (see \cite[page 2368-2369]{Bern20}) in his original proof. Let $E$ be a finite dimensional $\mathbb C$-vector space. Let $||\cdot||_\tau$, ${\rm Re}\, \tau>0$, be a family of Hermitian norms on $E$ that depends only on $t={\rm Re}\, \tau$. We shall always assume that the limit 
		$$
		||\cdot||_0:=\lim_{t\to 0} ||\cdot||_t
		$$
		exists and gives a Hermitian norm on $E$, moreover, there exists $\lambda>0$ such that
		$$
		e^{-\lambda t}||\cdot||^2_0\leq  ||\cdot||^2_t\leq ||\cdot||^2_0, \ \ \forall \ t>0.
		$$
		The example in mind is $E=H^0(X, K_X+L)$ with the $t$-norm in \eqref{eq:salpha-intro}. We know that $||\cdot||_\tau$ defines a Hermitian metric on the product bundle $E\times \mathbb H$, where
		$$
		\mathbb H:=\{\tau\in\mathbb C: {\rm Re}\, \tau>0\}
		$$
		denotes the right half plane. Following \cite{Bern20}, we say that $||\cdot||_\tau$ is positively curved if the dual metric $||\cdot||_{*\tau}$ is \emph{negatively curved} on $V\times \mathbb H$, where $V=E^*$ denotes the dual space of $E$, i.e. for every holomorphic mapping
		$$
		f: U \to V, \ \ \text{$U$ is an open set in $\mathbb H$},
		$$
		the function $\log||f(\tau)||_{*\tau}$ is subharmonic for $\tau\in U$. Theorem \ref{th:Bern20} is a special case of the following theorem of Berndtsson.
		
		\begin{theorem}\label{th:Bern20-E} With the notations above, if $||\cdot||_\tau$ is positively curved then for every $F\in E$ and $\alpha\in\mathbb R$ there exists an $F_\alpha\in E$ such that
			$$
			\int_0^\infty ||F_\alpha-F||_t^2 e^{t\alpha}\,dt<\infty
			$$
			and 
			\begin{equation}\label{eq:Bern20-E}
				||F_\alpha||_0^2 \leq \liminf_{t\to\infty} e^{\alpha  t} ||F||_t^2.
			\end{equation}			
		\end{theorem}
		
		We need a few lemmas to prepare the proof.

		\begin{lemma} Assume that $||\cdot||_\tau$ is positively curved on $E\times\mathbb H$. Then  the limit
			$$
			\alpha(u):=\lim_{t\to\infty} \frac{\log(||u||_{*t}^2)}{t}, \ \ u\in V=E^*,
			$$
			exists (we call it the \emph{jumping number} of $u$).
		\end{lemma}
		
		\begin{proof} From the above definition of positive curved metric, we know that $\log(||u||_{*\tau}^2)$ is a subharmonic function that depends only on $t={\rm Re}\, \tau$. Thus $\log(||u||_{*t}^2)$ is convex in $t>0$, from which the lemma follows.	
		\end{proof}
		
		Put
		$$
		V_\alpha:=\{u\in V: \alpha(u) \leq \alpha\}.
		$$
		One may observe that each $V_\alpha$ is a $\mathbb C$-linear subspace of $V$. Since $V$ has finite dimension, we know that the set of jumping numbers has only finite elements, let us write it as
		$$
		\{\alpha(u): u\in V\setminus\{0\}\}=\{\alpha_1, \cdots, \alpha_N\}, \ \ \alpha_1<\cdots<\alpha_N.
		$$
		Then we obtain a filtration of $V$
		$$
		\{0\} \subset V_{\alpha_1} \subset \cdots \subset V_{\alpha_N}=V.
		$$
		The dual filtration on $E$ is defined by
		\begin{equation}\label{eq:Bo-dual-E}
			\mathcal F_\alpha:=\{F\in E: u(F)=0, \ \forall \ u\in V_\alpha \}.
		\end{equation}
		Then Berndtsson proved the following result in \cite{Bern20}.
		
		\begin{lemma}\label{le:key1}  With the notation above, we have
			$$
			F\in \mathcal F_\alpha \Leftrightarrow \int_0^\infty ||F||_t^2 e^{t\alpha}\, dt <\infty\Leftrightarrow \limsup_{t\to\infty} \frac{\log (||F||_t^2)}{t} \leq -\min\{\alpha_j: \alpha_j>\alpha\}.
			$$	
		\end{lemma}	
		
		To prove this lemma, we need the following flat reduction lemma of Berndtsson (see Proposition 2.2 in \cite{Bern20}).

		\begin{lemma}\label{le:key2}  With the notation above, there  exists a unique flat metric $||\cdot||_{*t, \infty}$ on $V\times \mathbb H$ such that $||\cdot||_{*0, \infty} =||\cdot||_{*0} $, $||\cdot||_{*t, \infty} \geq ||\cdot||_{*t}$ and
			$$
			\lim_{t\to\infty} \frac{\log(||u||_{*t}^2)}{t}=  \lim_{t\to\infty} \frac{\log(||u||_{*t,\infty}^2)}{t}, \ \forall \ u\in V.
			$$
			Moreover, there exists a basis, say $\{u_j\}_{1\leq j\leq \dim V}$, of $V$ such that
			$$
			||\sum a_j u_j||_{*t,\infty}^2 =\sum_{j=1}^{\dim V} |a_j|^2 e^{t\lambda_j}, \ \ \  \lambda_1 \leq \lambda_2 \leq\cdots \leq \lambda_{\dim V},
			$$
			where $\{\lambda_j: 1\leq j \leq \dim V\}=\{\alpha_1, \cdots, \alpha_N\}$.
		\end{lemma}	
		
		\begin{proof}[Proof of Lemma \ref{le:key1}] From the proof of Theorem 1.1 in \cite{Bern20}, we know that the only non-trivial part is to show that $F\in \mathcal F_\alpha$ implies
			\begin{equation}\label{eq:F-0}
				\limsup_{t\to\infty} \frac{\log(||F||_t^2)}{t} \leq -\hat \alpha, \ \ \hat\alpha:=\min\{\alpha_j: \alpha_j>\alpha\}.
			\end{equation}
			By Lemma \ref{le:key2}, we know that $u_j\in V_\alpha$ if and only if $\lambda_j \leq \alpha$. To prove \eqref{eq:F-0}, it suffices to show that for every $\varepsilon>0$ there exists $t_0>0$ such that for every $F\in \mathcal F_\alpha$, we have
			\begin{equation}\label{eq:F-1}
				||F||_t^2 \leq e^{-t(\hat \alpha-\varepsilon)} ||F||_0^2, \ \ \forall\ t\geq t_0.
			\end{equation}
			Consider the new metric $|\cdot|_{*\tau}$ on $V\times \mathbb H$ defined by
			\begin{equation}\label{eq:newstarnorm}
				|\sum a_j u_j|_{*\tau}:=||\sum a_j e^{-\tau\lambda_j/2}u_j||_{*\tau},
			\end{equation}
			the new metric $|\cdot|_{*\tau}$ on $V$ may depend also on the imaginary part of  $\tau$, but the associated metric on $\det V$ defined by
			$$
			|u_1\wedge\cdots \wedge u_{\dim V}|_{*\tau} := |e^{-\tau (\lambda_1+\cdots+\lambda_{\dim V})/2}|\cdot ||u_1\wedge\cdots \wedge u_{\dim V}||_{*\tau}
			$$
			depends only on $t={\rm Re}\,\tau$, i.e.
			$$
			|u_1\wedge\cdots \wedge u_{\dim V}|_{*\tau} = |u_1\wedge\cdots \wedge u_{\dim V}|_{*t},  \ \text{for} \ t= {\rm Re}\,\tau.
			$$
			The corresponding flat metric
			\begin{equation}\label{eq:flat-constant}
				|u_1\wedge\cdots \wedge u_{\dim V}|_{*t,\infty} := e^{-t (\lambda_1+\cdots+\lambda_{\dim V})/2}||u_1\wedge\cdots \wedge u_{\dim V}||_{*t,\infty} =1
			\end{equation}
			is a constant metric. Denote by $e^{ {\rm Re}\,\tau\cdot \mu_j(\tau)}$ ($\mu_1(\tau) \leq \mu_2(\tau) \leq \cdots \leq 
			\mu_{\dim V}(\tau) \leq 0$) the eigenvalues of $|\cdot|^2_{*\tau}$ with respect to $|\cdot|^2_{*0}$. Lemma \ref{le:real} below implies that each $\mu_j(\tau)$ depends only on $t= {\rm Re}\,\tau$. By the flat bundle construction in Lemma \ref{le:key2} (see \cite[the proof of Proposition 2.2]{Bern20}), we have
			\begin{equation}\label{eq:flat-unique}
				\lim_{t\to\infty} \frac{\log (|u_1\wedge\cdots \wedge u_{\dim V}|^2_{*t,\infty})}{t} = \lim_{t\to\infty} \frac{\log (|u_1\wedge\cdots \wedge u_{\dim V}|^2_{*t})}{t}.
			\end{equation} 
			Note that the left hand side is zero by \eqref{eq:flat-constant}. By the negative curvature property, we also know that $\log (|u_1\wedge\cdots \wedge u_{\dim V}|^2_{*t})$ is convex in $t$, thus
			$$
			\frac{\log (|u_1\wedge\cdots \wedge u_{\dim V}|^2_{*t}) -\log( |u_1\wedge\cdots \wedge u_{\dim V}|^2_{*0} )}{t}  = \sum_j \mu_j(t),
			$$  
			defines an increasing function that converges to the right hand side of 
			\eqref{eq:flat-unique} as $t\to\infty$. Hence 
			$$
			\lim_{t\to\infty} \sum_j\mu_j(t) =0.
			$$
			Thus for every $\varepsilon>0$, there exists $t_0>0$ such that
			$$
			\sum_j\mu_j(t) \geq -\varepsilon, \ \ \forall \ t\geq t_0.
			$$
			Since all $\mu_j(t) \leq 0$, we must have
			$$
			\mu_j(t) \geq   -\varepsilon, \ \ \forall \ 1\leq j\leq \dim V,  \   \ t\geq t_0,
			$$
			which gives
			$$
			|u|^2_{*t} \geq e^{-\varepsilon t} |u|^2_{*0}, \ \ \forall \ u\in V,  \   \  t \geq t_0.
			$$
			Hence we have
			\begin{equation}\label{eq:F-2}
				||\sum a_j e^{-t\lambda_j/2}u_j||^2_{*t} \geq  e^{-\varepsilon t} \sum |a_j|^2 , \ \ \forall \ t\geq t_0.
			\end{equation}
			Denote by $\{F_j\}$ the dual basis of $E$ with respect to $\{u_j\}$, for $F=\sum c_j F_j$, \eqref{eq:F-2} gives
			\begin{equation}\label{eq:F-3}
				||F||_t^2  = \sup \frac{|\sum c_j a_j e^{-t\lambda_j/2}|^2 }{||\sum a_j e^{-t\lambda_j/2}u_j||^2_{*t}} \leq e^{\varepsilon t}  \sum |c_j|^2 e^{-t \lambda_j}.
			\end{equation}
			Note that $F \in \mathcal F_\alpha$ if and only if $c_j=0$ for all $j$ with $\lambda_j \leq
			\alpha$, thus we can write
			$$
			F=\sum_{\lambda_j\geq \hat \alpha} c_j F_j,
			$$
			for $F\in \mathcal F_\alpha$; then \eqref{eq:F-3} gives
			$$
			||F||^2_{t} \leq e^{\varepsilon t} \sum_{\lambda_j \geq \hat \alpha} e^{-t\lambda_j}|c_j|^2  \leq e^{-t(\hat \alpha-\varepsilon)} ||F||_0^2 , \ \ \forall \ t\geq t_0. 
			$$
			The proof is complete now.	
		\end{proof}

		\begin{lemma}\label{le:real} The eigenvalues of $|\cdot|^2_{*\tau}$ in \eqref{eq:newstarnorm} with respect to $|\cdot|^2_{*0}$  depends only on $t= {\rm Re}\,\tau$.
		\end{lemma}
		
		\begin{proof}
			Let $A_t$ be the self-adjoint operator on $V$ such that 
			$$
			||u||^2_{*t}=(A_t u, u)_{*0}, \ \forall \ u\in V,
			$$
			where $(\cdot, \cdot)_{*0}$ denotes the inner product associated to the Hilbert norm $||\cdot||_{*0} = |\cdot|_{*0}$. Write $\tau =t+is$, by \eqref{eq:newstarnorm}, we have
			$$
			|\sum a_j u_j|^2_{*\tau} =  \left(A_t U_s (\sum a_j e^{-t\lambda_j/2} u_j), U_s (\sum a_j e^{-t\lambda_j/2} u_j) \right)_{*0},
			$$ 
			where $U_s$ is the operator on $V$ defined by $U_s(u_j)=e^{-is\lambda_j/2} u_j$. Note that each $U_s$ is unitary with respect to the $||\cdot||_{*0}$ norm. Thus we have
			$$
			|\sum a_j u_j|^2_{*\tau} =  \left(U_s^{-1}A_t U_s (\sum a_j e^{-t\lambda_j/2} u_j),  \sum a_j e^{-t\lambda_j/2} u_j \right)_{*0}.
			$$
			Define $B_t$ such that $B_t(u_j)= e^{-t\lambda_j/2} u_j$, we  know that $B_t$ is self-adjoint with respect to $||\cdot||_{*0}$ and $B_tU_s= U_sB_t$, thus
			$$
			|u|^2_{*\tau} = (U_s^{-1}B_tA_tB_t U_s u, u)_{*0}.
			$$
			Notice that, the eigenvalues of $|\cdot|^2_{*\tau}$ in \eqref{eq:newstarnorm} with respect to $|\cdot|^2_{*0}$ are precisely the eigenvalues of $U_s^{-1}B_tA_tB_t U_s$. Hence they depend only on $t$.
		\end{proof}
		
		Now we are ready to prove Theorem \ref{th:Bern20-E}.

		\begin{proof}[Proof of Theorem \ref{th:Bern20-E}] For $F\in E$, recall that the quotient norm of the quotient class $[F]\in E/\mathcal F_\alpha$ is defined by
			\begin{equation}\label{eq:quotient-norm-E}
				||[F]||_t^2:=\inf \{ ||F_\alpha||_t^2: F_\alpha-F\in \mathcal F_\alpha\}.
			\end{equation}
			By the definition of $\mathcal F_\alpha$ in \eqref{eq:Bo-dual-E}, we know that $V_\alpha$ is the dual of $E/\mathcal F_\alpha$. Since $||\cdot||_{t}$ has positive curvature, we know that $||\cdot||_{*t}$ has negative curvature.  Thus $\log||u||_{*t}$ is convex in $t$ for all $u\in V$. Thus $u\in V_\alpha$ if and only if $||u||_{*t}^2 e^{-t\alpha}$ is decreasing in $t$.  Since $\mathcal F_\alpha$ is the dual of $V_\alpha$, we know that $e^{\alpha t}||[F]||^2_t $ is increasing in $t$ for $F\in E$. By the definition of the quotient norm \eqref{eq:quotient-norm-E}, we know that there exists $F_\alpha\in E$ such that $F_\alpha-F\in \mathcal F_\alpha$ and
			$$
			||F_\alpha||_0^2 \leq \lim_{t\to\infty} e^{\alpha t}||[F]||^2_t \leq \liminf_{t\to\infty} e^{\alpha t}||F||_t^2.
			$$
			Now we know that Theorem \ref{th:Bern20-E} follows from Lemma \ref{le:key1}.
		\end{proof}
		
		\subsection{Sharpness criterion of the First Main Theorem}
		
		From the above proof of Theorem \ref{th:Bern20-E} we know that the First Main Theorem is sharp if and only if
		\begin{equation}\label{eq:sharpA1}
			||[F]||_0^2=e^{\alpha t} ||[F]||_t^2 \ \  \text{for all} \ t\geq 0  \ \text{and}\  ||[F]||_0^2=\liminf_{t\to \infty} e^{\alpha t} ||F||_t^2.
		\end{equation}
		From the proof of Proposition 2.2 in \cite{Bern20}, we know that the $||\cdot||_t$ norms define a flat metric  if and only if there exists a basis $F_j$ of $H^0(X, \mathcal O(K_X+L)\otimes \mathcal I(\phi))$ such that
		\begin{equation}\label{eq:sharpA2}
			||\sum c_jF_j||_t^2=\sum |c_j|^2e^{-\lambda_j t}, \ \  \forall \ c_j\in \mathbb C,
		\end{equation}
		where $\lambda_1 \leq \cdots \leq \lambda_N$ are constants, in which case we have
		\begin{equation}\label{eq:sharpA3}
			H_\alpha=S_\alpha=\mathcal F_\alpha={\rm span}_{\mathbb C}\{F_j: \lambda_j>\alpha\},
		\end{equation}
		and 
		\begin{equation}\label{eq:sharpA4}
			||[F]||_t^2=\sum_{\lambda_j\leq \alpha} |c_j|^2 e^{-\lambda_j t}
		\end{equation}
		for $F=\sum c_j F_j$. Thus in this case, \eqref{eq:sharpA1} holds for every $F$ if and only if $\alpha$ equals the first jumping number defined by
		\begin{equation}\label{eq:newversion}
			\sup \left\lbrace \lambda \in \mathbb R: H^0(X, \mathcal O(K_X+L) \otimes \mathcal I(
			\psi_\lambda)) = H^0(X, \mathcal O(K_X+L)\otimes\mathcal I(\phi) ) \right\rbrace.
		\end{equation}
		
		\begin{proposition}\label{pr:sharp} The First Main Theorem is sharp if the $||\cdot||_t$ norms in \eqref{eq:salpha-intro} define a flat metric on $H^0(X, \mathcal O(K_X+L)\otimes\mathcal I(\phi) )\times \mathbb H$  and  $\alpha$ is equal to the first jumping number.
		\end{proposition}
		
		The following flatness criterion for the $||\cdot||_t$ norm is a deep result proved by Berndtsson in \cite{Bern15} (see Proposition 3.3 in  \cite{Ber16} for a further development).
		
		\begin{theorem}\label{th:bern15} Let $(L, e^{-\phi})$ be a positive line bundle over a compact complex manifold $X$. Fix $v\in {\rm TC}_\phi$. Assume that elements in $H^0(X, K_X+L)$ have no common zeros. Then the followings are equivalent:
			\begin{itemize}
				\item [(1)] the $||\cdot||_t$ norms in \eqref{eq:salpha-intro} define a flat metric on $H^0(X, K_X+L)\times \mathbb H$;
				\item [(2)] the $||\cdot||_t$ norms in \eqref{eq:salpha-intro} define a flat metric on $\det H^0(X, K_X+L)\times \mathbb H$;
				\item [(3)] there is a holomorphic vector field $V$ on $X$ such that $\left(\frac\partial{\partial t}-V\right) \rfloor \,dd^c(\hat \psi_{{\rm Re}\tau})=0$.
			\end{itemize}
		\end{theorem}
		
		\begin{proof} (1) implies (2) is obvious. (2) implies (3) is included in the proof of \cite[Theorem 9.1]{Bern15}. To see that (3) implies (1), one may observe that (3) implies that  $\hat \psi_t(z)$ is smooth in $(z,t)$ and $\Theta:=i\partial\dbar(\hat \psi_{{\rm Re}\tau})$ is positive along the $X$ direction. Hence we know that the horizontal lift of $\frac\partial{\partial t}$ with respect to $\Theta$ equals our holomorphic vector field $\frac\partial{\partial t}-V$ and the geodesic curvature of the relative K\"ahler form $\Theta$ vanishes. Thus both  $c(\phi)$ and $\eta$ vanish in the curvature formula (1.4) in \cite{Bern11} for the bundle  $H^0(X, K_X+L)\times \mathbb H$. Hence (1) follows.  
		\end{proof}

	\end{sloppypar}	
	
\end{document}